\documentclass[a4paper, 10pt]{amsart}

\usepackage{amsmath, amscd, amssymb, mathrsfs, mathtools}
\usepackage{extarrows}
\usepackage[all]{xy}
\usepackage[colorlinks=false, linktocpage]{hyperref} 
\usepackage[utf8]{inputenc}
\usepackage[T1]{fontenc}
\usepackage{tikz-cd}
\usepackage[left=1.23in, right=1.23in, top=1in, bottom=1.3in,  headheight=13.6pt]{geometry} 
\usepackage{enumitem}
\numberwithin{equation}{section}

\title[Deligne Conjecture for Hecke Characters]{Deligne's Conjecture on the Critical Values of Hecke $L$-Functions}

\author{Han-Ung Kufner}
\address{Fakultät für Mathematik\\
  Universität Regensburg\\
  93040 Regensburg, Germany}
\email{han-ung.kufner@mathematik.uni-regensburg.de}
\date{\today}

\theoremstyle{plain}
    \newtheorem{theorem}{Theorem}[section]
	\newtheorem*{theorem*}{Theorem}
    \newtheorem{lemma}[theorem]{Lemma}
    \newtheorem{proposition}[theorem]{Proposition}
    \newtheorem*{proposition*}{Proposition}
    
    \newtheorem{definition}[theorem]{Definition}
    
    \newtheorem{notation}[theorem]{Notation}

\theoremstyle{definition}
    \newtheorem{remark}[theorem]{Remark}

\DeclareMathOperator{\End}{End}
\DeclareMathOperator{\Aut}{Aut}
\DeclareMathOperator{\Sym}{Sym}
\DeclareMathOperator{\TSym}{TSym}
\DeclareMathOperator{\Hom}{Hom}
\DeclareMathOperator{\Spec}{Spec}
\DeclareMathOperator{\GL}{GL}
\DeclareMathOperator{\tr}{tr}
\DeclareMathOperator{\Gal}{Gal}
\DeclareMathOperator{\im}{im}
\DeclareMathOperator{\Map}{Map}

\newcommand{\sH}{\mathscr{H}}
\newcommand{\sT}{\mathscr{T}}
\newcommand{\fra}{\mathfrak{a}}
\newcommand{\frb}{\mathfrak{b}}
\newcommand{\frc}{\mathfrak{c}}
\newcommand{\frf}{\mathfrak{f}}
\newcommand{\frp}{\mathfrak{p}}
\newcommand{\cO}{\mathcal{O}}
\newcommand{\sP}{\mathcal{P}}
\newcommand{\cR}{\mathcal{R}}
\newcommand{\CC}{\mathbf{C}}
\renewcommand{\AA}{\mathbf{A}}
\newcommand{\PP}{\mathbf{P}}
\newcommand{\QQ}{\mathbf{Q}}
\newcommand{\RR}{\mathbf{R}}
\newcommand{\ZZ}{\mathbf{Z}}
\newcommand{\TT}{\mathbf{T}}

\renewcommand{\det}{\mathrm{det}}
\newcommand{\can}{\mathrm{can}}
\newcommand{\op}{\mathrm{op}}
\newcommand{\Gm}{\mathbf{G}_m}
\newcommand{\dR}{\mathrm{dR}}
\newcommand{\id}{\mathrm{id}}
\newcommand{\into}{\hookrightarrow}

\newcommand{\Fr}{\mathrm{Fr}}
\newcommand{\Mod}{\mathrm{Mod}}
\newcommand{\Res}{\operatorname{Res}}
\newcommand{\Lie}{\operatorname{Lie}}

\newcommand{\xto}{\xrightarrow}
\renewcommand{\Re}{\operatorname{Re}}
\newcommand{\ab}{\mathrm{ab}}
\newcommand{\bs}{\backslash}
\newcommand{\tor}{\mathrm{tor}}
\newcommand{\et}{\mathrm{et}}
\newcommand{\cyc}{\mathrm{cyc}}
\newcommand{\ol}{\overline}

\newcommand{\wtilde}{\widetilde}
\newcommand{\cV}{\mathcal{V}}
\newcommand{\AH}{\mathrm{AH}}
\newcommand{\EK}{\operatorname{EK}}
\newcommand{\tEK}{\wtilde{\operatorname{EK}}}
\newcommand{\cEK}{\mathcal{EK}}
\newcommand{\mom}{\mathrm{mom}}
\newcommand{\Ab}{\mathrm{Ab}}
\newcommand{\Crit}{\mathrm{Crit}}
\newcommand{\Sch}{\mathrm{Sch}}
\newcommand{\CM}{\mathrm{CM}}
\newcommand{\M}{\mathrm{M}}
\newcommand{\balpha}{\boldsymbol{\alpha}}
\newcommand{\bbeta}{\boldsymbol{\beta}}

\begin{document}
\begin{abstract}
In this paper we give a proof of Deligne’s conjecture on the critical values of $L$-functions for arbitrary algebraic Hecke characters. This extends a result of Blasius, \cite{blasius}, which only works in the case of CM fields. The key new insight is that the Eisenstein-Kronecker classes of Kings-Sprang, \cite{kingssprang}, which allow for a cohomological interpretation of the value $L(\chi,0)$ for Hecke characters $\chi$ of arbitrary totally imaginary fields, can be regarded as de Rham classes of Blasius’ reflex motive.
\end{abstract}
\maketitle
\tableofcontents

\section*{Introduction}
In \cite{deligne}, Deligne formulated his far-reaching conjecture which expresses the critical values of a motivic $L$-function in terms of a period. In this paper, we prove Deligne's conjecture for arbitrary critical algebraic Hecke characters.
More precisely, let $\chi$ be an algebraic Hecke character of a number field $L$ and let $T$ denote the number field generated by the values of $\chi$. Within the full Tannakian subcategory of motives for absolute Hodge cycles that are generated by Artin motives and potentially CM abelian varieties, one can attach to $\chi$ a unique isomorphism class of a motive $M(\chi)$ over $L$ with coefficients in $T$. For $s \in \CC$ with sufficiently large real part, the $L$-function of $M(\chi)$ is given by the absolutely convergent Hecke $L$-series $L(\chi,s) = \sum_\fra \frac{\chi(\fra)}{N\fra^s}$ with values in $T \otimes \CC$, where $\fra$ runs through all integral ideals coprime to the conductor of $\chi$. By classical results due to Hecke, $L(\chi,s)$ admits a meromorphic continuation to $\CC$ and satisfies a functional equation. We say that $\chi$ is \emph{critical} if the $\Gamma$-factors on both sides of the functional equation of $L(\chi,s)$ have no pole at $s = 0$. For this, it is necessary that $L$ is either totally real or that it contains a CM-field. 

Let $RM(\chi)$ denote the restriction of scalars to $\QQ$ and let $RM(\chi)^+_B$ be the subspace in the Betti realization fixed by the involution induced by complex conjugation. Let $F^\bullet$ denote the Hodge filtration on the de Rham realization $M(\chi)_\dR$. Let $I_\infty$ denote the comparison isomorphism between the de Rham and the Betti realization of a motive. If $\chi$ is critical, the composite
\[
I_+: RM(\chi)_B^+ \otimes \CC \subset RM(\chi)_B \otimes \CC \xto{I_\infty^{-1}} RM(\chi)_\dR \otimes \CC \to RM(\chi)_\dR / F^0 
\]
is a $T \otimes \CC$-linear isomorphism and Deligne defines the period
\[
c^+(\chi) = \det(I_+) \in (T \otimes \CC)^\times,
\]
where the determinant is calculated with respect to bases of the $T$-vector spaces $RM(\chi)_B^+$ and $RM(\chi)_\dR / F^0$. Note that $c^+(\chi)$ is only well-defined up to a factor in $T^\times$. Our main result is a proof of Deligne's conjecture (\cite{deligne}, Conjecture 2.8) for the motive $M(\chi)$:

\begin{theorem*}[Theorem \ref{deligne_conjecture}]
Let $L$ and $T$ be number fields and let $\chi$ be a critical algebraic Hecke character of $L$ with values in $T$. Then $L(\chi,0)$ coincides with $c^+(\chi)$ up to a factor in $T$.
\end{theorem*}

In the case where $L$ is totally real, the above result follows from the work of Euler, Siegel, \cite{siegel}, and Klingen, \cite{klingen}. When $L$ is an imaginary quadratic field, Damerell, \cite{damerell}, showed that the critical $L$-values agree with certain periods of CM elliptic curves up to an algebraic number. Deligne's conjecture was settled in the imaginary quadratic case by Goldstein-Schappacher in \cite{goldsteinschappacher} and \cite{goldsteinschappacher2}. 
If $L$ is a CM-field, algebraicity results were obtained by Shimura, \cite{shimura2}, and Katz, \cite{katz} and it was shown by Deligne, \cite{deligne}, that the results of Shimura are compatible with his conjecture. In \cite{blasius}, Blasius was able to prove Deligne's conjecture for algebraic Hecke characters of CM-fields. 
If $L$ contains an imaginary quadratic number field, algebraicity results have previously been obtained by 
Colmez, \cite{colmez}, and more recently by Bergeron-Charollois-Garcia, \cite{BCG}.
For arbitrary number fields $L$ containing a CM-field, Kings-Sprang, \cite{kingssprang}, were able to relate the $L$-value $L(\chi,0)$ to periods of abelian varieties with complex multiplication by $L$ up to an algebraic integer. This allowed them to establish Deligne's conjecture up to a factor in $T \otimes \overline{\QQ}$. Central to their approach is a novel construction of equivariant coherent cohomology classes on abelian varieties with complex multiplication by $L$. For arbitrary $L$ containing a CM-field $K$, Deligne's conjecture was announced by Harder-Schappacher, \cite{harderschappacher}, but details were never published. Roughly speaking, the strategy outlined in loc.\ cit.\ is to use results on the Eisenstein cohomology for $\GL_n$, where $n = [L:K]$, to show that the ratio $L(\chi,0)/L(\chi|_{\AA_K^\times},0)$ behaves exactly like the corresponding ratio between the $c^+$-periods. This would reduce the general case to the CM-case established by Blasius. For $n = 2$, the necessary theory is developed in \cite{harder}.

\subsection*{Outline of the argument}

\subsubsection*{Short summary}
Let us first give a quick summary of the proof presented in this paper, before we give a more detailed outline. Let $L$ be a totally imaginary number containing a CM-field and let $\chi$ be a critical algebraic Hecke character of $L$. The classes constructed in \cite{kingssprang}, which allow to access $L(\chi,0)$, live on abelian varieties with complex multiplication by the field $L$. In contrast, the motive $M(\chi)$ has $L$ as its field of definition: A prototypical example for such a motive would be the $h^1$ of a CM abelian variety defined over $L$. Thus, the question arises, how to relate the motive $M(\chi)$ and its period $c^+(\chi)$ to such abelian varieties with complex multiplication by $L$. 
It is a seminal insight of Blasius how to establish this connection: From $M(\chi)$ he constructs a motive $M(\Xi)$, referred to as the reflex motive in the following, and recovers the period $c^+(\chi)$ (up to a factor in $T^\times$) in terms of another period of $M(\Xi)$, which we call the reflex period for a moment.
The reflex motive $M(\Xi)$ is defined over the reflex field of a certain CM-type of $L$ and the key point is, as our notation already suggests, that it is the motive for a Hecke character $\Xi$, which essentially is obtained from $\chi$ by precomposing with the reflex norm. This makes it possible to give an alternative construction of $M(\Xi)$ in terms of an abelian variety with complex multiplication by $L$. The construction shows that the Eisenstein-Kronecker classes of Kings-Sprang naturally live in the de Rham realization $M(\Xi)_\dR$ and they are admissible for computing the reflex period. By giving an explicit presentation, in terms of the alternative construction of $M(\Xi)$, of those Betti classes which are needed to calculate the reflex period and slightly extending results of Kings-Sprang, we show that the reflex period coincides with the $L$-value $L(\chi,0)$. Deligne's conjecture then follows from Blasius' result.

Let us note that also in the case where $L$ is a CM-field, our approach simplifies the proof in \cite{blasius}. There Blasius proves a subtle reciprocity law, which combined with his period relation establishes a connection between special values of Hilbert modular forms and the period $c^+(\chi)$. In our approach, this reciprocity law can be avoided entirely by directly constructing classes in the reflex motive.

We now give a more detailed outline of the argument. Let us fix a bit of notation: For any number field $F$, we let $J_F$ denote the set of embeddings $F \into \CC$ and let $I_F$ denote the free abelian group generated by $J_F$. For any element $\mu = \sum_\sigma \mu(\sigma) \sigma$ in $I_F$ we let $d(\mu) = \sum_{\sigma} \mu(\sigma)$ denote its degree. If $x \in F$, we write $x^\mu \coloneqq \prod_{\sigma} \sigma(x)^{\mu(\sigma)}$. Moreover, we let $F^\mu$ denote the number field generated by the elements $x^\mu$ for $x \in F$. Let $F_f$ be the ring of finite adeles of $F$.

\subsubsection*{Blasius' period relation}
Assume that $L$ is a number field containing a CM-field and that $\chi$ is an algebraic Hecke character of $L$ with field of values $T$. 
The criticality of $\chi$ implies that there exists a CM-type $\Phi$ of $L$, induced from a CM-subfield of $L$, such that the infinity-type $\chi_a \in I_L$ of $\chi$ can be expressed in the form
$\chi_a = \beta - \alpha$
where $\alpha$ is supported on $\Phi$ with $\alpha(\sigma) \geq 1$ for all $\sigma \in \Phi$ and $\beta$ is supported in $\ol{\Phi}$ with $\beta(\ol{\sigma}) \geq 0$ for all $\ol{\sigma} \in \ol{\Phi}$.
Let $E$ denote the reflex field of $(L,\Phi)$ and let $\Phi^\ast$ denote the reflex CM-type of $\Phi$ (cf.\ Section \ref{section_cm_types}). Then, by regarding $\Phi^\ast$ as a continuous homomorphism $\Phi^\ast: \AA_E^\times \to \AA_L^\times$, we define the Hecke character (cf.\ (\ref{def_character_xi}))
\[
\Xi = (\chi \circ \Phi^\ast)^{-1}  \cdot \lVert \cdot \rVert_E^{-d(\beta)}  \cdot \varepsilon_\Phi : \AA_E^\times \to T^\times,
\]
where $\lVert \cdot \rVert_E$ is the idele norm and $\varepsilon_\Phi$ is a certain sign character attached to $\Phi$, cf.\ Definition \ref{def_sign_epsilon}.
Attached to $\Xi$ there is a motive $M(\Xi)$ over $E$ with coefficients in $T$.
The following theorem of Blasius, which may be seen as the starting point of our proof, expresses the period $c^+(\chi)$ in terms of the motive $RM(\Xi)$:

\begin{theorem*}[Blasius] The motive $M(\Xi)$ satisfies the following properties:
\begin{enumerate}
\item For every embedding $\eta$ of $E$, there exists a $T$-basis $a_\eta$ of the Betti realization $M(\Xi)_B^\eta$ such that
\begin{equation*}
I_f(a_\eta)^\tau = \xi(\tau,\eta) \cdot I_f(a_{\tau \eta})
\end{equation*}
for all $\tau \in \Gal(\ol{\QQ}/\QQ)$. Here, the $\xi(\tau,\eta)$ denote explicitly defined elements in $T_f$, cf.\ (\ref{def_xi_elements}), and $I_f$ denotes the comparison isomorphism between the Betti and the finite adelic realization.
\item The $T$-vector space $F^{d(\alpha)+d(\beta)} RM(\Xi)_\dR$ is $1$-dimensional.
\item If $\omega \in F^{d(\alpha)+d(\beta)} RM(\Xi)_\dR$ is a basis, one has
\[
I_\infty(\omega) = (2\pi i)^{d(\beta)} \cdot c^+(\chi) \cdot \sum_{\eta \in J_E} e(\eta) a_\eta \mod T^\times,
\]
where $e(\eta) \in T \otimes \CC$ denotes the idempotent that cuts out the direct factor $T \otimes_{E,\eta} \CC$.
\end{enumerate}
\end{theorem*}

\subsubsection*{Relation to CM abelian varieties}
Let us now describe a more geometric construction of the reflex motive $M(\Xi)$.
Let $A/F$ be an abelian variety over a number field $F$ of CM-type $(L,\Phi)$ and let $\psi$ denotes its associated Hecke character. Then $F$ contains $E$ and by comparing infinity-types, one sees that
\begin{equation}\tag{1}\label{hecke_identitiy}
\Xi \circ N_{F/E} = \smash{\psi}^\alpha \cdot \smash{\ol{\psi}}^\beta
\end{equation}
for sufficiently large $F$. Here $\smash{\psi}^\alpha$ (and similarly $\smash{\ol{\psi}}^\beta$) is simply given by composing $\psi$ with the homomorphism $\alpha: L^\times \to (L^\alpha)^\times$.
The identity of Hecke characters (\ref{hecke_identitiy}) should admit a counterpart on the level of motives. By definition, $h^1(A)$ is a motive for the Hecke character $\psi$. 
One can construct from $h^1(A)$ a motive $h^1(A)^\alpha$ for the Hecke character $\psi^\alpha$ as a suitable direct factor in 
\[
\Sym^{d(\alpha)} R_{L/\QQ} h^1(A),
\] 
where $R_{L/\QQ}$ denotes restriction of coefficients. This construction, which has been outlined by Blasius in \cite{blasius}, 5.5, will be carefully investigated in Section \ref{section_5_M_alpha}.
Since $h^1(A^\vee)$ for the dual abelian variety $A^\vee$ is a motive for $\ol{\psi}$, we obtain from (\ref{hecke_identitiy}) an isomorphism of motives
\begin{equation}\tag{2}\label{intro_motives_iso}
M(\Xi) \times_E F \cong h^1(A)^\alpha \otimes h^1(A^\vee)^\beta.
\end{equation}
Assume now that $F/E$ is a Galois extension. Applying restriction of scalars, we then obtain an isomorphism $R_{F/E}(M(\Xi) \times_E F) \cong R_{F/E}(h^1(A)^\alpha \otimes h^1(A^\vee)^\beta)$ and by transport of structure the Galois group $\Gal(F/E)$ acts on $R_{F/E}(h^1(A)^\alpha \otimes h^1(A^\vee)^\beta)$ such that $M(\Xi)$ can be recovered by passing to the coinvariants. We obtain:

\begin{proposition*}[Section \ref{subsection_alt_construction}]
The motive $M(\Xi)$ is isomorphic to a quotient (in fact, a direct factor) of the restriction of scalars
\[
R_{F/E}(h^1(A)^\alpha \otimes h^1(A^\vee)^\beta).
\]
\end{proposition*}

\subsubsection*{Special values of $L$-functions}
The next step is to investigate Blasius' theorem in the context of the newly obtained presentation of the motive $M(\Xi)$. 
The strategy is now as follows:
\begin{itemize}
\item Find a presentation for the special basis elements $\{a_\eta \mid \eta \in J_E\}$ in terms of Betti classes of the motive $h^1(A)$.
\item Construct an element in $F^{d(\alpha)+d(\beta)} RM(\Xi)_\dR$, which admits a relation to the $L$-values $L(\chi,0)$. This is provided by taking a suitable linear combination of the Eisenstein-Kronecker classes constructed by Kings-Sprang.
\end{itemize}
Surprisingly, in the Betti realization the elements $a_\eta$ can be made very explicit. 
The construction of the motive $h^1(A)^\alpha$ allows to canonically attach an element $\gamma^{\balpha} \in h^1(A)^\alpha_B$ to every $\gamma \in h^1_B(A)$ in the Betti realization. 
For every embedding $\sigma$ of $F$, we choose an $L$-basis $\gamma_\sigma$ of $h^1_B(A^\sigma)$ and let $\gamma_\sigma^\vee \in h^1_B(A^{\sigma,\vee})$ denote dual basis with respect to the canonical pairing $h^1_B(A^\sigma) \otimes h^1_B(A^{\sigma,\vee}) \to (2 \pi i)^{-1} L$. Then, for every $\eta \in J_E$, we define $a_\eta$ as the image of 
\[
\sum_{\sigma|_E = \eta} t_\sigma \cdot \gamma_\sigma^{\balpha} \otimes (\gamma_\sigma^\vee)^{\bbeta} \in R_{F/E}(h^1(A)^\alpha \otimes h^1(A^\vee)^\beta)^\eta_B 
\]
in $M(\Xi)_B^\eta$. Here, the elements $t_\sigma \in T^\times$ are explicit factors depending on the choice of $\{\gamma_\sigma \mid \sigma \in J_F\}$, cf.\ Definition \ref{t_elements_independence}. 
The following theorem shows that the elements $a_\eta$ constructed above are in fact suitable for applying Blasius' theorem. The difficult part consists in showing that the elements $a_\eta$ do not vanish when passing to the quotient $M(\Xi)_B^\eta$.

\begin{theorem*}[Theorem \ref{a_eta_elements_transformation}]
For every embedding $\eta \in J_E$, the element $a_\eta$ is a basis element of $M(\Xi)_B^\eta$. For any $\tau \in \Gal(\ol{\QQ}/\QQ)$, one has
$I_f(a_\eta)^\tau = \xi(\tau,\eta) \cdot I_f(a_{\tau\eta})$.
\end{theorem*}

As the other main advantage of our approach, one sees that the values of the Eisenstein-Kronecker classes of Kings-Sprang, \cite{kingssprang}, naturally live in $F^{d(\alpha)+d(\beta)} RM(\Xi)_\dR$. We define a normalized class $\cEK$ in the de Rham realization of $h^1(A)^\alpha \otimes h^1(A^\vee)^\beta$ as certain linear combination of these classes and by adapting and slightly expanding the results of \cite{kingssprang}, we obtain:

\begin{theorem*}[Theorem \ref{the_main_theorem}]
For every $\sigma \in J_F$, one has
\[
I_\infty(\cEK^\sigma) = (2\pi i)^{d(\beta)} \cdot e(\sigma|_E) \cdot L(\chi,0) \cdot t_\sigma \cdot \gamma_\sigma^{\balpha} \otimes (\gamma_\sigma^\vee)^{\bbeta}
\]
in $(h^1(A)^\alpha \otimes h^1(A)^\beta)^\sigma_B \otimes \CC$.
\end{theorem*}

The de Rham realization of $RM(\Xi)$ simply is $M(\Xi)_\dR$ regarded as a $T$-vector space and in view of our construction, the latter is a direct factor in $(h^1(A)^\alpha \otimes h^1(A)^\beta)_\dR$. This allows to regard $\cEK$ as a class in $RM(\Xi)_\dR$ and the above theorem implies that
\[
I_\infty(\cEK) = \sum_{\sigma \in J_F} I_\infty(\cEK^\sigma) = (2 \pi i)^{d(\beta)} \cdot L(\chi,0) \cdot \sum_{\eta \in J_E} e(\eta) a_\eta.
\]
Now Blasius' theorem implies that $L(\chi,0)$ and $c^+(\chi)$ only differ by an element in $T$.

\subsection*{Outline of the paper}
In the first preliminary three sections, we recollect some facts from the theory of complex multiplication. In Section \ref{section_1_alg_hecke}, we review CM-types, algebraic Hecke characters and their associated $L$-functions. In Section \ref{section_2_CM_ab_var} we discuss CM abelian varieties and recall the main theorem of complex multiplication. Finally, in Section \ref{section_3_cm_motives}, the category of CM-motives and in particular its connection to algebraic Hecke characters is reviewed.

In Section \ref{section_5_M_alpha}, starting from a motive $M$ for a Hecke character $\psi$ with values in a number field $L$ and $\alpha \in I_L^+$, we give a detailed construction of the motive $M^\alpha$ for the Hecke character $\psi^\alpha$. 
Moreover, we discuss how to associate to Betti classes $\gamma$ (resp.\ de Rham classes $\omega$) of $M$ elements $\gamma^{\balpha}$ (resp.\ $\omega^{\balpha}$) in the Betti (resp.\ de Rham) realization of $M^\alpha$ and we express the resulting periods of $M^\alpha$ in terms of periods of $M$.
We formulate various decomposition results in terms of representations of a certain algebraic torus. While Galois descent is completely sufficient for these purposes, we choose this formalism for a more clean presentation.

In Section \ref{section_6_reflex_motive}, we start with a review of the reflex motive and Blasius' period relation. Various important objects, like the Hecke character $\Xi$, will be introduced. Via the results of \ref{section_5_M_alpha}, we then arrive at our alternative presentation of the reflex motive in terms of abelian varieties. Finally, we discuss the Betti realization in detail and provide the explicit presentation of the special basis elements $a_\eta$ in this new context.

In Section \ref{section_7_ek_classes}, we recall the definition of the Eisenstein-Kronecker classes of \cite{kingssprang}. The novel feature of this section is the systematic adaptation of the results of loc.\ cit.\ to also cover arbitrary Galois conjugates of Eisenstein-Kronecker classes. Along the way, we observe that the process of conjugation by $\sigma \in \Gal(\ol{\QQ}/\QQ)$ introduces the sign $\varepsilon_\Phi(\sigma)$. Moreover, the relation between the Eisenstein-Kronecker classes with special values of Hecke $L$-functions is extended to such conjugates.

Finally, in Section \ref{section_8_deligne_conj}, we show that the Eisenstein-Kronecker classes live naturally in the de Rham realization of the reflex motive. To any critical algebraic Hecke character we associate a normalized class $\cEK$ in the de Rham realization of the reflex motive and we reinterpret the value $L(\chi,0)$ as a period. Blasius' period relation and the main result of Section \ref{section_6_reflex_motive} are used to deduce Deligne's conjecture. 

\subsection*{Notations}
We assume throughout this text that number fields are embedded in $\CC$. We let $\ol{\QQ}$ denote the algebraic closure of $\QQ$ in $\CC$. In particular, complex conjugation on $\CC$ restricts to an automorphism of $\ol{\QQ}$ which we denote by the letter $c$ or $x \mapsto \ol{x}$ for $x \in \ol{\QQ}$.
For a number field $L$ we let $J_L \coloneqq \Hom(L,\CC) = \Hom(L, \ol{\QQ})$ denote the set of embeddings into $\CC$ (equivalently, into $\ol{\QQ}$) and let $I_L$ denote the free abelian group generated by $J_L$. 
We denote the fixed embedding of $L$ into $\CC$ by $1_L$. 
If $K \subset L$ is a subfield and $\sigma \in J_K$, we denote by $J_{L,\sigma}$ the set of those embeddings in $J_L$ which restrict to $\sigma$ on $K$ and we write $J_{L/K}$ instead of $J_{L, 1_K}$. 
For any element $\mu = \sum_{\sigma} \mu(\sigma) \sigma \in I_L$, we let $d(\mu) \coloneqq \sum_{\sigma} \mu(\sigma) \in \ZZ$ denote the \emph{degree of $\mu$}. For any subset $\Delta \subset J_L$, we write $I_\Delta \coloneqq \{\mu \in I_L \mid \mu(\sigma) = 0 \textrm{ for } \sigma \notin \Delta\}$ and $I_\Delta^+ \coloneqq \{\mu \in I_\Delta \mid \mu(\sigma) \geq 0 \textrm{ for all } \sigma \in \Delta\}$. For any integer $d \geq 0$, we also write $I^{+,d}_\Delta = \{\mu \in I^+_\Delta \mid d(\mu) = d\}$.

For a Galois extension $L/K$, we abbreviate its Galois group by $G_{L/K}$.
The absolute Galois group of a field $L$ will be denoted by $G_L$. 
For any number field $L$, we write $\AA_L$ for the ring of adeles and $\AA_L^\times$ for the group of ideles. We denote the ring of finite adeles by $L_f$ and the group of finite ideles by $L_f^\times$. Let $L_\infty = L \otimes \RR$ denote the archimedean part of $\AA_L$. The idele norm of $L$ will be denoted by $\lVert\cdot\rVert_L$. It is defined by $\lVert(x_v)\rVert_L = \prod_v |x_v|_v$, if $|\cdot|_v$ denotes the normalized absolute value on the completion $L_v$ for any place $v$ of $L$.
If $x \in L_f^\times$ is a finite idele, we denote its associated fractional ideal in $L$ by $[x]_L$.
For an integral ideal $\frf \subset \cO_L$, we let $I_L^\frf$ denote the group of fractional ideals of $L$ which are coprime to $\frf$. We let $P^\frf$ denote the subgroup of principal fractional ideals $(\alpha)$ with $\alpha \equiv 1 \mod \frf$ and $\alpha$ totally positive. 
If $L$ is a number field, we use the \emph{geometric} normalization of the Artin homomorphism and denote it by 
$r_L: \AA_L^\times \to G_L^\ab$.
If $L$ is totally complex, $r_L$ factors through $L_f^\times$.
We let
$\chi_\cyc: G_\QQ \to \widehat{\ZZ}^\times$
denote the cyclotomic character. Throughout this paper, we fix a square root $i$ of $-1$.

Let $R$ be a commutative ring and $M$ an $R$-module. For an integer $n \geq 0$, we let $\Sym^n M$ denote the $n$th symmetric powers, i.e.\ the coinvariants of $M^{\otimes n}$ under the permutation action of the symmetric group $S_n$, and we let $\TSym^n(M)$ denote the $R$-module of symmetric tensors in $M^{\otimes n}$, i.e.\ the $S_n$-invariants of $M^{\otimes n}$. 
For $m \in M$ we set 
\[
m^{[n]} \coloneqq m \otimes \ldots \otimes m \in \TSym^n(M).
\]
We have the canonical map 
\begin{equation}\label{TSymSym}
\psi: \TSym^n(M) \subset M^{\otimes n} \twoheadrightarrow \Sym^n(M)
\end{equation}
and a canonical morphism of graded algebras
$\Sym^\bullet(M) \to \TSym^\bullet(M)$,
given by the symmetrization map in each degree.
One checks that $\varphi \circ \psi$ and $\psi \circ \varphi$ are given by multiplication by $n!$. Hence, if $R$ is a $\QQ$-algebra, the map (\ref{TSymSym}) is an isomorphism that sends $m^{[n]}$ to $m^n$. The drawback of (\ref{TSymSym}) is that it does not define a morphism of graded algebras. Still, let us fix the convention that we identify $\TSym$ and $\Sym$ via (\ref{TSymSym}). The reason is that we can avoid some denominators in Definition \ref{def_modified_ek_class}.

\subsection*{Acknowledgements}
This paper is a simplified and more concise version of my PhD thesis at the University of Regensburg \cite{thesis}. I would like to thank my advisor Guido Kings for his constant support throughout the last years. I am also grateful to Johannes Sprang for helpful discussions about Eisenstein-Kronecker classes. It is a pleasure to thank Don Blasius for many valuable comments. Moreover, it should become evident to the reader how much this paper owes to the ideas developed in \cite{blasius}. This work was supported by the SFB 1085 "Higher Invariants" funded by the Deutsche Forschungsgesellschaft (DFG).

\section{Algebraic Hecke characters}\label{section_1_alg_hecke}
In this section, we recall basic definitions and facts about CM-types, algebraic Hecke characters and Hecke $L$-functions.

\subsection{CM-types}\label{section_cm_types}
In the following, let $L$ and $T$ be number fields. Recall that $J_L \coloneqq \Hom(L,\ol{\QQ})$ denotes the set of embeddings into the fixed algebraic closure $\ol{\QQ}$ and that $I_L$ denotes the free abelian group generated by $J_L$. For an element $\mu \in I_L$, we write
$\mu = \sum_{\sigma \in J_L} \mu(\sigma) \sigma$
with $\mu(\sigma) \in \ZZ$ and, for any $\ell \in L$, we set
\[
\ell^\mu \coloneqq \prod_{\sigma \in J_L} \sigma(\ell)^{\mu(\sigma)} \in \ol{\QQ}.
\]
For any extension of number fields $L \supset K$ and $\mu \in I_K$, we define $N_{L/K}^\ast(\mu) \in I_L$ by 
$N_{L/K}^\ast(\mu)(\sigma) = \mu(\sigma|_K)$
for all $\sigma \in J_L$.
There is a linear and continuous left action of $G_\QQ$ on $I_L$ by the rule $(\tau \mu)(\sigma) = \mu(\tau^{-1} \sigma)$ for $\tau \in G_\QQ$ and $\sigma \in J_L$. We let $L^\mu$ denote the fixed field of the stabilizer of $\mu$. As the notation suggests, $L^\mu$ is the number field generated by the elements $\ell^\mu$ for $\ell \in L$.

Recall that a CM-field is a number field $K$ which is a imaginary quadratic extension of a totally real field. Equivalently, $K$ admits a non-trivial automorphism $c_K$ such that $\sigma \circ c_K = c \circ \sigma$ for all embeddings $\sigma \in J_K$. The composite of CM-fields (and in particular the Galois closure) of CM-fields is again a CM-field. A \emph{CM-type} of a CM-field $K$ is a subset $\Phi \subset J_K$ such that 
\[
\Phi \cup \ol{\Phi} = J_K \;\;\; \textrm{ and } \;\;\; \Phi \cap \ol{\Phi} = \emptyset.
\]
Usually, by considering its characteristic function, we also regard $\Phi$ as an element in $I_K$.
More generally, let $L$ be a number field containing a CM-field and let $K$ denote its maximal CM subfield. We say that a subset $\Phi \subset J_L$ is a CM-type if there exists a CM-type $\Phi_K$ of $K$ such that 
$\Phi = N_{L/K}^\ast(\Phi_K)$.
The \emph{reflex field of $(L,\Phi)$} is defined as 
$E \coloneqq L^\Phi$.
Let us recall the definition of the reflex CM-type. We choose a Galois extension $L'$ of $\QQ$ containing $L$ and $E$ and consider the induced CM-type $\Phi' = N_{L'/L}^\ast(\Phi)$. We then define $(\Phi')^{-1} \in I_{L'}$ by the rule $(\Phi')^{-1}(\sigma) = \Phi'(\sigma^{-1})$ for all $\sigma \in J_{L'}$, where we identify $J_{L'}$ with $\Gal(L'/\QQ)$ via our fixed embedding. One can show that there exists a unique CM-type $\Phi^\ast$ of $E$ satisfying
\[
N_{L'/E}^\ast (\Phi^\ast) = (\Phi')^{-1}
\]
and the pair $(E,\Phi^\ast)$ is minimal with respect to this property. The CM-type $\Phi^\ast$ is independent of the choice of $L'$ and is called the \emph{reflex CM-type of $(L,\Phi)$}.

\subsection{Algebraic Hecke characters}
As before, let $L$ and $T$ be number fields. 
A group homomorphism $\theta: L^\times \to T^\times$ is called \emph{algebraic} if it is induced from a homomorphism of algebraic tori $\Res_{L/\QQ} \Gm \to \Res_{T/\QQ}\Gm$ on $\QQ$-valued points. Equivalently, there exists an element $\mu \in I_L$, which is invariant under the action of $G_T$, such that 
$\theta(\ell) = \ell^\mu$
for all $\ell \in L^\times$.
In the following, we equip $T$ with the discrete topology.
An \emph{algebraic Hecke character of $L$ with values in $T$} is a continuous group homomorphism $\chi: \AA_L^\times \to T^\times$, whose restriction $\chi|_{L^\times}: L^\times \to T^\times$ is an algebraic homomorphism. The corresponding element in $I_L$ is denoted by $\chi_a$ and called the \emph{infinity-type of $\chi$.} 

Note that if $L$ is totally imaginary, the archimedean part $L_\infty^\times \coloneqq (L \otimes \RR)^\times$ is connected and hence $\chi$ factors as a continuous character $\chi: L_f^\times \to T^\times$ of the finite ideles. In this case the idele norm character $\lVert \cdot \rVert_L : L_f^\times \to \QQ^\times$ is an algebraic Hecke character of infinity-type $- \sum_{\sigma \in J_L} \sigma$. Also note that by class field theory the Hecke characters with trivial infinity-type correspond precisely to the finite Galois characters $G_L^\ab \to T^\times$.

\begin{remark}
Hecke character as defined above appear in this form in \cite{serretate} and are sometimes referred to as Serre-Tate characters in the literature. Often Hecke characters are defined as certain quasicharacters of the idele class group $\AA_L^\times/L^\times$ and for algebraic Hecke characters the two definitions are equivalent: We can regard an algebraic Hecke character $\chi: \AA_L^\times \to T^\times$ as having values in $\CC^\times$ via our fixed embedding $T \subset \CC$ and the infinity-type $\chi_a$ induces a continuous homomorphism $\chi_\infty: L_\infty^\times \to T_\infty^\times \to \CC^\times$, where the latter map is again induced by our embedding. Then $\chi \chi_\infty^{-1}$ descends to a quasi-character of the idele class group.
\end{remark}

Let us briefly discuss the relationship with the classical definition of Hecke characters in terms of ideals. Let $\chi: \AA_L^\times \to T^\times$ be an algebraic Hecke character. By continuity, there exists an integral ideal $\frf \subset \cO_L$ such that the kernel of $\chi$ contains the open subgroup 
\[
U_\frf = (L_\infty^\times)^0 \times \prod_{v \mid \frf} (1 + \frf \cO_{L_v}) \times \prod_{v \nmid \frf \infty} \cO_{L_v}^\times.
\]
In this case, we say that \emph{$\chi$ is of conductor dividing $\frf$}. Let $S = {v \mid \frf \infty}$ and $\AA_L^{S,\times}$ be the ideles away from $S$. Recall that $I^\frf$ denote the group of fractional ideals of $L$ that are coprime to $\frf$.
Via the surjection $[\cdot]_L : \AA_L^{S,\times} \to I^\frf$, which maps an idele to its associated fractional ideal, the restriction of $\chi$ to $\AA_L^{S,\times}$ factors as a homomorphism 
\[
\chi': I^\frf \to T^\times
\] 
satisfying $\chi'((\alpha)) = \chi_a(\alpha)$ for all $(\alpha) \in P^\frf$, i.e.\ $\chi'$ is an algebraic Hecke character in the classical sense. Conversely, by weak approximation, any such homomorphism defines a unique algebraic Hecke character of $L$ with values in $T$. Note that the idele norm character $\lVert \cdot \rVert_L$ corresponds to the inverse of the norm character $N_{L/\QQ}^{-1} : I^\frf \to \QQ^\times$.

It follows from the definition that the infinity-type of an algebraic Hecke character becomes trivial on a finite index subgroup $\Gamma$ of $\cO_L^\times$. We refer to elements in $I_L$ with this property as being of \emph{Hecke character type with respect to $\Gamma$}.
It was shown by Weil that $\mu \in I_L$ is of Hecke character type with respect to some $\Gamma$ if $\mu$ is induced from either a totally real or CM subfield of $L$ and if there exists an integer $w \in \ZZ$ such that $\mu(\sigma) + \mu(\ol{\sigma}) = w$ for all $\sigma \in J_L$. 
As an important consequence any algebraic Hecke character $\chi$ of $L$ satisfies
\[
\chi \cdot \ol{\chi} = \lVert \cdot \rVert_L^{-w}
\]
for some integer $w \in \ZZ$, the \emph{weight of $\chi$}.

\begin{remark}
In fact, one obtains the following classification of algebraic Hecke characters (see e.g.\ \cite{schappacher}, §0, section 3):
If $L$ does not contain a CM-field, then $\chi$ is of the form
$\chi = \varrho \cdot  \lVert \cdot \rVert_L^{-w/2}$,
where $\varrho$ is a finite Hecke character, $\lVert \cdot \rVert$ is the idele norm character and $w \in 2\ZZ$ is an even integer.
On the other hand, if $L$ contains a CM-field and $K \subset L$ is the maximal CM-subfield, then $\chi$ is of the form
$\chi = \varrho \cdot (\varphi \circ N_{L/K})$,
where $\varrho$ is a finite Hecke character and $\varphi$ is an algebraic Hecke character of $K$.
\end{remark}

In this paper, we will be mainly concerned with Hecke characters of the latter type where the field of definition $L$ contains a CM-field.

\begin{definition}\label{criticality_def}
Let $L$ be a number field which contains a CM-field. An element $\mu \in I_L$ is called \emph{critical} if $\mu$ is of Hecke character type for some finite index subgroup $\Gamma \subset \cO_L^\times$ and if there exist a CM-type $\Phi$ of $L$ such that $\mu$ can be decomposed in the form
\[
\mu = \beta - \alpha
\] 
with $\alpha \in I_{\Phi}$ and 
$\beta \in I_{\ol{\Phi}}$ satisfying
\[
\alpha(\sigma) \geq 1 \;\;\; \textrm{ and } \;\;\; \beta(\ol{\sigma}) \geq 0 \;\;\; \textrm{ for all } \sigma \in \Phi.
\]
We say that an algebraic Hecke character is critical if its infinity-type is critical.
\end{definition}

By the above classification, $\alpha$ and $\beta$ in the definition above are induced from the maximal CM subfield of $L$.
Also note that the CM-type $\Phi$ is uniquely determined. To emphasize its role, we sometimes say that $\mu$ (or $\chi$) is critical with respect to $\Phi$. For a finite index subgroup $\Gamma \subset \cO_L^\times$, we let $\Crit_L(\Gamma)$ denote the set of all $\mu \in I_L$ which are critical and of Hecke character type with respect to $\Gamma$.

\subsection{Hecke $L$-functions}
We now discuss Hecke $L$-functions and their partial versions. Let $\frf \subset \cO_L$ be an integral ideal such that we can regard $\chi$ as a homomorphism $\chi: I^\frf \to T^\times$.
For any embedding $\iota \in J_T$, we put $\chi_\iota \coloneqq \iota \circ \chi$ and define the Hecke $L$-function
\begin{align}\label{hecke_l_function_def}
L_\frf(\chi_\iota, s) \coloneqq \sum_{\substack{\fra \in I^\frf \\ \fra \subset \cO_L}} \frac{\chi_\iota(\fra)}{N\fra^s} 
= \prod_{\frp \nmid \frf} (1 - \chi_\iota(\frp) N\frp^{-s})^{-1}
\end{align}
for all $s \in \CC$ with $\Re(s) > \frac{w}{2}+1$, where $w$ denotes the weight of $\chi$. The $L$-series $L_\frf(\chi,s)$ admits a meromorphic continuation to $\CC$ and satisfies a functional equation.
For every fractional ideal $\frb \in I^\frf$ we denote its corresponding ray class modulo $\frf$ by $[\frb]$. We define the partial $L$-functions	
\[
L_\frf(\chi_\iota, s, [\frb]) \coloneqq \sum_{\substack{\fra \in [\frb] \\ \fra \subset \cO_L}} \frac{\chi_\iota(\fra)}{N\fra^s}
\]
for $\Re(s) > \frac{w}{2}+1$, which also admit meromorphic continuations to $\CC$.
Finally, we use the identification $T \otimes \CC \cong \CC^{J_T}$, $t \otimes z \mapsto (\iota(t) z)_{\iota \in J_T}$ in order to define a $T \otimes \CC$-valued $L$-function
\[
L_\frf(\chi,s) \coloneqq (L_\frf(\chi_\iota,s))_{\iota \in J_T} 
\]
and similarly we define its partial analogue $L_\frf(\chi,s,[\frb])$
for all $\frb \in I^\frf$.

\section{Abelian varieties with complex multiplication}\label{section_2_CM_ab_var}
To fix definitions, we recollect some basics about abelian varieties with complex multiplication. We also recall Serre's tensor construction and discuss the main theorem of complex multiplication.

\subsection{CM abelian varieties}
In this section, let $L$ denote a totally imaginary number field of degree $[L:\QQ] = 2g$ which contains a CM-field. Let $F$ be a number field.
An \emph{abelian variety with complex multiplication by $L$} is an abelian variety $A/F$ of dimension $g$ equipped with an embedding 
\[
i: L \into \End^0_F(A) \coloneqq \End_F(A) \otimes_\ZZ \QQ.
\]
If $(A/F,i)$ is an abelian variety with complex multiplication by $L$, the ring
\[
L \cap \End_F(A) = i^{-1}(\End_F(A))
\]
is an order in $L$. If it is the maximal order $\cO_L$, we say that \emph{$(A,i)$ has complex multiplication by $\cO_L$}. We usually drop the embedding $i$ from the notation.
Via the CM-structure, the singular cohomology $H^1(A(\CC),\CC)$ is a free $L \otimes \CC$-module of rank $1$. It follows from the Hodge decomposition theorem that there exists a CM-type $\Phi$ of $L$ such that $L$ acts through the $g$ embeddings in $\Phi$ on the Lie algebra $\Lie(A/F)$. We say that \emph{$A/F$ is of CM-type $(L,\Phi)$}. 
We equip the dual abelian variety $A^\vee$ of $A$ with with the induced CM-structure 
\[
L \to \End_F(A) \xto{(-)^\vee} \End_F(A^\vee).
\]
With respect to this CM-structure the dual abelian variety $A^\vee$ is of CM-type $(L,\ol{\Phi})$.

Let $\pi: A \to \Spec(F)$ be an abelian variety with complex multiplication by $L$ and let $e$ denote its unit section.
We write
\[
\omega_{A/F} \coloneqq e^\ast \Omega^1_{A/F} \cong \pi_\ast \Omega^1_{A/F}
\]
for the co-Lie algebra of $A/F$ and abbreviate $\omega_{A/F}^g = \Lambda^g \omega_{A/F}$.
For an isogeny $\varphi: A \to B$ over $F$, we have homomorphisms $\varphi^\ast: \omega_{B/F} \to \omega_{A/F}$ and 
\[
\varphi_{\#} \coloneqq (\varphi^\vee)^\ast: \omega_{A^\vee/F} \to \omega_{B^\vee/F}
\] 
by functoriality. Since $F$ is a field of characteristic zero, all isogenies over $F$ are étale, so that $\varphi^\ast$ and $\varphi_{\#}$ are isomorphisms.
We let $H^1_\dR(A/F)$ (or sometimes $h^1_\dR(A)$) denote the first algebraic de Rham cohomology of $A/F$. 
Via the CM-structure, $\omega_{A/F}$, $\Lie(A/F)$ and $H^1_\dR(A/F)$ acquire the structures of $L \otimes F$-modules.
There is a canonical isomorphism $R^1\pi_\ast \cO_A \cong \Lie(A^\vee/F)$ and an exact sequence
\[
0 \to \omega_{A/F} \to H^1_\dR(A/F) \to \Lie(A^\vee/F) \to 0
\]
of $L \otimes F$-modules.

\subsection{Serre's tensor construction}\label{section_serre_tensor}
Let us collect a few facts about Serre's tensor construction. More details can for example be found in \cite{CCO}, 1.7.4.
Let $A/F$ be an abelian variety with CM by $\cO_L$. For every finitely generated projective $\cO_L$-module $\fra$, there is an abelian variety 
$A(\fra) \coloneqq A \otimes_{\cO_L} \fra$
defined over $F$ which represents the functor
\[
\Sch_F^\op \to \Ab, \;\;\; (A \otimes_{\cO_L} \fra)(T) = A(T) \otimes_{\cO_L} \fra.
\]
If $\fra, \frb$ are fractional ideals of $L$ with $\fra \subset \frb$, one has a canonical isogeny $A(\fra) \to A(\frb)$
of degree $[\frb:\fra]$. In the case where $\fra \subset \cO_L$ is an integral ideal, we denote the isogeny induced by the inclusion $\cO_L \subset \fra^{-1}$ by 
\[
[\fra]: A \to A(\fra^{-1})
\]
and define
$A[\fra] \coloneqq \ker([\fra])$.

\begin{remark}\label{serretensorCM}
If $K$ is a CM-subfield of $L$ and $B$ an abelian variety with complex multiplication by $\cO_K$ then $B \otimes_{\cO_K} \cO_L$ canonically acquires complex multiplication by $\cO_L$.
Conversely, let $A/F$ be an abelian variety with CM by $\cO_L$. It follows (cf.\ \cite{CCO}, Theorem 1.3.1.1) that $A$ is isotypical, i.e.\ $F$-isogenous to $B^n$ where $B/F$ is a simple abelian variety of CM-type. Then the center of $\End^0(B)$ is a CM-field $K$ with $[L:K] = n$. We may assume that $B$ has CM by $\cO_K$. Via
$K \subset M_{n \times n}(\End^0(B)) \cong \End^0(A)$,
one obtains $K \subset L$ as $L$ is its own centralizer in $\End^0(A)$. This way one obtains an $\cO_L$-linear $F$-isogeny $A \simeq B \otimes_{\cO_K} \cO_L$.
\end{remark}

\subsection{The main theorem of complex multiplication}
In the following, let $L$ denote a number field containing a CM-field and let $\Phi$ be a CM-type of $L$. Let $E$ denote the reflex field of $(L,\Phi)$ and let $\Phi^\ast \in I_E$ be the reflex CM-type. For each embedding $\sigma \in J_L$, we choose a lift $w_\sigma \in G_\QQ$ such that $cw_\sigma = w_{c\sigma}$. We define the type transfer
$\tr_\Phi: G_\QQ \to G_L^\ab$
by the formula
\[
\tr_{L,\Phi}(\tau) = \prod_{\sigma \in \Phi} (w_{\tau \sigma}^{-1} \tau w_\sigma)^{-1} \mod G_L^c,
\]
where $G_L^c$ denotes the topological closure of the commutator subgroup of $G_L$. The type transfer satisfies the following properties:
 
\begin{proposition}\label{properties_type_transfer} 
The map $\tr_{L,\Phi}$ is a well-defined, continuous and independent of the choice of lifts $w_\sigma$ for $\sigma \in J_L$. It satisfies the following properties:
\begin{enumerate}
\item For all $\tau_1, \tau_2 \in G_\QQ$, one has $\tr_{L,\Phi}(\tau_1 \tau_2) = \tr_{L,\Phi \circ \tau_2^{-1}}(\tau_1) \cdot \tr_{L,\Phi}(\tau_2)$.
\item Let $L' \supset L$ be a finite extension and $\Phi' = N_{L'/L}^\ast(\Phi)$. Let $\tr_{L'/L}: G_L^\ab \to G_{L'}^\ab$ denote the usual transfer map. One has $\tr_{L',\Phi'} = \tr_{L'/L} \circ \tr_{L,\Phi}$.
\item The restriction $\tr_{L,\Phi}: G_E \to G_L^\ab$ is a homomorphism and the diagram
\[
\begin{tikzcd}
E_f^\times \arrow{r}{({\Phi^\ast})^{-1}} \arrow{d}[swap]{r_K} & L_f^\times \arrow{d}{r_L} \\
G_E^\ab \arrow{r}{\tr_{L,\Phi}} & G_L^\ab
\end{tikzcd}
\]
is commutative.
\end{enumerate}
\end{proposition}

\begin{proof}
Except for (iii), the proposition is a straight-forward consequence of the definitions. In the setting of (ii), the transfer map is injective and hence one can reduce assertion (iii) to the case where $L$ is a CM-field. A proof can now be found in \cite{lang}, Ch.7, Theorem 1.1.
\end{proof}

If $L$ is a CM-field and $\tau \in G_\QQ$, we define 
\begin{align}\label{coset_cm_case}
L(\Phi, \tau) \coloneqq \{\lambda \in L_f^\times \mid r_L(\lambda) = \tr_{L,\Phi}(\tau) \textrm{ and } \lambda \lambda^c \chi_\cyc(\tau) \in L^\times\}.
\end{align}
The set $L(\Phi,\tau)$ forms an $L^\times$-coset in $L_f^\times$, cf.\ \cite{lang}, Ch.7, Theorem 2.2. 
If, more generally, $L$ is a totally imaginary field containing a CM-field $K$ and $\Phi$ is a type of $L$, which is lifted from a CM-type $\Phi_K$ of $K$, we define
\begin{align}\label{modified_coset}
L(\Phi,\tau) \coloneqq L^\times \cdot K(\Phi_K,\tau).
\end{align}
By Proposition \ref{properties_type_transfer} (iii) and the compatibility of the Artin map with the transfer, this definition is independent of the choice of $(K,\Phi_K)$ and we have 
$r_L(\lambda) = \tr_{L,\Phi}(\tau)$
for any $\lambda \in L(\Phi,\tau)$.
We record two properties of the cosets (\ref{modified_coset}) for later use:

\begin{proposition}\label{properties_of_cosets}
Let $K$ be a CM-field and $\Phi_K$ a CM-type. Let $L$ be a number field containing $K$ and let $\Phi$ be the CM-type lifted from $\Phi_K$. Let $E$ denote the reflex field of $(L,\Phi)$.
\begin{enumerate}
\item For any $\tau_1, \tau_2 \in G_\QQ$, one has $L(\Phi, \tau_1 \tau_2) = L(\Phi \circ \tau_2^{-1}, \tau_1) \cdot L(\Phi, \tau_2)$.
\item If $\tau \in G_E$ and $e \in E_f^\times$ satisfies $r_E(e) = \tau|_{E^\ab}$, one has
$L(\Phi,\tau) = L^\times \cdot \Phi^\ast(e)^{-1}$.
\end{enumerate}
\end{proposition}

\begin{proof}
The assertions immediately reduce to the case where $L$ is a CM-field. Statement (i) follows from Proposition \ref{properties_type_transfer} (ii) and assertion (ii) follows from Proposition \ref{properties_type_transfer} (iv).
\end{proof}

For the statement of the following version of the main theorem of complex multiplication, we use the notations $h^1_B(A) \coloneqq H^1(A(\CC),\QQ)$ and $h^1_f(A) \coloneqq (\varprojlim_n H^1_\et(A \times_F \ol{F},\ZZ/n\ZZ)) \otimes_\ZZ \QQ$ for Betti and finite adelic cohomology, cf.\ Section \ref{section_absolute_hodge}.

\begin{theorem}[Shimura-Taniyama, Tate, Langlands, Deligne] \label{main_theorem_of_cm}
Let $F$ be a number field and $A/F$ be an abelian variety of CM-type $(L,\Phi)$. Let $a \in h^1_B(A)$ be an $L$-basis. Let $\tau \in G_\QQ$ and $\lambda \in L(\Phi,\tau)$. Then there exists a unique $L$-basis $b \in h^1_B(A^\tau)$ such that the diagram
\[
\begin{tikzcd}
L_f \arrow{r}{a_f} \arrow{d}[swap]{\lambda} & h^1_f(A) \arrow{d}{\tau} \\
L_f \arrow{r}{b_f} & h^1_f(A^\tau)
\end{tikzcd}
\]
commutes. Here $a_f$ and $b_f$ are the $L_f$-linear isomorphisms corresponding to the bases $I_f(a \otimes 1)$ and $I_f(b \otimes 1)$, respectively.
\end{theorem}

\begin{proof}
In the case where $L$ is a CM-field, this is a dual version (without the additional data of polarizations) of \cite{milne}, Theorem 1.3. The general case reduces to the CM-case by Remark \ref{serretensorCM}.
\end{proof}

\begin{remark}
Equivalently, if $a \in h^1_B(A)$ and $b \in h^1_B(A^\tau)$ are given $L$-bases, there exists a unique element $\lambda \in L(\Phi,\tau)$ such that the diagram in Theorem \ref{main_theorem_of_cm} commutes.
For a generalization of the above theorem to motives for arbitrary Hecke characters, we refer the reader to \cite{blasius}, 4.1 and 4.2.
\end{remark}

Let us also state the following equivalent formulation of Theorem \ref{main_theorem_of_cm} in terms of complex uniformizations:

\begin{theorem}\label{main_theorem_uniformizations}
Let $F$ be a number field and $A/F$ be an abelian variety of CM-type $(L,\Phi)$. Let $\theta: \CC^\Phi/\fra \to A(\CC)$ be an $\cO_L$-linear complex uniformization where $\fra$ is a fractional ideal of $L$. Let $\tau \in G_\QQ$ and $\lambda \in L(\Phi,\tau)$. Then there exists an $\cO_L$-linear complex uniformization $\theta_\tau: \CC^{\tau \Phi}/ [\lambda]^{-1}_L \fra  \to A^\tau(\CC)$ such that the diagram
\begin{equation}\label{maintheoremdiagram}
\begin{tikzcd}
 L/\fra \arrow{d}{\lambda^{-1}} \arrow{r}{\theta} & A(\ol{\QQ})_\tor \arrow{d}{\tau} \\
 L/ [\lambda]^{-1}_L \fra \arrow{r}{\theta_\tau} & A^\tau(\ol{\QQ})_\tor
\end{tikzcd}
\end{equation}
commutes.
\end{theorem}

\begin{proof}
This is part of \cite{milne}, Theorem 1.1. In loc.\ cit., Lemma 1.4 the assertion is shown to be equivalent to Theorem \ref{main_theorem_of_cm}.
\end{proof}

Let us briefly recall the definition of the left vertical map in (\ref{maintheoremdiagram}): For any finite place $v$ write $\fra_v = \fra \cO_{L_v}$. The inclusions $L \subset L_v$ induce an $\cO_L$-linear isomorphism $L/\fra \to \bigoplus_{v} L_v/\fra_v$. Any idele $\lambda \in L_f^\times$ defines a map $\lambda: \bigoplus_v L_v/\fra_v \to \bigoplus_v L_v/ \lambda_v \fra_v$ by $(x_v)_v \mapsto (\lambda_v x_v)_v$ and hence gives rise to a map $\lambda : L/\fra \to L/ [\lambda]_L \fra$.

\begin{remark}
In this paper the general cosets (\ref{modified_coset}) will be used but their explicit description in terms of the type transfer is strictly-speaking only relevant to us for automorphisms which fix the reflex field of $(L,\Phi)$, i.e.\ we only need the reciprocity law of Shimura-Taniyama. The main theorem of complex multiplication in its general form is crucially used in \cite{blasius}, Theorem 4.5.2 in order to show that the reflex motive, which a priori is defined from $M(\chi)$ as a suitable direct factor in an exterior power of its dual, is the motive of a Hecke character. Up to a Tate twist this Hecke character agrees with $\Xi$ introduced in Section \ref{setup_hecke_character}.
\end{remark}

\section{CM-motives}\label{section_3_cm_motives} 
We will now introduce the motives that we will work with throughout this paper. By definition, the category of CM-motives is obtained as the full Tannakian subcategory within the category of motives for absolute Hodge cycles spanned by Artin motives and abelian varietes which are potentially of CM-type. We will start the following section by recalling the construction of this category and we discuss the fact that it embeds fully faithfully into a certain category of compatible systems of realizations. Moreover, we recall various basic notions like extension and restriction of scalars, motives with coefficients and $L$-functions. Finally, we discuss the relationship between CM-motives and algebraic Hecke characters.

\subsection{Motives for absolute Hodge cycles and CM-motives}\label{section_absolute_hodge}
Let $L$ be a subfield of $\CC$ and let $\ol{L}$ denote the algebraic closure of $L$ in $\CC$.
Let $X/L$ be a smooth projective variety. We have the following cohomology theories: 
\begin{enumerate}
\item[•] For every embedding $\sigma \in J_L$, we have the Betti cohomology $h^i_\sigma(X) \coloneqq H^i(X^\sigma(\CC),\QQ)$. It is a finite-dimensional $\QQ$-vector space and it is equipped with a pure rational $\QQ$-Hodge structure $h^i_\sigma(X) \otimes \CC = \bigoplus_{p+q=i} h_\sigma^{p,q}(X)$. If $\sigma$ is a real embedding, complex conjugation induces an involution $F_{\infty,\sigma}$ on $h^i_\sigma(X)$.
When $\sigma = 1_L$ is the fixed embedding, we simply write $h^i_B(X)$ for $h^i_{1_L}(X)$.

\item[•] The algebraic de Rham cohomology $h^i_\dR(X) \coloneqq H^i(X, \Omega^\bullet_{X/L})$. It is a finite dimensional $L$-vector space and equipped with the Hodge filtration $F^i \coloneqq \im(H^i(X, \Omega^{\geq i}_{X/L}) \to H^i(X,\Omega^\bullet_{X/L})$.

\item[•] The finite adelic cohomology $h^i_f(X) = (\varprojlim_{n} H^i_\et(X_{\ol{L}},\ZZ/n\ZZ)) \otimes_{\ZZ} \QQ$. It is a finite free $\QQ_f$-module equipped with a $\QQ_f$-linear continuous action of $G_L$.
 
\item[•] For every $\sigma \in J_L$, there is a comparison isomorphism
\[
I_{\infty,\sigma}: h^i_\dR(X) \otimes_{L,\sigma} \CC \xto{\sim} h^i_\sigma(X) \otimes \CC,
\]
such that 
\[
I_{\infty,\sigma}(F^p h_\dR^i(X) \otimes_{L,\sigma} \CC) = \bigoplus_{p' \geq p, q} h_\sigma^{p',q}(X).
\]
Let us simply write $I_\infty$ if $\sigma$ is the fixed embedding $L \subset \CC$.
\item[•] For every embedding $\ol{\sigma} \in J_{\ol{L}}$ which restricts to $\sigma \in J_L$, we have a comparison isomorphism
\[
I_{f,\ol{\sigma}} : h^i_\sigma(X) \otimes \QQ_f \xto{\sim} h^i_f(X),
\] 
such that, for every $\tau \in G_L$, the diagram
\[
\begin{tikzcd}
h^i_\sigma(X) \otimes \QQ_f \arrow{r}{I_{f, \ol{\sigma}\tau}} \arrow{rd}[swap]{I_{f,\ol{\sigma}}} & h^i_f(X) \arrow{d}{\tau} \\
& h^i_f(X)
\end{tikzcd}
\]
commutes. When $\ol{\sigma}$ is the fixed embedding $\ol{L} \subset \CC$, we simply denote it by $I_f$.

\item[•] We define the $\QQ$-rational Hodge structure 
$\QQ(1)_B \coloneqq 2\pi i \QQ$
with $\QQ(1)_B \otimes \CC$ of pure bidegree $(-1,-1)$, the filtered $L$-vector space 
$\QQ(1)_\dR \coloneqq L$
with $\QQ(1)_\dR = F^{-1} \supset F^0 = 0$ and the $G_L$-representation 
$\QQ(1)_f \coloneqq (\varprojlim \mu_n(\ol{L}) ) \otimes_{\ZZ} \QQ$.
For every $\sigma \in J_L$, we identify 
\[
\QQ(1)_\dR \otimes_{L,\sigma} \CC \cong \CC \cong 2\pi i \QQ \otimes \CC = \QQ(1)_B \otimes \CC
\]
via the inclusion $2 \pi i \QQ \subset \CC$. For every embedding $\ol{\sigma}: \ol{L} \into \CC$, we identify 
\[
\QQ(1)_B \otimes \QQ_f \cong (\varprojlim_n 2\pi i \ZZ/2\pi i n \ZZ) \otimes \QQ \xto{\sim} (\varprojlim_n \mu_n(\CC)) \otimes \QQ \cong \QQ(1)_f
\]
via the exponential map and $\ol{\sigma}: \mu_n(\ol{L}) \cong \mu_n(\CC)$. For every integer $m \in \ZZ$, we put
$h^i_\sigma(X)(m) \coloneqq h^i_\sigma(X) \otimes \QQ(1)_B^{\otimes m}$, $h^i_\dR(X)(m) \coloneqq h^i_\dR(X) \otimes_L \QQ(1)_\dR^{\otimes m}$ and $h^i_f(X)(m) \coloneqq h^i_f(X) \otimes_{\QQ_f} \QQ(1)_f^{\otimes m}$ in the respective categories, and we define comparison isomorphisms $I_{\infty,\sigma}$ and $I_{f,\ol{\sigma}}$ for all $\sigma \in J_L$ and $\ol{\sigma} \in J_{\ol{L}}$ by using the above identifications for $\QQ(1)$.
\end{enumerate}

Let $X/L$ be a smooth projective variety. Fix an integer $p \geq 0$. An \emph{absolute Hodge cycle on $X$} is an element 
\[
x = (x_\sigma, x_\dR, x_f)_{\sigma \in J_L} \in \prod_{\sigma \in J_L} h_\sigma^{2p}(X)(p) \times h^{2p}_\dR(X)(p) \times h^{2p}_f(X)(p)
\]
satisfying 
$I_{\infty,\sigma}(x_\sigma) = x_\dR$, $I_{f,\ol{\sigma}}(x_\sigma) = x_f$ and $x_\dR \in F^0 h^{2p}_\dR(X)(p)$
for all $\sigma \in J_L$ and $\ol{\sigma} \in J_{\ol{L}}$ extending $\sigma$. Denote the $\QQ$-vector space of absolute Hodge cycles by $C^p_{\AH}(X)$.

One then defines the category $\M_L$ of \emph{motives for absolute Hodge cycles} in the usual way: Starting from the category whose objects are smooth projective varieties over $L$ and where the morphisms between connected varieties $X$ and $Y$ are given by 
$\Hom(X,Y) = C^{\dim(X)}_\AH(X \times Y)$, one first passes to the Karoubian envelope. The usual cohomological grading gives rise to a natural $\ZZ$-grading. The category $\M_L$ is then obtained by inverting the Lefschetz motive $H^2(\PP^1)$ and finally modifying the commutativity constraint of the tensor structure, which is induced by the Cartesian product of varieties. For details we refer the reader to \cite{DMOS}, II.6. One can show that $\M_L$ is a semisimple Tannakian category over $\QQ$. Let $\cV_L$ denote the category of smooth projective varieties over $L$. There is a canonical contravariant functor 
\[
\cV_L^{\op} \to \M_L, \;\;\; X \mapsto h(X)
\]
and we let $h(X) = \bigoplus_{i \in \ZZ} h^i(X)$ denotes its canonical grading.

The category of \emph{CM-motives} $\CM_L$ is defined as the full Tannakian subcategory of $\M_L$
generated by the Artin motives over $L$, i.e.\ motives of zero-dimensional varieties over $L$, and motives $h^1(A)$ for potentially CM abelian varieties $A/L$.
We remark that $\QQ(n)$ for $n \in \ZZ$ lies in $\CM_L$ since $\QQ(-1)$ can be obtained as $h^2(E)$ for any potentially CM elliptic curve over $L$.
Throughout the remainder of this paper, we will exclusively work within $\CM_L$ and we simply refer to its objects as motives.

\subsection{Realizations}
It follows from the construction that the category $\M_L$ embeds fully faithfully into a category of compatible systems of realizations, which essentially consists of the abstract data listed at the beginning of Section \ref{section_absolute_hodge}, cf.\ \cite{DMOS}, Ch.II, Theorem 6.7 (g). For this, all embeddings of the base field $L$ have to be taken into account. Deligne's theorem that every Hodge cycle on an abelian variety over an algebraically closed field of characteristic zero is already absolute Hodge allows to embed the subcategory $\CM_L$ of CM-motives into a category of realizations $\cR_L$, where only the fixed embedding $L \subset \CC$ has to be considered. In the following, we will review the precise definition of $\cR_L$ as well as the fully faithful embedding $\CM_L \into \cR_L$.

Let $\cR_L$ denote the category of tuples $H = \{H_B, H_\dR, H_f, I_\infty, I_f\}$ consisting of the following data:
\begin{enumerate}
\item[•] A finite dimensional $\QQ$-vector space $H_B$ equipped with a decomposition
\[
H_B \otimes \CC = \bigoplus_{p,q \in \ZZ} H_B^{p,q},
\]
into $\CC$-subspaces such that $(H_B^{p,q})^{1 \otimes c} = H_B^{q,p}$.
\item[•] A finite dimensional $L$-vector space $H_\dR$ equipped with a decreasing filtration $F^\bullet H_\dR$. 

\item[•] A finite free $\QQ_f$-module $H_f$ equipped with a continuous linear action of the absolute Galois group $G_L$. We denote the action by $v \mapsto v^\tau$ for $\tau \in G_L$ and $v \in H_f$.

\item[•] A $\CC$-linear isomorphism $I_\infty: H_\dR \otimes_L \CC \xto{\sim} H_B \otimes \CC$ satisfying 
\[
I_\infty(F^p H_\dR \otimes \CC) = \bigoplus_{p' \geq p, q} H_B^{p',q}.
\]

\item[•] A $\QQ_f$-linear isomorphism $I_f: H_B \otimes \QQ_f \xto{\sim} H_f$.
\end{enumerate}
A morphism $\varphi: H \to H'$ in $\cR_L$ is a triple $\varphi = (\varphi_B,\varphi_\dR,\varphi_f)$, where:
\begin{enumerate}
\item[•] $\varphi_B: H_B \to H'_B$ is a $\QQ$-linear homomorphism such that $\varphi_B(H_B^{p,q}) \subset (H'_B)^{p,q}$ for all $p,q \in \ZZ$;
\item[•] $\varphi_\dR : H_\dR \to H'_\dR$ is a $L$-linear homomorphism satisfying $\varphi_\dR(F^p H_\dR) \subset F^p H'_\dR$ for all $p \in \ZZ$;
\item[•] $\varphi_f: H_f \to H'_f$ is a $G_L$-equivariant $\QQ_f$-linear homomorphism;
\item[•] $\varphi_f \circ I_f = I_f \circ \varphi_B$; 
\item[•] $I_\infty \circ \varphi_\dR = \varphi_B \circ I_\infty$.
\end{enumerate}

For any proper smooth variety $X/L$, 
the cohomology theories described in Section \ref{section_absolute_hodge} give rise to a tensor functor 
\[
\CM_L \to \cR_L, \;\;\; M \mapsto \{M_B, M_\dR, M_f, I_\infty, I_f\}.
\]
A motive is \emph{pure of weight $w \in \ZZ$}, if $M_B^{pq} = 0$ unless $p+q = w$. In this paper, we will only encounter pure motives.

Let us recall the following consequence of Deligne's theorem (\cite{DMOS}, Ch.I, Theorem 2.11) that every Hodge cycle on an abelian variety over an algebraically closed field of characteristic zero is absolute Hodge:

\begin{proposition}\label{fully_faithful_realizations}
The functor $\CM_L \to \cR_L$ is fully faithful. In fact, even more is true: For specifying a morphism $\varphi: M \to N$ in $\CM_L$, it suffices to give a morphism $\varphi_B : M_B \to N_B$ of $\QQ$-rational Hodge structures and a $G_L$-equivariant homomorphism $\varphi_f : M_f \to N_f$ such that $\varphi_B$ and $\varphi_f$ are compatible via the comparison isomorphism $I_f$.
\end{proposition}

\begin{proof}
This follows immediately from the following two observations: First, Deligne's theorem implies that the functor $M \mapsto M_B$ from the category of CM-motives over the algebraic closure $\ol{L}$ into the category of $\QQ$-rational Hodge structures is fully faithful, cf.\ \cite{DMOS}, Ch.II, Theorem 6.25. Second, it follows from the definition of an absolute Hodge cycle, that any morphism $M \times_L \ol{L} \to N \times_L \ol{L}$ with $M, N \in \CM_L$, whose finite adelic realization is $G_L$-equivariant, is induced from a morphism $M \to N$. For a recollection of the extension of scalars $M \times_L \ol{L}$ we refer the reader to the next section.
\end{proof}

In order to account for the other embeddings of $L$, we employ the following formalism that is also used in \cite{blasius}:
For every $\tau \in G_\QQ$ and $M \in \CM_L$, there is a motive $M^\tau \in \CM_{L^\tau}$, which is induced by taking the base change along $\tau: L \to L^\tau$ on the level of varieties. There is a $\tau$-semilinear bijection
$M_\dR \to M^\tau_\dR$, $\omega \mapsto \omega^\tau$
as well as a $\QQ_f$-linear bijection
$M_f \to M_f^\tau$, $v \mapsto v^\tau$.
Note that, in contrast to the maps $M_f \to M_f^\tau$, the object $M^\tau$ only depends on the restriction $\tau|_L \in J_L$. If $\tau \in G_L$, we recover the $G_L$-action on $M_f$. 

\subsection{Extension and restriction of scalars}\label{section_extn_restr_of_scalars}
Let $L'/L$ be an extension of subfields of $\CC$. Base change from $L$ to $L'$ on the level of varieties induces the extension of scalars functor
\[
(-) \times_L L' : \CM_L \to \CM_{L'}.
\]
For a motive $M$ in $\CM_L$, the extension of scalars $M \times_L L' \in \CM_{L'}$ satisfies $(M \times_L L')_\dR = M_\dR \otimes_L L'$, $(M \times_L L')_f = M_f$ with the Galois action restricted to $G_{L'} \subset G_L$ and all the other data is unchanged.

Now assume that $L'/L$ is an extension of number fields. By regarding varieties over $L'$ as being defined over $L$, one obtains the restriction of scalars functor
\[
R_{L'/L} : \CM_{L'} \to \CM_L.
\] 
For every motive $N$ in $\CM_{L'}$, the restriction of scalars $R_{L'/L}(N) \in \CM_L$ satisfies
\begin{enumerate}
\item[•] $R_{L'/L}(N)_B = \bigoplus_{\sigma \in J_{L'/L}} N_B^\sigma$,
\item[•] $R_{L'/L}(N)_\dR = N_\dR$ regarded as a $L$-module, and
\item[•] $R_{L'/L}(N)_f = \bigoplus_{\sigma \in J_{L'/L}} M_f^\sigma$,
\end{enumerate}
where an element $\tau \in G_L$ acts via the $\QQ_f$-linear maps $\tau: M_f^\sigma \xto{\sim} M_f^{\tau \sigma}$, for $\sigma \in J_{L'/L}$.

\subsection{Motives with coefficients}
Let $T$ be a number field.
A motive $M$ defined over $L$ \emph{with coefficients in $T$} is an object in $\CM_L$ equipped with a ring homomorphism
$T \to \End(M)$.
A morphism $\varphi: M \to N$ between motives with coefficients in $T$ is a morphism in $\CM_L$ which is compatible with the structure maps in the obvious way. The collection $\Hom_T(M,N)$ of such morphisms naturally forms a $T$-module. We denote the category of motives over $L$ with coefficients in $T$ by $\CM_L(T)$. We call $\dim_T M_B$ the rank of the motive $M$.

For a motive $M$ in $\CM_L(T)$, the realizations $M_B$, $M_\dR$ and $M_f$ acquire the structures of finite free $T$-, $T \otimes L$- and $T_f$-modules, respectively. Moreover the comparison isomorphisms $I_\infty$ and $I_f$ are $T \otimes \CC$-linear and $T_f$-linear, respectively.
Let $T' \supset T$ be an extension of number fields. There is a motive $M \otimes_T T'$ in $\CM_L(T')$, the \emph{extension of coefficients of $M$ from $T$ to $T'$}, satisfying 
\[
(M \otimes_T T')_B = M_B \otimes_T T', \;\;\; (M \otimes_T T')_\dR = M_\dR \otimes_T T', \;\;\; 
(M \otimes_T T')_f = M_f \otimes_T T'.
\]
Let $N$ be a motive in $\CM_L(T')$. There exists a motive $R_{T'/T} N$, \emph{the restriction of coefficients of $N$ from $T'$ to $T$}, whose realizations $(R_{T'/T} N)_B$, $(R_{T'/T} N)_\dR$ and $(R_{T'/T} N)_f$ are given by $N_B$, $N_\dR$ and $N_f$ regarded as modules over $T$, $T \otimes L$ and $T_f$, respectively.
For any two motives $M, N \in \CM_L(T)$, one can define a tensor product $M \otimes_T N$, which in terms of realizations is given by $(M \otimes_T N)_B = M_B \otimes_T N_B$, $(M \otimes_T N)_\dR = M_\dR \otimes_{T \otimes L} N_\dR$ and $(M \otimes_T N)_f = M_f \otimes_{T_f} N_f$. Similarly, one can define $T$-linear symmetric and exterior powers in the category $\CM_L(T)$.

\subsection{Motivic $L$-functions}
Let $L$ and $T$ be number fields and $M$ an object in $\CM_L(T)$. 
Let $v$ be a finite place of $L$. Fix a decomposition group $D_v \subset G_L$ of $v$ and denote the inertia subgroup by $I_v$. Let $\Fr_v \in D_v$ denote a geometric Frobenius element at $v$. For any prime $p$ such that $v \nmid p$, we define the polynomial
\[
P_v(M,t) \coloneqq \det_{T_p}(1- \Fr_v \cdot t| M_p^{I_v}) \in T_p[t],
\]
where $T_p \coloneqq T \otimes \QQ_p$. 
Conjecturally, the polynomial $P_v(M,t)$ actually lies in $T[t]$ and is independent of the prime $p$ with $v \nmid p$.
Assuming that this is true for $M$, 
we define the \emph{$L$-function of $M$} as the Euler product
\[
L(M,s) \coloneqq \prod_{v} P_v(M,Nv^{-s})^{-1},
\]
where $Nv$ denote the cardinality of the residue field at $v$.
This converges for $\Re(s) \gg 0$ and on its domain of convergence $L(M,s)$ is a holomorphic function with values in $T \otimes \CC$. Moreover, if $S$ is a finite set of places of $L$, we put
$L_S(M,s) \coloneqq \prod_{v \notin S} P_v(M,Nv^{-s})^{-1}$
 
It follows from the definition that for any $n \in \ZZ$ one has
\[
L(M(n),s) = L(M,s+n).
\]
In the study of the values of $L(M,s)$ at integers, the above equality allows to restrict the attention to the value at $s =0$.

\subsection{Motives of Hecke characters}
Let $L$ and $T$ be number fields and assume for simplicity that $L$ is totally imaginary. 
We briefly review the relationship between rank $1$ motives in $\CM_L(T)$ and algebraic Hecke characters $\chi: L_f^\times \to T^\times$, also cf.\ \cite{blasius}, 3.2. 

\begin{definition}
Let $\chi: L_f^\times \to T^\times$ be an algebraic Hecke character. We say that a motive $M \in \CM_L(T)$ is of type $(L,\chi,T)$ if $M$ is of rank $1$ and for every $\tau \in G_L$ and $\ell \in L_f^\times$ satisfying $r_L(\ell) = \tau|_{L^\ab}$, one has 
\[
v^\tau = \chi_a(\ell)^{-1} \chi(\ell) \cdot v 
\]
for all $v \in M_f$.
\end{definition}

Let $M$ be a motive of type $(L,\chi,T)$. Let $\frf \subset \cO_L$ be the exact conductor of $\chi$. For any finite place $v | \frf$, one has $\chi|_{\cO_{L_v}^\times} \neq 1$ and hence $M_p^{I_v} = 0$ for any prime $p$ with $v \nmid p$. If $v \nmid \frf$, the inertia group $I_v$ acts trivially on $M_p$ and any geometric Frobenius element $\Fr_v$ acts via multiplication by $\chi(\pi_v) \in T^\times$ where $\pi_v \in L_v^\times \subset L_f^\times$ is a uniformizer at $v$. This establishes the independence of $p$ for the local Euler factors of $M$ and shows
$L(M,s) = L_\frf(\chi,s)$,
where $L_\frf(\chi,s)$ is the Hecke $L$-function of $\chi$ defined in (\ref{hecke_l_function_def}). If $\chi$ is only of conductor dividing $\frf$ and if $S$ denote the finite set of places of $L$ which divide $\frf$, then one has $L_S(M,s) = L_\frf(\chi,s)$. 

In the next theorem, we summarize the close relationship between CM-motives and algebraic Hecke characters. In fact, there is a bijection between the isomorphism classes of \emph{rank $1$} CM-motives and algebraic Hecke characters:

\begin{theorem}[Deligne, Langlands]\label{rank_1_motives_are_heckecharacters} \leavevmode
\begin{enumerate}
\item Any two motives of type $(L,\chi,T)$ are isomorphic in $\CM_L(T)$.
\item For every algebraic Hecke character $\chi: L_f^\times \to T^\times$, there exists a motive $M(\chi)$ of type $(L,\chi,T)$.
\item Every motive $M \in \CM_L(T)$ of rank $1$ is of type $(L,\chi,T)$ for a unique algebraic Hecke character $\chi: L_f^\times \to T^\times$.
\end{enumerate}
\end{theorem}

The correspondence between $T$-isomorphism classes of motives of rank $1$ in $\CM_L(T)$ with algebraic Hecke characters of $L$ with values in $T$ has the following properties: 

\begin{proposition}[cf.\ \cite{blasius}, Proposition 3.2.1]\label{dictionary_motives}
Let $\chi: L_f^\times \to T^\times$ be an algebraic Hecke character of weight $w(\chi)$.
Let $M(\chi)$ be a motive of type $(L,\chi,T)$ and let $\sigma \in G_\QQ$. Let $L'/L$ and $T'/T$ be finite extensions. 
\begin{enumerate}
\item $M(\chi)$ is pure of weight $w(\chi)$. 
\item $M(\chi)^\sigma$ is of type $(L^\sigma, \chi \circ \sigma^{-1}, T^\times)$.
\item $M(\chi) \times_L L'$ is of type $(L', \chi \circ N_{L'/L}, T)$. 
\item $M(\chi) \otimes_T T'$ is of type $(L, \chi, T')$.
\end{enumerate}
\end{proposition}

\subsection{Motives of CM abelian varieties}\label{section_weil_pairing}
Let $A/F$ be an abelian variety over a number field $F$ with complex multiplication by a totally complex field $L$. The CM-structure equips the motive $h^1(A)$ in $\CM_F$ with the structure of a motive with coefficients in $L$. By the theory of complex multiplication, there exists a unique algebraic Hecke character 
\[
\psi: \AA_F^\times \to L^\times
\]
such that $h^1(A)$ is of type $(F,\psi,L)$. Let $\Phi$ denote the CM-type of $A$. Let $E$ be the reflex field of $(L,\Phi)$ and let $\Phi^\ast$ denote the reflex CM-type. Then $E \subset F$ and the infinity-type of $\psi$ is given by $\psi_a = N_{F/E}^\ast(\Phi^\ast)$. In particular, $\psi$ is of weight $1$. 

Finally, let us recall the Weil pairing: There is a perfect pairing (in fact, an isomorphism)
\[
h^1(A) \otimes_L h^1(A^\vee) \to L(-1)
\]
of motives over $F$ with coefficients in $L$ which induces the corresponding well-known pairings in each realization: In the de Rham realization one has a canonical pairing $h^1_\dR(A) \otimes_{L \otimes F} h^1_\dR(A^\vee) \to L \otimes F$, which is an $L$-linear version of the pairing given by \cite{BBM}, Théorème 5.1.6. In the finite adelic realization, it induces a perfect $G_F$-equivariant pairing
$\langle \cdot, \cdot \rangle_A: h^1_f(A) \otimes_{L_f} h^1_f(A^\vee) \to L_f(-1)$ which is an dual and $L$-linear version of the usual Weil pairing. For later reference, let us observe that for any $\tau \in G_\QQ$ the diagram
\[
\begin{tikzcd}
h^1_f(A) \otimes h^1_f(A^\vee) \arrow{d}[swap]{\tau} \arrow{r}{\langle \cdot, \cdot \rangle_A} & L_f(-1) \arrow{d}{\tau} \\
h^1_f(A^\tau) \otimes h^1_f(A^{\tau,\vee}) \arrow{r}{\langle \cdot, \cdot \rangle_{A^\tau}} & L_f(-1)
\end{tikzcd}
\]
commutes. Similarly, there is a canonical pairing $h^1_B(A) \otimes_L h^1_B(A^\vee) \to (2\pi i)^{-1} L$ and the above pairings are compatible via the comparison isomorphisms. For further information, we refer the reader to \cite{hodgeIII}, 10.2.

\section{The motive \texorpdfstring{$M^\alpha$}{M alpha}}\label{section_5_M_alpha}
Let $M$ be a motive of type $(F,\psi,L)$ and let $\alpha \in I_L^+$. The aim of this section is to provide a detailed construction of a motive $M^\alpha$ of type $(F,\psi^\alpha, L^\alpha)$. We obtain $M^\alpha$ as a direct factor in the $d(\alpha)$th symmetric power of the restriction of coefficients $R_{L/\QQ} M$. The desired projector will be constructed on the level of realizations and we can essentially treat all the realizations simultaneously by means of a general linear algebra construction. We start the section by discussing this general construction using the language of algebraic tori and their representations. We note that the construction of $M^\alpha$ presented here has already been outlined by Blasius in \cite{blasius}, Proof of Proposition 5.5.1. We discuss the realizations of $M^\alpha$ is greater detail and construct maps $M_B \to M_B^\alpha$, $\gamma \mapsto \gamma^{\balpha}$ as well as $M_\dR \to M^\alpha_\dR$, $\omega \mapsto \omega^{\balpha}$. Finally, the associated periods are expressed in terms of periods of $M$.

\subsection{Algebraic $\TT_{L,F}$-Modules}\label{section_alg_T_modules}
Let $L$ be a number field and $F$ a field of characteristic zero. Fix an algebraic closure $\ol{F}$ as well as an embedding $\ol{\QQ} \subset \ol{F}$. The chosen embedding $\ol{\QQ} \subset \ol{F}$ induces a homomorphism
\begin{align}\label{galois_group_homomorphism}
G_F \to G_\QQ
\end{align}
by restriction.
We consider the algebraic torus
\[
\TT_{L,F} \coloneqq \Res_{L/\QQ} \Gm \times_\QQ F
\]
over $F$. Its character group $X^\ast(\TT_{L,F})$ 
identifies with the free abelian group $I_L = \ZZ[J_L]$. The canonical action of $G_F$ is induced by the $G_\QQ$-action on $I_L$ via the homomorphism (\ref{galois_group_homomorphism}). Equivalently, we can regard $I_L$ as the free abelian group on 
$\Hom(L, \ol{F})$, on which $G_F$ acts in the obvious way.
An \emph{algebraic $\TT_{L,F}$-module} or \emph{representation of $\TT_{L,F}$} is an $F$-vector space $V$ which, for every $F$-algebra $R$, is equipped with a $R$-linear action
\[
\rho_R: \TT_{L,F}(R) = (L \otimes_\QQ R)^\times \to \Aut_{R}(V \otimes_F R)
\]
that is functorial in $R$.
The representations of algebraic groups of multiplicative type, such as $\TT_{L,F}$, can be described as follows:

\begin{proposition}\label{torus_decomposition}
Let $V$ be an algebraic $\TT_{L,F}$-module. There is a decomposition
\begin{align}\label{decomposition_along_characters_pt1}
V = \bigoplus_{\Theta \in G_F \bs I_L} V_\Theta
\end{align}
indexed by the $G_F$-orbits of the character group of $\TT_{L,F}$. Here $V_\Theta$ consists of those elements $v \in V$ such that there exists $\mu \in \Theta$ with
$\rho_{\ol{F}}(\ell)(v \otimes 1) = \ell^\mu \cdot (v \otimes 1)$ for all $\ell \in L^\times$.
\end{proposition}

Let us consider a field extension $F'/F$ and fix an algebraic closure $\ol{F'}$ of $F'$ as well as an embedding $\ol{F} \subset \ol{F'}$. Again, $G_{F'}$ acts on $I_L$ via the restriction map $G_{F'} \to G_F$ and in particular $G_{F'}$ acts on each $G_F$-orbit in $I_L$.
The $F'$-vector space $V' = V \otimes_F F'$ acquires the structure of an algebraic $\TT_{L,F'}$-module in a canonical way and 
for every $G_F$-orbit $\Theta$ of $I_L$, one has 
\begin{equation}\label{decomposition_base_change}
V_\Theta \otimes_F F' = \bigoplus_{\Theta' \in G_{F'} \bs \Theta} V'_{\Theta'}
\end{equation}
as $F'$-subspaces of $V \otimes_F F'$.

We now discuss a particular example, which will turn out to be useful in the next section for constructing direct factors in some motive.
Let $V$ be a $L \otimes_\QQ F$-module and fix an integer $d \geq 0$.
For any $F$-algebra $R$, there is a functorial action of $(L \otimes_\QQ R)^\times$ on $\Sym^d_F(V) \otimes_F R$ by letting an element $\ell \in L^\times$ act via $\Sym^d(\ell)$.
In this way, $\Sym^d_F(V)$ acquires the structure of an algebraic $\TT_{L,F}$-module.
By Proposition \ref{torus_decomposition}, we then obtain a decomposition
\begin{align}\label{decomposition_of_sym}
\Sym^d_F(V) = \bigoplus_{\Theta \in G_F \bs I_L} \Sym^d_F(V)_\Theta
\end{align}
indexed by the orbits of the action of $G_F$ on $I_L$ through the restriction map $G_F \to G_\QQ$.

\begin{definition}
In the following, let $\alpha \in I_L^+$ be a type and let $d = d(\alpha)$ denote its degree. Let $\Theta_\alpha = G_\QQ \cdot \alpha$ denote the $G_\QQ$-orbit of the type $\alpha$. For a $L \otimes F$-module $V$, we define $V^\alpha$ to be the direct summand
\[
V^\alpha = \bigoplus_{\Theta \in G_F \bs \Theta_\alpha} \Sym_F^d(V)_\Theta.
\]
of $\Sym_F^d(V)$ in (\ref{decomposition_of_sym}).
\end{definition}

Given an extension field $F' \supset F$ with algebraic closure $\ol{F'}$ and an embedding $\ol{F} \subset \ol{F'}$, we immediately deduce from (\ref{decomposition_base_change}) that 
\begin{equation}\label{change_of_fields}
(V \otimes_F F')^\alpha = V^\alpha \otimes_F F'.
\end{equation}

Let us briefly describe the decomposition in the split case, i.e.\ we assume that the Galois closure of $L$ in $\ol{F}$ already lies in $F$, so that the torus $\TT_{L,F}$ is split and $G_F$ acts trivially on $I_L$.
In this situation, let us fix the following notation: 
We write
\begin{align}\label{example_decomposition}
V = \bigoplus_{\sigma \in J_L} V(\sigma),
\end{align}
for the decomposition induced by the canonical decomposition $L \otimes F = \prod_{\sigma \in J_L} F_\sigma$, where $F_\sigma$ denotes the $L$-algebra $\sigma:L \to F$.
There are canonical isomorphisms
$V(\sigma) = V \otimes_{L \otimes F, \sigma} F$,
where the tensor product is taken with respect to the homomorphism $\sigma: L \otimes F \to F$, $\ell \otimes x \mapsto \sigma(\ell)x$. For any $\mu \in I^+_L$, we set
\begin{align} \label{sym_mu_notation}
\Sym^\mu(V) \coloneqq \bigotimes_{\sigma \in J_L} \Sym^{\mu(\sigma)}_F(V(\sigma))
\end{align}
and analogously for the symmetric tensors $\TSym$. Applying symmetric powers to (\ref{example_decomposition}) yields a canonical isomorphism
\begin{align}\label{decomposition_of_sym_split_case}
\Sym_F^d(V) &\cong \bigoplus_{\mu \in I_L^{+,d}} \Sym^\mu(V),
\end{align}
which coincides with the decomposition of $\Sym^d(V)$ according to Proposition \ref{torus_decomposition}.

From the above discussion we can now deduce that the $F$-vector space structure on $V^\alpha$ canonically extends to a $L^\alpha \otimes F$-module structure. Note that
$\Theta_\alpha = \{\iota \alpha \mid \iota \in J_{L^\alpha}\}$.
By (\ref{decomposition_of_sym_split_case}), we have the decomposition
\[
V^\alpha \otimes_F \ol{F} = \bigoplus_{\iota \in J_{L^\alpha}} \Sym^{\iota \alpha}(V \otimes_F \ol{F}).
\]
We now let an element $t \in L^\alpha$ act on $V^\alpha \otimes_F \ol{F}$ via multiplication by $\iota(t)$ on each 
summand $\Sym^{\iota \alpha}(V \otimes_F \ol{F})$ and one immediately checks that the action of $t$ on $V^\alpha \otimes_F \ol{F}$ is $G_F$-equivariant, so that it descends to an $F$-linear endomorphism of $V^\alpha$. This equips $V^\alpha$ with the structure of an $L^\alpha \otimes F$-module.
In fact, this construction gives rise to a functor $(-)^\alpha: \Mod_{L \otimes F} \to \Mod_{L^\alpha \otimes F}$. Any $L \otimes F$-linear homomorphism $f: V \to W$ induces a $\TT_{L,F}$-equivariant homomorphism $\Sym^d(f): \Sym^d_F(V) \to \Sym^d_F(W)$. In particular, $\Sym^d(f)$ is compatible with the decomposition (\ref{decomposition_of_sym}) and hence induces a homomorphism
\[
f^\alpha: V^\alpha \to W^\alpha,
\]
which is easily checked to be $L^\alpha \otimes F$-linear.

\subsection{Construction of $M^\alpha$} 
Let $F$ and $L$ be number fields and $M$ a CM-motive over $F$ with coefficients in $L$. Let $\alpha \in I_L^+$ and let $d \coloneqq d(\alpha)$. The goal of this section is to construct from $M$ a motive $M^\alpha$ in $\CM_F(L^\alpha)$. In the case where $M$ is of type $(F,\psi,L)$ for an algebraic Hecke character $\psi: \AA_F^\times \to L^\times$, the motive $M^\alpha$ will be of type $(F,\psi^\alpha,L^\alpha)$, where $\psi^\alpha$ denotes the Hecke character 
\[
\psi^\alpha : \AA_F^\times \to L^\times \xto{\alpha} (L^\alpha)^\times.
\]
For the construction, let us consider the motive
\[
N \coloneqq \Sym^d(R_{L/\QQ} M),
\]
where $R_{L/\QQ} M$ denotes restriction of coefficients. The idea is to apply the functor $(-)^\alpha$ from the previous section to obtain compatible direct summands in each realization and to obtain $M^\alpha$ as the resulting direct summand in $N$ by virtue of Proposition \ref{fully_faithful_realizations}.
Applying the functor $(-)^\alpha$ constructed in Section \ref{section_alg_T_modules} to the $L$-vector space $M_B$, we obtain a direct summand $M_B^\alpha$ in $N_B = \Sym^d_\QQ(M_B)$. Similarly, we obtain direct summands $M_\dR^\alpha$ in $N_\dR = \Sym_F^d(M_\dR)$ and, for each prime number $\ell$, a direct summand $M_\ell^\alpha$ in $N_\ell = \Sym^d_{\QQ_\ell}(M_\ell)$. 
Let $P_B$, $P_\dR$ and $P_\ell$ denote the corresponding idempotent endomorphisms of $N_B$, $N_\dR$ and $N_\ell$, respectively.

\begin{lemma}\label{projector_compatibilities} The projectors $P_B$, $P_\dR$ and $P_\ell$ satisfy the following properties:
\begin{enumerate}
\item The endomorphism $P_B \in \End(N_B)$ respects the Hodge structure, i.e.\ for all $p,q \in \ZZ$, one has
\[
P_B(N_B^{p,q}) \subset N_B^{p,q}.
\]
\item For each prime $\ell$, the endomorphism $P_\ell \in \End(N_\ell)$ is $G_F$-equivariant.
\item The comparison isomorphisms $I_\infty: N_\dR \otimes_F \CC \to N_B \otimes \CC$ and $I_\ell : N_B \otimes \QQ_\ell \to N_\ell$ are compatible with $P_\dR$, $P_B$ and $P_\ell$.
\end{enumerate}
\end{lemma}

\begin{proof}
By (\ref{decomposition_base_change}) and (\ref{decomposition_of_sym_split_case}), we have a decomposition
\begin{align}\label{random_decomposition_used_once}
N_B \otimes \CC = \bigoplus_{\mu \in I_L^{+,d}} \Sym^\mu(M_B \otimes \CC)
\end{align}
where 
$\Sym^\mu(M_B \otimes \CC) = \bigotimes_{\sigma \in J_L} \Sym_\CC^{\mu(\sigma)}(M_B \otimes_{L,\sigma} \CC)$, and $P_B$ is the projector onto the summand indexed by the Galois orbit of $\alpha$.
Since the Hodge decomposition is $L$-linear, the decomposition(\ref{random_decomposition_used_once}) respects the Hodge bigradings.
For (ii), one has to show that the direct summand $M_\ell^\alpha$ is stable under the action of $G_F$. It suffices to check this after base change to $\ol{\QQ}_\ell$. Then from (\ref{decomposition_base_change}) and (\ref{decomposition_of_sym_split_case}), one obtains
\[
N_\ell \otimes_{\QQ_\ell} \ol{\QQ}_\ell = \bigoplus_{\mu \in I_L^{+,d}} \Sym^\mu (M_\ell \otimes \ol{\QQ}_\ell),
\]
where
$\Sym^\mu (M_\ell \otimes \ol{\QQ}_\ell) = \bigotimes_{\sigma \in J_L} \Sym_{\ol{\QQ}_\ell}^{\mu(\sigma)} ( M_\ell \otimes_{L \otimes \QQ_\ell, \sigma} \ol{\QQ}_\ell)$.
Since the $G_F$-representation $M_\ell$ is $L$-linear, this is a decomposition of $G_F$-representations and the assertion follows. 
Finally, since $I_\infty: N_\dR \otimes_F \CC \to N_B \otimes \CC$ identifies with the $d$th $\CC$-linear symmetric power of the $L \otimes \CC$-linear comparison isomorphism $I_{\infty,M}$, it is a $\TT_{L,\CC}$-equivariant isomorphism and by construction it restricts to the isomorphism $I_{\infty,M}^\alpha: (M_\dR \otimes_F \CC)^\alpha \to (M_B \otimes \CC)^\alpha$. The desired compatibility now follows from (\ref{change_of_fields}).
The compatibility of $I_\ell$ with $P_B$ and $P_\ell$ follows completely analogously since $I_\ell: N_B \otimes \QQ_\ell \to N_\ell$ is induced by the $L \otimes \QQ_\ell$-linear isomorphism $I_{\ell,M}: M_B \otimes \QQ_\ell \to M_\ell$ on $d$th symmetric powers.
\end{proof}

From the above discussion and Proposition \ref{fully_faithful_realizations} we obtain an idempotent endomorphism
\[
P \in \End(\Sym^d(M_\QQ))
\]
which agrees with the projectors $P_B$, $P_\dR$ and $P_\ell$ in the respective realizations.

\begin{definition}
Let $M$ be a CM-motive over $F$ with coefficients in $L$ and let $\alpha \in I_L^+$ be of degree $d$. We define $M^\alpha = P \cdot \Sym^d(M_\QQ)$.
\end{definition}

\begin{lemma}
The motive $M^\alpha$ canonically acquires coefficients in the number field $L^\alpha$.
\end{lemma}

\begin{proof}
Recall that each of the realizations $M^\alpha_B$, $M^\alpha_\dR$ and $M^\alpha_\ell$ has the structure of a module over $L^\alpha$, $L^\alpha \otimes_\QQ F$ and $L^\alpha \otimes_\QQ \QQ_\ell$, respectively. Moreover, by the proof of Lemma \ref{projector_compatibilities}, the comparison isomorphism $I_\infty: M_\dR^\alpha \otimes_F \CC \to M_B^\alpha \otimes \CC$ is given, via the identifications of Lemma \ref{change_of_fields}, by applying the functor $(-)^\alpha$ to the $L \otimes \CC$-linear isomorphism $I_{\infty,M}$. In particular, it is $L^\alpha \otimes \CC$-linear. Similarly, the comparison isomorphism $I_\ell: M^\alpha_B \otimes \QQ_\ell \to M_\ell^\alpha$ is $L^\alpha \otimes \QQ_\ell$-linear.
Hence, the motive $M^\alpha$ acquires coefficients in $L^\alpha$ and in each realization the induced $L^\alpha$-module structure is given as in subsection \ref{section_alg_T_modules}.
\end{proof}

Next, we show that $M^\alpha$ is a motive of type $(F,\psi^\alpha,L^\alpha)$. Let us first make the following observation about algebraic homomorphisms:
Let $\alpha_\ell: (L \otimes \ol{\QQ}_\ell)^\times \to (L^\alpha \otimes \ol{\QQ}_\ell)^\times$ denote the homomorphism induced $\alpha$. Then, under the usual identification $L^\alpha \otimes \ol{\QQ}_\ell = \ol{\QQ}_\ell^{J_{L^\alpha}}$, the homomorphism $\alpha_\ell$ is given by 
\begin{align}\label{lpart_of_alpha}
\alpha_\ell(y) = \Bigl( \prod_{\sigma \in J_L} \sigma(y)^{\iota \alpha(\sigma)} \Bigr)_{\iota \in J_{L^\alpha}},
\end{align}
where $y \in (L \otimes \ol{\QQ}_\ell)^\times$ and $\sigma: L \otimes \ol{\QQ}_\ell \to \ol{\QQ}_\ell$ denotes the homomorphism given by $\ell \otimes z \mapsto \sigma(\ell)z$.

\begin{proposition}
Let $M$ be a motive of type $(F,\psi,L)$. Then $M^\alpha$ is of type $(F,\psi^\alpha,L^\alpha)$.
\end{proposition}

\begin{proof}
By Lemma \ref{decomposition_base_change} and (\ref{decomposition_of_sym_split_case}), one has
\[
M_B^\alpha \otimes \ol{\QQ} = \bigoplus_{\mu \in G_\QQ \cdot \alpha} \Sym^\mu(M_B \otimes \ol{\QQ})
\]
where each summand 
is a $1$-dimensional $\ol{\QQ}$-vector space.
As the orbit $G_\QQ \cdot \alpha$ is of cardinality $[L^\alpha:\QQ]$, it follows that $M_B^\alpha$ is a $1$-dimensional $L^\alpha$-vector space and hence $M^\alpha$ is a motive of rank $1$.
We now show that, for a fixed rational prime $\ell$, the $G_F$-representation $M^\alpha_\ell$ is determined by the Hecke character $\psi^\alpha$. To check this, we may pass to the algebraic closure $\ol{\QQ}_\ell$. Then, by Proposition \ref{decomposition_base_change}, (\ref{decomposition_of_sym_split_case}) and by using that the $G_\QQ$-orbit of $\alpha$ is given by $\{\iota \alpha \mid \iota \in J_{L^\alpha}\}$, we have
\[
M^\alpha_\ell \otimes_{\QQ_\ell} \ol{\QQ}_\ell = \bigoplus_{\iota \in J_{L^\alpha}} \Sym^{\iota \alpha}(M_\ell \otimes_{\QQ_\ell} \ol{\QQ}_\ell),
\]
and the module $L^\alpha \otimes \ol{\QQ}_\ell$ acts in the obvious way via $L^\alpha \otimes \ol{\QQ}_\ell \cong \ol{\QQ}_\ell^{J_{L^\alpha}}$. Let $\tau \in G_F$ and $x \in F_f^\times$ such that $r_F(x) = \tau|_{F^\ab}$. 
Since $M$ is of type $(F,\psi,L)$, the action of $\tau$ on $M_\ell \otimes_{\QQ_\ell} \ol{\QQ}_\ell$ is given via multiplication by $\psi^{-1}_{a,\ell}(x) \psi(x) \in (L \otimes \ol{\QQ}_\ell)^\times$. Thus, $\tau$ acts on $M^\alpha_\ell \otimes_{\QQ_\ell} \ol{\QQ}_\ell$ via multiplication by the element
\[
\Bigl( \prod_{\sigma \in J_L} \sigma \bigl( \psi^{-1}_{a,\ell}(x) \psi(x) \bigr)^{\iota \alpha(\sigma)} \Bigr)_{\iota \in J_{L^\alpha}}
\]
in $\ol{\QQ}_\ell^{J_{L^\alpha}}$
and, by (\ref{lpart_of_alpha}) above, this corresponds to the element
\[
\alpha_\ell (\psi^{-1}_{a,\ell}(x) \psi(x)) = \psi^{-\alpha}_{a,\ell}(x) \cdot \psi^\alpha(x)
\]
in $(L^\alpha \otimes \ol{\QQ}_\ell)^\times$.
\end{proof}

\subsection{Construction of elements in $M_B^\alpha$ and $M_\dR^\alpha$}
Starting from elements in $M_B$ and $M_\dR$, we now explain how to construct elements in $M^\alpha_B$ and $M^\alpha_\dR$. In the de Rham realization, it is sufficient for our purposes to restrict to the case where the field of definition $F$ contains the Galois closure of $L$, so that the torus $\TT_{L,F}$ is split. In this case, the de Rham realization becomes easier to manage. More generally, one could define a map $V \to V^\alpha$ for any $L \otimes F$-module $V$.

\subsubsection*{Betti realization}
As before, let $F$ and $L$ be number fields and $M$ be a motive of type $(F,\psi,L)$. Let $\gamma \in M_B$. We will now construct elements $\gamma^{\boldsymbol{\alpha}} \in M_B^\alpha$ via Galois descent. Recall the notation introduced in (\ref{sym_mu_notation}) and that
\[
M_B^\alpha \otimes \ol{\QQ} = \bigoplus_{\iota \in J_{L^\alpha}} \Sym^{\iota \alpha}(M_B \otimes \ol{\QQ}).
\]
For every $\iota \in J_{L^\alpha}$, we define
\[
\gamma^{\iota \alpha} \coloneqq \bigotimes_{\sigma \in J_L} (\gamma \otimes_{\sigma} 1_{\ol{\QQ}})^{\iota \alpha(\sigma)} \in \Sym^{\iota\alpha}(M_B \otimes \ol{\QQ})
\]
and we set
$
\gamma^{\boldsymbol{\alpha}} \coloneqq (\gamma^{\iota\alpha})_{\iota \in J_{L^\alpha}} \in M_B^\alpha \otimes \ol{\QQ}$.
Here, we write $\gamma \otimes_\sigma 1_{\ol{\QQ}}$ to emphasize that the element lies in $M_B \otimes_{L,\sigma} \ol{\QQ}$. 
The action of an element $\tau \in G_\QQ$ sends $\gamma^{\iota \alpha}$ to $\bigotimes_{\sigma \in J_L} (\gamma \otimes_{\tau \sigma} 1)^{\iota\alpha(\sigma)} = \bigotimes_{\sigma \in J_L} (\gamma \otimes_\sigma 1)^{\tau \iota \alpha(\sigma)} = \gamma^{\tau \iota \alpha}$. It follows that $\tau(\gamma^{\boldsymbol{\alpha}}) = \gamma^{\boldsymbol{\alpha}}$.
Hence, we obtain an element
\begin{align}\label{betti_element_in_m_alpha}
\gamma^{\boldsymbol{\alpha}} \in M_B^\alpha.
\end{align}

\subsubsection*{De Rham realization}
For simplicity, we restrict ourselves to the case where $F$ contains the Galois closure of $L$. Then the torus $\TT_{L,F}$ is split and we may regard every embedding of $L$ to take values in $F$. The de Rham realization of $M^\alpha$ is given by
\[
M_\dR^\alpha = \bigoplus_{\iota \in J_{L^\alpha}} \Sym^{\iota \alpha}(M_\dR)
\]
Now, if we are given $\omega \in M_\dR$ and $\iota \in J_{L^\alpha}$, we define
\[
\omega^{\iota \alpha} \coloneqq \bigotimes_{\sigma \in J_L} (\omega \otimes_\sigma 1_F)^{\iota \alpha(\sigma)} \in \Sym^{\iota \alpha}(M_\dR)
\]
and set
\begin{align}\label{de_rham_element_in_m_alpha}
\omega^{\boldsymbol{\alpha}} \coloneqq (\omega^{\iota\alpha})_{\iota \in J_{L^\alpha}} \in M_\dR^\alpha.
\end{align}

\subsection{Periods of $M^\alpha$}
Let $M$ be a motive of type $(F,\psi,L)$ and assume that $F$ contains the Galois closure of $L$. 
Let $\gamma \in M_B$ be an $L$-basis and $\omega \in M_\dR$. Consider the $L \otimes \CC$-linear comparison isomorphism
\[
I_{\infty, M} : M_\dR \otimes_F \CC \to M_B \otimes \CC
\]
attached to the motive $M$. There is a unique element $\Omega \in L \otimes \CC$ satisfying
\begin{align}\label{omega_period}
I_{\infty,M}(\omega) = \Omega \cdot \gamma.
\end{align}

\begin{notation}\label{OmegaNotation}
Let $\alpha \in I_L^+$.
For an element $\Omega \in L \otimes \CC$ and $\iota \in J_{L^\alpha}$, we define
\[
\Omega^{\iota \alpha} \coloneqq \prod_{\sigma \in J_L} \Omega(\sigma)^{\iota \alpha(\sigma)},
\]
where $\Omega = (\Omega(\sigma))_{\sigma \in J_L}$ under the canonical isomorphism $ L \otimes \CC = \CC^{J_L}$. Moreover, we put
\[
\Omega^{\boldsymbol{\alpha}} \coloneqq (\Omega^{\iota \alpha})_{\iota \in J_{L^\alpha}} \in L^\alpha \otimes \CC,
\]
where again, we use the identification $L^\alpha \otimes \CC = \CC^{J_{L^\alpha}}$.
\end{notation}

\begin{proposition}\label{MAlphaPeriod}
Let $M$ be a motive of type $(F,\psi,L)$ and assume that $F$ contains the Galois closure of $L$. Let $\gamma \in M_B$ be an $L$-basis, $\omega \in M_\dR$ and $\Omega \in L \otimes \CC$ as in (\ref{omega_period}). Then 
\[
I_{\infty, M^\alpha}(\omega^{\boldsymbol{\alpha}}) = \Omega^{\boldsymbol{\alpha}} \cdot \gamma^{\boldsymbol{\alpha}}.
\]
\end{proposition}

\begin{proof}
According to the canonical isomorphism $L \otimes \CC = \CC^{J_L}$, the isomorphism $I_{\infty,M}$ decomposes into $\CC$-linear isomorphisms
$
I_{\infty,\sigma}: M_\dR \otimes_{F \otimes L,\sigma} \CC \to M_B \otimes_{L,\sigma} \CC$,
which satisfy
\[
I_{\infty,\sigma}(\omega \otimes_\sigma 1_\CC) = \Omega_\sigma \cdot (\gamma \otimes_\sigma 1_\CC).
\]
Since the $L^\alpha \otimes \CC$-linear isomorphism $I_{\infty, M^\alpha}$ is induced by the $L \otimes \CC$-linear isomorphism $I_{\infty,M}$ by applying the functor $(-)^\alpha$, it now follows from the construction of the elements $\gamma^{\balpha}$ and $\omega^{\balpha}$ that
$I_{\infty,M^\alpha}(\omega^{\balpha}) = \Omega^{\balpha} \cdot \gamma^{\balpha}$, as desired.
\end{proof}

\section{The reflex motive}\label{section_6_reflex_motive}
We recall the reflex motive $M(\Xi)$ of Blasius and state his period relation which makes it possible to recover the Deligne period $c^+(\chi)$ from the reflex motive. Afterwards, we use the results of Section \ref{section_5_M_alpha} to obtain an alternative presentation of $M(\Xi)$ in terms of an abelian variety $A$ with complex multiplication by $L$. Finally, we show that the special element in the Betti realization in Blasius' period relation can be expressed explicitly in terms of Betti classes of $A$.

\subsection{Blasius' period relation}\label{setup_hecke_character}
In the following, let $L$ be a totally imaginary field containing a CM-field. Let $K$ be the maximal CM-subfield of $L$ and $\Phi_K$ a CM-type of $K$. Let $E$ denote the reflex field of $(K,\Phi)$ and let $\Phi^\ast \in I_E$ denote the reflex CM-type on $E$. Let $\Phi = N_{L/K}^\ast(\Phi_K)$ be the CM-type of $L$ lifted from $\Phi_K$. Let $\chi_a \in I_L$ be critical with respect to $\Phi$, i.e.\ $\chi_a$ is of Hecke character type and
\[
\chi_a = \beta - \alpha
\]
with $\alpha \in I_\Phi$ and $\beta \in I_{\ol{\Phi}}$ such that 
$\alpha(\sigma) \geq 1$ and $\beta(\ol{\sigma}) \geq 0$
for all $\sigma \in \Phi$, cf.\ Definition \ref{criticality_def}.
In particular, there exist $\alpha_0 \in I_{\Phi_K}$, $\beta_0 \in I_{\ol{\Phi}_K}$ such that 
$\alpha = N_{L/K}^\ast(\alpha_0)$ and $\beta = N_{L/K}^\ast(\beta_0)$
and there is an integer $w \in \ZZ$ such that 
\begin{align}\label{weight_condition}
\beta(\ol{\sigma}) - \alpha(\sigma) = w 
\end{align}
for all $\sigma \in \Phi$, i.e.\ $c \beta - \alpha = w\Phi$.

\begin{lemma} Let $\chi_a = \beta - \alpha$ be of Hecke character type as above.
\begin{enumerate}
\item The field $L^\alpha$ is a CM-field containing the reflex field $E$ of $(K,\Phi_K)$.
\item One has $L^\beta \subset L^\alpha$.
\end{enumerate}
\end{lemma}

\begin{proof}
Since $K$ is a CM-field, it is clear that $L^\alpha = K^{\alpha_0}$ is either totally real or a CM-field. Any $\tau \in G_\QQ$ which fixes $\alpha_0 \in I_{\Phi_K}$ also has to fix the CM-type $\Phi_K$ on which $\alpha_0$ is supported. It follows that $E \subset L^\alpha$. In particular, $L^\alpha$ is a CM-field.
For assertion (ii), consider $\tau \in G_\QQ$ such that $\tau \alpha = \alpha$. Since $\alpha$ is induced from a CM-field, it follows that $\tau$ fixes $c \alpha$ as well. Moreover, by (i), $\tau$ fixes the CM-type $\Phi$ and hence also its complex conjugate $\ol{\Phi}$. Finally, $\beta = w \ol{\Phi} + c \alpha$ by the weight condition (\ref{weight_condition}) and thus $\tau \beta = \beta$, which proves the claim.
\end{proof}

Now assume that $T$ is a CM-field containing $L^\alpha$ and that 
\[
\chi: L_f^\times \to T^\times
\]
is an algebraic Hecke character with infinity type $\chi_a$. 

\begin{definition}\label{def_sign_epsilon}
Let $\eta \in J_E$. We define the map
\[
\varepsilon_{\eta \Phi} : G_\QQ \to \{\pm 1\}
\]
as follows: For any $\tau \in G_\QQ$, we let $\varepsilon_{\eta \Phi}(\tau)$ be the sign of the bijection
\[
\langle c \rangle \bs J_L \xto{\sim} \eta \Phi \xto{\tau} \tau \eta \Phi \xto{\sim} \langle c \rangle \bs J_L,
\]
where the first and the last bijection are induced from the inclusions $\Phi \subset J_L$ and $\eta \Phi \subset J_L$.
\end{definition}
For $\tau_1, \tau_2 \in G_\QQ$, one immediately checks that
\begin{align}\label{sign_epsilon_properties}
\varepsilon_{\eta\Phi}(\tau_1 \tau_2) = \varepsilon_{\tau_2 \eta \Phi}(\tau_1) \varepsilon_{\eta \Phi}(\tau_2).
\end{align}
In particular, the restriction $\varepsilon_{\Phi} : G_{E} \to \{ \pm 1\}$ is a character and we may regard it as a finite Hecke character $\varepsilon_{\Phi}: E_f^\times \to T^\times$. 

\begin{definition}\label{def_hecke_character_xi}
Let
$\Xi: E_f^\times \to T^\times$
denote the algebraic Hecke character
\begin{align}\label{def_character_xi}
\Xi = (\chi \circ \Phi^\ast)^{-1}  \lVert \cdot \rVert_E^{-d(\beta)}  \varepsilon_\Phi.
\end{align}
Note that the infinity-type of $\Xi$ is given by $\Xi_a = (\Phi^\ast)^\alpha \cdot (\ol{\Phi}^\ast)^{\beta}$.
\end{definition}
For any $\eta \in J_E$, we have the coset $L(\eta \Phi, \tau)$, defined in (\ref{modified_coset}), associated to the CM-type $\eta \Phi$ of $L$.
For any $\tau \in G_\QQ$, let us define the elements $\xi(\tau,\eta) \in T_f^\times$ by 
\begin{align}\label{def_xi_elements}
\xi(\tau,\eta) = \chi_a(\lambda)^{-1} \chi(\lambda) \chi_\cyc(\tau)^{-d(\beta)} \varepsilon_{\eta \Phi}(\tau),
\end{align}
where $\lambda \in L(\eta\Phi, \tau)$ can be chosen arbitrarily.

The next result asserts the existence of non-trivial elements $a_\eta$ in $M(\Xi)^\eta_B$ for each $\eta \in J_E$ whose images in the finite adelic realizations are related by the Galois action up to the factors $\xi(\tau,\eta)$:

\begin{proposition}[\cite{blasius}, Corollary 4.4.2]\label{existence_betti_system}
Let $M(\Xi)$ be a motive of type $(E,\Xi,T)$. For every $\eta \in J_E$ there exists a $T$-basis $a_\eta \in M(\Xi)^\eta_B$ such that for any $\tau \in G_\QQ$ the diagram
\[
\begin{tikzcd}
T_f \arrow{r}{I_f(a_\eta)} \arrow{d}[swap]{\xi(\tau,\eta)} & M(\Xi)_f^\eta \arrow{d}{\tau} \\
T_f \arrow{r}{I_f(a_{\tau \eta})} & M(\Xi)_f^{\tau \eta}
\end{tikzcd}
\]
is commutative, where the element $\xi(\tau,\eta) \in T_f^\times$ is defined as in (\ref{def_xi_elements}).
\end{proposition}

The following theorem allows to recover the Deligne period $c^+(\chi)$ from the reflex motive $M(\Xi)$:

\begin{theorem}[Blasius] \label{blasius_period_relation}
Let $\Xi$ be the Hecke character as defined in (\ref{def_character_xi}), $M(\Xi)$ a motive of type $(E,\Xi,T)$ and $\{a_\eta \mid \eta \in J_E\}$ as in Proposition \ref{existence_betti_system}. Then $F^{d(\alpha)+d(\beta)} RM(\Xi)_\dR$ is a $1$-dimensional $T$-vector space. If $\omega$ is a basis, there exists a unique element $v(\chi) \in (T \otimes \CC)^\times$ satisfying
\[
I_\infty(\omega) = v(\chi) \cdot \sum_{\eta \in J_E} e(\eta) a_\eta,
\]
and one has
\[
(2\pi i)^{-d(\beta)} v(\chi) = c_+(\chi) \mod T^\times.
\]
\end{theorem}

\begin{proof}
This is \cite{blasius}, Proposition 5.3.4 in the case where $L$ is a CM-field. For general $L$, the result is \cite{blasius97}, Theorem H.3.  The proof works just as in the CM-case if one uses the sign from Definition \ref{def_sign_epsilon} instead of the slightly different sign defined in \cite{blasius}, Theorem 4.5.2.
\end{proof}

\subsection{An alternative presentation of the reflex motive}\label{subsection_alt_construction}
Let $F/E$ be a finite Galois extension and $A/F$ an abelian variety of CM-type $(L,\Phi)$ such that the CM-character $\psi: F_f^\times \to L^\times$ of $A$ satisfies
\begin{equation}\label{xiandpsi}
\Xi \circ N_{F/E} = \psi^\alpha \smash{\ol{\psi}}^\beta.
\end{equation}
This is possible since the infinity-type of $\Xi$ is given by $(\Phi^\ast)^\alpha (\ol{\Phi^\ast})^\beta$.
By the constructions and results of Section \ref{section_5_M_alpha}, we have the motives $h^1(A)^\alpha$ and $h^1(A^\vee)^\beta$ of types $(F,\psi^\alpha,L^\alpha)$ and $(F,\smash{\ol{\psi}}^\beta, L^\beta)$, respectively. Let $h^1(A)^\alpha_T$ and $h^1(A^\vee)^\beta_T$ denote the extension of coefficients to $T$. Then
\[
N \coloneqq h^1(A)_T^\alpha \otimes h^1(A^\vee)_T^\beta
\]
is a motive of type $(F,\psi^\alpha \smash{\ol{\psi}}^\beta, T)$. 
Let us now fix a motive $M(\Xi)$ of type $(E,\Xi,T)$. By (\ref{xiandpsi}) there exists an isomorphism 
\[
M(\Xi) \times_E F \cong N
\]
of CM-motives over $F$ with coefficients in $T$ and we fix such a choice. This induces an isomorphism
\[
R_{F/E}(M(\Xi) \times_E F) \cong R_{F/E} N
\]
and by transport of structure, we obtain a $G_{F/E}$-action on $R_{F/E} N$ whose coinvariants (and invariants) identify with $M(\Xi)$. In fact, there exists an idempotent endomorphism $e_\Xi$ of $N$ whose image is isomorphic to $M(\Xi)$, i.e.\
\begin{equation}\label{alt_constr_m_xi}
M(\Xi) \cong e_\Xi \cdot R_{F/E}(h^1(A)^\alpha_T \otimes h^1(A^\vee)^\beta_T).
\end{equation}
However, for our purposes it is actually sufficient that $M(\Xi)$ can be realized as a quotient.

\begin{remark}
By generalizing results of Goldstein-Schappacher, \cite{goldsteinschappacher}, to motives for Hecke characters, some additional information can be acquired: Using a theorem of Casselman it is possible to construct a suitable CM abelian variety $A/F$ satisfying (\ref{xiandpsi}) over an abelian extension $F/E$. The restriction of scalars $R_{F/E} N$ acquires coefficients in the $T$-algebra $\sT \coloneqq \bigoplus_{\sigma \in G_{F/E}} \Hom_T(N,N^\sigma)$, where the multiplication is defined by the rule $(\phi_\sigma \sigma) (\phi_\tau \tau) = (\phi_\tau^\sigma \circ \phi_\sigma) \sigma\tau$ and one can show that $\sT$ is a commutative semisimple $T$-algebra such that $R_{F/E} N$ is of rank $1$ over $\sT$. In particular, there exists a Hecke character $\wtilde{\Xi}: E_f^\times \to \sT^\times$ whose values generate $\sT$ and such that $\wtilde{\Xi} \circ N_{F/E} = \Xi \circ N_{F/E}$. It follows that there exists an idempotent $e_\Xi$ in $\sT$ such that $M(\Xi) = e_\Xi \cdot R_{F/E} N$. For more details we refer the reader to sections 4 and 9.1 of \cite{thesis}.
\end{remark}

\subsection{The Betti realization of $M(\Xi)$}\label{BettiRealizationOfMXi}
For each $\sigma \in J_F$, let $\gamma_\sigma \in h^1_B(A^\sigma)$ be an $L$-basis. Via the canonical $L$-linear pairing
$h^1_B(A^\sigma) \otimes_L h^1_B(A^{\vee,\sigma}) \to L(-1)_B = (2\pi i)^{-1} L$,
we can associate dual basis vectors
$\gamma_\sigma^\vee \in h^1_B(A^{\vee,\sigma})$.
For each $\eta \in J_E$, the Betti realization of $M(\Xi)^\eta$ is given by
\begin{align}\label{betti_realization_M_xi}
M(\Xi)^\eta_B = e_\Xi \cdot \bigoplus_{\sigma \in J_{F,\eta}} h^1_B(A^\sigma)^\alpha_T \otimes h^1_B(A^{\vee,\sigma})^\beta_T.
\end{align}
Recall the construction $\gamma \mapsto \gamma^{\balpha}$ from (\ref{betti_element_in_m_alpha}).
For each $\sigma \in J_F$ and any $L$-basis $\gamma_\sigma \in h^1_B(A^\sigma)$, we consider the $T$-basis
\[
\gamma_\sigma^{\balpha} \otimes (\gamma_\sigma^{\vee})^{\bbeta}
\]
of $h^1_B(A^\sigma)^\alpha_T \otimes h^1_B(A^{\vee,\sigma})^\beta_T$. 

\begin{proposition} \label{betticoefficients1}
For $\sigma \in J_F$ and $\tau \in G_\QQ$, let $\lambda(\tau, \sigma) \in L_f^\times$ denote the unique element satisfying
\[
I_f(\gamma_\sigma)^\tau = \lambda(\tau,\sigma) I_f(\gamma_{\tau \sigma}).
\]
Then one has
\[
I_f(\gamma_\sigma^{\balpha} \otimes (\gamma_\sigma^\vee)^{\bbeta})^\tau = \lambda(\tau,\sigma)^{\alpha-\beta} \cdot \chi_\cyc(\tau)^{-d(\beta)} \cdot I_f(\gamma_{\tau\sigma}^{\balpha} \otimes (\gamma_{\tau\sigma}^\vee)^{\bbeta}).
\]
\end{proposition}

\begin{proof}
Let $\lambda^\vee(\tau,\sigma) \in L_f^\times$ denote the coefficients obtained from the system of dual basis elements $\{\gamma_\sigma^\vee \mid \sigma \in J_F\}$. By construction of the comparison isomorphism $I_f$ for the motives $h^1(A^\sigma)^\alpha$ and $h^1(A^{\sigma,\vee})^\beta$, it immediately follows that 
\[
I_f(\gamma_\sigma^{\balpha} \otimes (\gamma_\sigma^\vee)^{\bbeta})^\tau = \lambda(\tau,\sigma)^\alpha \cdot \lambda^\vee(\tau,\sigma)^\beta \cdot I_f(\gamma_{\tau\sigma}^{\balpha} \otimes (\gamma_{\tau\sigma}^\vee)^{\bbeta})
\]
and it suffices to show $\lambda^\vee(\tau, \sigma) = \lambda(\tau, \sigma)^{-1} \chi_\cyc(\tau)^{-1}$.
Recall from Section \ref{section_weil_pairing} that the Weil pairings
\[
\langle \cdot, \cdot \rangle_\sigma : h^1_f(A^\sigma) \times h^1_f(A^{\vee,\sigma}) \to L_f(-1)
\]
are compatible with the action of $G_\QQ$, i.e.\ for any $\tau \in G_\QQ$, $a \in h^1_f(A^\sigma)$ and $b \in h^1_f(A^{\vee,\sigma})$, one has $\langle a, b \rangle_\sigma^\tau = \langle a^\tau, b^\tau \rangle_{\tau \sigma}$, where we use the full $G_\QQ$-action on $L_f(-1)$. It follows that 
\[
\langle \gamma_{\sigma,f}, \gamma_{\sigma,f}^\vee \rangle_\sigma^\tau 
= \lambda(\tau, \sigma) \lambda^\vee(\tau,\sigma) \cdot  \langle \gamma_{\tau \sigma,f}, \gamma_{\tau \sigma,f}^\vee \rangle_{\tau \sigma} 
\]
and on the other hand we have
\[
\langle \gamma_{\sigma,f}, \gamma_{\sigma,f}^\vee \rangle_\sigma^\tau = \chi_\cyc(\tau)^{-1} \cdot  \langle \gamma_{\sigma,f} , \gamma_{\sigma,f}^\vee \rangle_\sigma.
\]
Since the Weil pairing is of motivic origin and by definition of the dual bases $\gamma_\sigma^\vee$ in terms of the pairing in the Betti realization, the $L_f$-basis $\langle \gamma_{\sigma,f} , \gamma_{\sigma,f}^\vee \rangle_\sigma$ of $L_f(-1)$ is independent of $\sigma$.
Indeed, if $\langle \cdot, \cdot \rangle_B$ denotes the pairing in the Betti realization, we have $\langle \gamma_{\sigma,f}, \gamma_{\sigma,f}^\vee \rangle_\sigma = I_f(\langle \gamma_\sigma, \gamma_\sigma^\vee \rangle_B)$ for all $\sigma \in J_F$ and by construction, $\langle \gamma_\sigma, \gamma_\sigma^\vee \rangle_B = (2\pi i)^{-1} \in L(-1)_B$ is independent of the embedding $\sigma$.
Thus
\[
\lambda^\vee(\tau, \sigma) = \lambda(\tau, \sigma)^{-1} \chi_\cyc(\tau)^{-1},
\]
as desired.
\end{proof}

\subsection{Basis elements in $M(\Xi)^\eta_B$}
As before, let $F/E$ be a finite Galois extension and $A/F$ an abelian variety of CM-type $(L,\Phi)$ such that its CM-character $\psi$ satisfies $\Xi \circ N_{F/E} = \psi^\alpha \smash{\ol{\psi}}^\beta$. Recall that we can regard $M(\Xi)$ as a quotient of $R_{F/E}(h^1(A)^\alpha_T \otimes h^1(A^\vee)^\beta_T)$ as in (\ref{alt_constr_m_xi}). Fix a system $\{\gamma_\sigma \mid \sigma \in J_F\}$ of $L$-bases $\gamma_\sigma \in h^1_B(A^\sigma)$ and, for every $\tau \in G_\QQ$ and $\sigma \in J_F$, let 
$\lambda(\tau,\sigma) \in L_f^\times$
denote the unique elements such that
\[
I_f(\gamma_\sigma)^\tau = \lambda(\tau,\sigma) \cdot I_f(\gamma_{\tau\sigma}).
\]
Recall the definition of the elements $\xi(\tau,\eta)$ in (\ref{def_xi_elements}).
The aim of this section is to use (\ref{alt_constr_m_xi}) to explicitly construct a system of basis elements $a_\eta \in M(\Xi)^\eta_B$ for $\eta \in J_E$ satisfying
\begin{align}\label{betti_coefficients_for_a_eta}
I_f(a_\eta)^\tau = \xi(\tau,\eta) \cdot I_f(a_{\tau \eta}).
\end{align}
Essentially, we define the elements $a_\eta$ to satisfy (\ref{betti_coefficients_for_a_eta}) by design. The more involved part will then be to show that these elements are in fact non-zero.

\begin{proposition}\label{t_elements_independence}
Let $\sigma \in G_\QQ$. The element
\[
t_\sigma \coloneqq \chi(\lambda(\sigma, 1_F))^{-1} \varepsilon_\Phi(\sigma) \in T^\times
\]
only depends on the restriction $\sigma|_F$.
\end{proposition}

\begin{proof}
Two element $\sigma, \sigma' \in G_\QQ$ with the same restriction to $F$ satisfy $\sigma' = \sigma \tau$ for some $\tau \in G_F$. We have
$\lambda(\sigma \tau, 1_F) = \lambda(\sigma, 1_F) \lambda(\tau, 1_F)$
and hence
$t_{\sigma\tau} = t_\sigma t_\tau$.
It therefore suffices to show $t_\tau = 1$. Since $h^1(A)$ is of type $(F,\psi,L)$ and $\psi_a = \Phi^\ast \circ N_{F/E}$, one has
\[
\lambda(\tau,1_F) = \psi(x) \cdot \Phi^\ast(N_{F/E}(x))^{-1} ,
\]
for any $x \in F_f^\times$ with $r_F(x) = \tau|_{F^\ab}$. Hence, one obtains
\begin{align*}
t_\tau &= \chi(\lambda(\tau,1_F))^{-1} \varepsilon_\Phi(\tau) \\
&= \chi_a(\psi(x))^{-1} \chi(\Phi^\ast(N_{F/E}(x))) \varepsilon_\Phi(\tau) \\
&= \psi(x)^{\alpha-\beta} \cdot  \Xi(N_{F/E}(x))^{-1} \cdot \lVert N_{F/E}(x)\rVert_E^{-d(\beta)} \cdot \varepsilon_\Phi(N_{F/\QQ}(x)) \cdot \varepsilon_\Phi(\tau) \\
&= \psi(x)^{-\beta} \ol{\psi}(x)^{-\beta} \cdot \lVert x \rVert_F^{-d(\beta)} \\
&= 1,
\end{align*}
which proves the claim.
\end{proof}

Let us introduce the following notation: Recall (\ref{betti_realization_M_xi}). For each $\sigma \in J_F$, consider the element
\[
b(\gamma_\sigma) \coloneqq e_\Xi \cdot \gamma_\sigma^{\balpha} \otimes (\gamma_\sigma^\vee)^{\bbeta}  
\]
in $M(\Xi)^\sigma_B$.
Proposition \ref{betticoefficients1} implies that
\begin{align}\label{b_element_coefficients}
I_f(b(\gamma_\sigma))^\tau = \lambda(\tau,\sigma)^{\alpha-\beta} \chi_\cyc(\tau)^{-d(\beta)} I_f(b(\gamma_{\tau\sigma}))
\end{align}
for all $\sigma \in J_F$ and $\tau \in G_\QQ$. 

\begin{definition}\label{def_a_eta_elements}
For each $\eta \in J_E$, we define the element
\[
a_\eta \coloneqq  \sum_{\sigma \in J_{F,\eta}} t_\sigma \cdot b(\gamma_\sigma)
\]
in $M(\Xi)^\eta_B$.
\end{definition}

We can now prove the main result of this section:

\begin{theorem}\label{a_eta_elements_transformation}
For every $\eta \in J_E$, the element $a_\eta \in M(\Xi)_B^\eta$ is non-zero. 
For all $\tau \in G_\QQ$ and $\eta \in J_E$, one has
\begin{equation}\label{transformation_property}
I_f(a_\eta)^\tau = \xi(\tau,\eta) \cdot I_f(a_{\tau\eta}).
\end{equation}
\end{theorem}

\begin{proof}
Let us first show that the elements satisfy (\ref{transformation_property}). By (\ref{b_element_coefficients}), we have
\[
I_f(a_\eta)^\tau 
= \sum_{\sigma \in J_{F,\eta}} t_\sigma \cdot \lambda(\tau,\sigma)^{\alpha-\beta} \chi_\cyc(\tau)^{-d(\beta)} I_f(b(\gamma_{\tau \sigma}))
\]
and it suffices to show that the elements $t_\sigma$ satisfy
\[
t_\sigma \cdot \chi_a(\lambda(\tau,\sigma))^{-1} \chi_\cyc(\tau)^{-d(\beta} =
t_{\tau \sigma} \cdot \xi(\tau,\eta)
\]
for all $\sigma \in J_{F,\eta}$ and $\tau \in G_\QQ$. Choosing an extension $\sigma \in G_\QQ$ for every $\sigma \in J_F$, this follows from the definitions and using
$\lambda(\tau \sigma,1_F) = \lambda(\sigma,1_F) \lambda(\tau,\sigma)$ and
$\varepsilon_\Phi(\sigma) \varepsilon_{\eta \Phi}(\tau) = \varepsilon_\Phi(\tau\sigma)$.

We now prove that for any $\eta \in J_E$, the element
$a_\eta$
in $M(\Xi)^\eta_B$ is in fact non-zero.
For every $\sigma \in J_F$, we choose an extension $\sigma \in G_\QQ$ and abbreviate $\lambda_\sigma \coloneqq \lambda(\sigma,1_F) \in L_f^\times$. Recall that
$t_\sigma = \chi(\lambda_\sigma)^{-1} \varepsilon_\Phi(\sigma)$. First, we note that the claim $a_\eta \neq 0$ is independent of the choice of the basis elements $\{\gamma_\sigma\}_{\sigma \in J_{F,\eta}}$ used in the construction of $a_\eta$. Indeed, for a different choice $\{\gamma_\sigma'\}_{\sigma \in J_F}$ of the form
\[
\gamma_\sigma' = u_\sigma \cdot \gamma_\sigma 
\]
with $u_\sigma \in L^\times$, one easily shows that the corresponding elements $a_\eta'$ are given by $a_\eta' = u_{1_F}^{\alpha-\beta} \cdot a_\eta$.
Moreover, we observe that it suffices to show 
$a_{1_E}  \neq 0$,
since for arbitrary $\eta \in J_E$ and any extension $\eta \in G_\QQ$ of $\eta \in J_E$, one has
\[
I_f(a_{1_E})^\eta = \xi(\eta,1_E) \cdot I_f(a_{\eta})
\]
by design.
Hence, we may choose the system $\{\gamma_\sigma\}_\sigma$ in the following way: For every $\sigma \in J_{F/E}$ fix an extension $\sigma \in G_E$ and $e_\sigma \in E_f^\times$ such that
$r_E(e_\sigma) = \sigma|_{E^\ab}$.
Assume that $\id$ extends $1_E$ and that $e_{1_E} = 1$.
Recall the $L^\times$-coset $L(\Phi,\sigma)$ defined in (\ref{modified_coset}).
Since every extension $\sigma$ fixes the reflex field, we have
\[
L(\Phi,\sigma) = L^\times \cdot \Phi^\ast(e_\sigma)^{-1}
\]
by Proposition \ref{properties_of_cosets} (ii). Fix an arbitrary $L$-basis $\gamma_1 \in h^1(A)_B$. For every $\sigma \in J_{F/E}$, the main theorem of complex multiplication (cf.\ Theorem \ref{main_theorem_of_cm}) implies the existence of a basis vector $\gamma_\sigma \in h^1(A)^\sigma_B$ satisfying
\begin{align}\label{equation0}
\lambda_\sigma = \Phi^\ast(e_\sigma)^{-1}.
\end{align}
On the other hand, the motive $M(\Xi)$ is of type $(E,\Xi,T)$ and hence
\begin{align}\label{equation2}
I_f(b(\gamma_1))^\sigma = \Xi_a(e_\sigma)^{-1} \cdot \Xi(e_\sigma) \cdot I_f(b(\gamma_1)).
\end{align}
Recall that $\Xi_a = (\Phi^\ast)^{\alpha + c \beta}$. Using (\ref{b_element_coefficients}), (\ref{equation0}) and (\ref{equation2}), we obtain
\begin{align*}
I_f(b(\gamma_\sigma)) &= \lambda_\sigma^{\alpha - \beta} \chi_\cyc(\sigma)^{d(\beta)} \Xi_a(e_\sigma)^{-1} \Xi(e_\sigma) \cdot I_f(b(\gamma_1)) \\
&= \Phi^\ast(e_\sigma)^{\alpha - \beta} \Phi^\ast(e_\sigma)^{-\alpha - c \beta} \chi_\cyc(\sigma)^{d(\beta)} \Xi(e_\sigma) \cdot I_f(b(\gamma_1)) \\
&= N_{E/\QQ}(e_\sigma)^{-d(\beta)} \chi_\cyc(\sigma)^{d(\beta)} \Xi(e_\sigma) \cdot I_f(b(\gamma_1)) \\
&= \lVert e_\sigma \rVert_E^{d(\beta)} \cdot \Xi(e_\sigma) \cdot I_f(b(\gamma_1)),
\end{align*}
and hence
\begin{align}\label{equation3}
b(\gamma_\sigma) = \lVert e_\sigma \rVert_E^{d(\beta)} \cdot \Xi(e_\sigma) \cdot b(\gamma_1)
\end{align}
in $M(\Xi)_B$. Therefore, again by (\ref{equation0}) and by definition of the Hecke character $\Xi$ (cf.\ Definition \ref{def_hecke_character_xi}), we obtain
\begin{align*}
a_1 &= \sum_{\sigma \in J_{F/E}} t_\sigma \cdot b(\gamma_\sigma) \\
&= \sum_{\sigma \in J_{F/E}} \chi(\Phi^\ast(e_\sigma)) \varepsilon_\Phi(e_\sigma) \lVert e_\sigma \rVert_E^{d(\beta)} \cdot \Xi(e_\sigma) \cdot b(\gamma_1) \\
&= [F:E] \cdot b(\gamma_1).
\end{align*}
Finally, observe that 
\[
b(\gamma_1) \neq 0
\]
in $M(\Xi)$ as otherwise one has $b(\gamma_\sigma) = 0$ for all $\sigma \in J_{F/E}$ by (\ref{equation3}). But the elements $b(\gamma_\sigma)$ for $\sigma \in J_{F/E}$ generate $M(\Xi)_B$ as a $T$-vector space and hence one would have $M(\Xi)_B = 0$, in contradiction to $\dim_T M(\Xi)_B = 1$.
\end{proof}

The theorem provides an explicit presentation, in terms of the CM abelian variety $A$, of the basis elements used in Blasius' Theorem \ref{blasius_period_relation}, which a priori are rather implicit in nature. This is crucial in our approach: In the next sections, we will turn our attention to the de Rham realization $M(\Xi)_\dR$ and construct elements in $F^{d(\alpha)+d(\beta)} RM(\Xi)_\dR$ by using Eisenstein-Kronecker classes as defined by Kings-Sprang in \cite{kingssprang}. The ultimate goal is then to compute the period $v(\chi)$ occuring in Blasius Theorem and for this we need the explicit presentation of Definition \ref{def_a_eta_elements}.

\section{Eisenstein-Kronecker classes}\label{section_7_ek_classes}
In this section, we recall the definition of the Eisenstein-Kronecker classes and along the way we adapt some results of \cite{kingssprang} to also account for all the Galois conjugates of these classes.  
In the following section, the alternative presentation (\ref{alt_constr_m_xi}) of the reflex motive will be used to show that the Eisenstein-Kronecker classes of \cite{kingssprang} naturally can be regarded as de Rham classes of the reflex motive. 
\subsection{Review of Eisenstein-Kronecker classes}

Let us briefly review the definition of the Eisenstein-Kronecker classes by Kings-Sprang, \cite{kingssprang}. For details we refer the reader to section 2 of loc.\ cit.

Let $S$ be a Noetherian scheme and $A/S$ an abelian scheme of relative dimension $g$ with unit section $e: S \to A$. We denote the dual abelian scheme by $A^\vee/S$ and write $e^\vee$ for its unit section and $\sP$ for the Poincaré bundle on $A \times_S A^\vee$.
Moreover, let $A^\natural$ be the universal vector extension of $A^\vee/S$, let $e^\natural$ denote its unit section and let $(\sP^\natural, \nabla)$ denote the universal line bundle with integrable connection on $A \times_S A^\natural$. 
Let $\widehat{\sP}$ (resp.\ $\widehat{\sP}^\natural$) denote the formal completion of $\sP$ (resp.\ $\sP^\natural$) along $\id \times e^\vee : A \to A \times_S A^\vee$ (resp.\ $\id \times e^\natural$), regarded as sheaves on $A$. The connection $\nabla$ induces a connection on $\widehat{\sP}^\natural$, which we denote by the same letter. 

Let $\Gamma$ be a discrete group acting on $A/S$ from the left. 
This induces structures of $\Gamma$-equivariant sheaves on $\Omega^1_{A/S}$, $\widehat{\sP}$ and $\widehat{\sP}^\natural$. In the following $H^i(A,\Gamma, -)$ denotes equivariant sheaf cohomology, i.e.\ the $i$th right derived functor of $H^0(A,-)^\Gamma$, cf.\ \cite{tohoku}, Chapitre V or \cite{kingssprang}, Appendix A.

\begin{theorem}[\cite{kingssprang}, Corollary 2.17]\label{cohomology_poincare_bundle}
One has 
\[
H^i(A, \Gamma, \widehat{\sP} \otimes \Omega^g_{A/S}) \cong \begin{cases} H^0(S,\cO_S) \textrm{ if } i = g \\ 0 \textrm{ if } i < g. \end{cases}
\]
\end{theorem}

Let now $D \subset A$ denote a $\Gamma$-stable closed subscheme which arises as the kernel of an étale isogeny $\delta: A \to A'$ with étale dual. Let $U_D = A \setminus D$ be its open complement. Let
$\cO_S[D]^0 = \ker(H^0(D,\cO_D) \xto{\tr} H^0(S,\cO_S))$ be the kernel of the trace map and let $\cO_S[D]^{0,\Gamma}$ denote its $\Gamma$-invariants.
By Theorem \ref{cohomology_poincare_bundle}, the localization sequence in equivariant cohomology gives rise to an exact sequence
\begin{equation}\label{localization_sequence}
0 \to H^{g-1}(U_D, \Gamma, \widehat{\sP} \otimes \Omega^g_{A/S}) \to H^g_D(A,\Gamma, \widehat{\sP} \otimes \Omega^g_{A/S}) \to H^0(S,\cO_S).
\end{equation}

\begin{theorem}[\cite{kingssprang}, Theorem 2.20]\label{EK_map}
There exists a canonical inclusion
\[
\cO_S[D]^\Gamma \into H^g_D(A,\Gamma,\widehat{\sP} \otimes \Omega^g_{A/S})
\]
whose composite with the map $H^g_D(A,\Gamma, \widehat{\sP} \otimes \Omega^g_{A/S}) \to H^0(S,\cO_S)$ from (\ref{localization_sequence}) agrees with the trace map $\tr: \cO_S[D]^\Gamma \to H^0(S,\cO_S)$. In particular, one obtains an induced injection
\[
\EK_{\Gamma, A} : \cO_S[D]^{0,\Gamma} \into H^{g-1}(U_D, \Gamma, \widehat{\sP} \otimes \Omega^g_{A/S}).
\]
Moreover, for $f \in \cO_S[D]^{0,\Gamma}$, we let 
\[
\EK_{\Gamma,A}^\natural(f) \in H^{g-1}(U_D,\Gamma, \widehat{\sP}^\natural \otimes \Omega^g_{A/S})
\]
denote its image under the canonical map induced by $\widehat{\sP} \to \widehat{\sP}^\natural$.
We drop the abelian scheme $A$ from the notation when it is apparent from the context.
\end{theorem}

For every integer $a \geq 0$, the connection $\nabla$ on $\widehat{\sP}^\natural$ induces a map
\[
\nabla^a : \widehat{\sP}^\natural \to \TSym^a(\Omega^1_{A/S}) \otimes \widehat{\sP}^\natural
\]
and we obtained an induced class
\[
\nabla^a \EK_\Gamma^\natural(f) \in H^{g-1}(U_D, \Gamma, \TSym^a(\Omega^1_{A/S}) \otimes \widehat{\sP}^\natural \otimes \Omega^g_{A/S}).
\]
Let $b \geq 0$ be an integer. Let $x: S \to U_D$ be a $\Gamma$-stable torsion section in the kernel of an isogeny $\varphi: A \to B$ which admits an étale dual. Let us abbreviate $\sH = H^1_\dR(A^\vee/S)$. By pulling back along $x$ and applying the moment map
${}_\rho \mom^b_x : x^\ast \widehat{\sP}^\natural \to \TSym^b(\sH)$ (cf.\ \cite{kingssprang}, Definition 2.10), one obtains the class
\[
{}_\rho \mom^b_x ( x^\ast \nabla^a \EK_\Gamma^\natural(f)) \in H^{g-1}(S, \Gamma, \TSym^a(\omega_{A/S}) \otimes \TSym^b(\sH) \otimes \omega^g_{A/S}).
\]
Assume from now on that $S$ is affine. Under this assumption the equivariant cohomology spectral sequence degenerates, so that
\[
H^{g-1}(S, \Gamma, \TSym^a(\omega_{A/S}) \otimes \TSym^b(\sH) \otimes \omega^g_{A/S}) \cong H^{g-1}(\Gamma, \TSym^a(\omega_{A/S}) \otimes \TSym^b(\sH) \otimes \omega^g_{A/S})
\]
and we obtain a class
\[
\EK^{b,a}_{\Gamma}(f,x) \in H^{g-1}(\Gamma, \TSym^a(\omega_{A/S}) \otimes \TSym^b(\sH) \otimes \omega_{A/S}^g).
\]

\begin{proposition}[Compatibility under base change] \label{prop_base_change}
Let the notations be as above. Let $\sigma: T \to S$ be a morphism between affine Noetherian schemes and let $A^\sigma/T$ denote the base change of $A/S$ along $\sigma$. Then $A^\sigma$ acquires an induced action by $\Gamma$. Let $D^\sigma$ denote the base change of $D$ along $\sigma$. For every $f \in \cO_S[D]^{0,\Gamma}$ let $f^\sigma$ denote the image of $f$ under the canonical map
$\cO_S[D]^{0,\Gamma} \to \cO_S[D^\sigma]^{0,\Gamma}$ induced by $D^\sigma \to D$. 
Let $x^\sigma$ denote the torsion section of $A^\sigma/S$ induced by $x$. There is a canonical map of $\Gamma$-modules
\[
\TSym^a(\omega_{A/S}) \otimes \TSym^b(\sH_A) \otimes \omega_{A/S}^g \to \TSym^a(\omega_{A^\sigma/T}) \otimes \TSym^b(\sH_{A^\sigma}) \otimes \omega_{A^\sigma/T}^g
\]
and the induced map on group cohomology sends $\EK^{b,a}_{\Gamma,A}(f,x)$ to $\EK^{b,a}_{\Gamma,A^\sigma}(f^\sigma,x^\sigma)$.
\end{proposition}

\begin{proof}
One easily checks that every step in the construction of $\EK^{b,a}_\Gamma(f,x)$ is compatible under base change. For the most part, this essentially boils down to the fact that $(A^\sigma)^\vee \cong A^\vee \times_S T$ and that Poincaré bundle on $A^\sigma \times (A^\sigma)^\vee$ is the base change of $\sP$ along the induced map $A^\sigma \times (A^\sigma)^\vee \to A \times A^\vee$. Regarding the Eisenstein-Kronecker map of Theorem \ref{EK_map}, note that the fundamental local isomorphism in Grothendieck duality theory, which is used in the construction (cf.\ \cite{kingssprang}, section 2.6) is compatible with base change, cf.\ \cite{hartshorne}, III, 7.4(b) for flat base change and \cite{conrad}, p.\ 52 for the general case.
\end{proof}

\subsection{The case of CM abelian varieties}
Let $L$ be a totally complex number field. We now specialize to the case where $A/F$ is an abelian variety with complex multiplication by $\cO_L$ defined over a number field $F$ which contains the Galois closure of $L$. More generally, one could assume that $A$ is defined over an $\cO_{L^{\Gal}}[1/d_L]$-algebra, where $d_L$ denotes the discriminant of $L/\QQ$. Let $\Gamma \subset \cO_L^\times$ be a subgroup of finite index. The CM-structure on $A/F$ induces a left $\Gamma$-action on $A/F$ via automorphisms. 
In this way, the $F$-vector spaces $\Lie(A/F)$, $\omega_{A^\vee/F}$ and $\sH_A = H^1_\dR(A^\vee/F)$ acquire induced $\Gamma$-module structures by functoriality. On $\omega_{A/F}$ we consider the $\Gamma$-action dual to the one on $\Lie(A/F)$.
Moreover, all of the above $F$-vector spaces carry the structure of $L \otimes_\QQ F$-modules by functoriality. 
Let $A/F$ be of CM-type $(L,\Phi)$. By the assumption that $F$ contains the Galois closure of $L$, we obtain decompositions
\begin{align*}
\Lie(A/F) &= \bigoplus_{\sigma \in \Phi} \Lie(A/F)(\sigma), & \omega_{A/F} &= \bigoplus_{\sigma \in \Phi} \omega_{A/F}(\sigma), \\
\omega_{A^\vee/F} &= \bigoplus_{\ol{\sigma} \in \ol{\Phi}} \omega_{A^\vee/F}(\ol{\sigma}),  & \sH_A &= \bigoplus_{\sigma \in J_L} \sH_A(\sigma)
\end{align*}
as $L \otimes F$-modules. The group action of an element $\gamma \in \Gamma$ on $\Lie(A/F)(\sigma)$, $\omega_{A^\vee/F}(\sigma)$ and $\sH_A(\sigma)$ is given by multiplication by $\sigma(\gamma)$, whereas the action on $\omega_{A/F}(\sigma)$ is given by $\sigma(\gamma)^{-1}$. Note that our notation slightly deviates from the one used in \cite{kingssprang}.
It follows from the above decompositions that the Hodge filtration
\begin{equation}\label{hodge_filtration}
0 \to \omega_{A^\vee/F} \to \sH_A \to \Lie(A/F) \to 0
\end{equation}
canonically splits.

Let $a, b \geq 0$ be non-negative integers. The splitting of (\ref{hodge_filtration}) provides a projection 
\begin{equation}\label{hodge_projection}
\TSym^b(\sH_A) \to \TSym^b(\omega_{A^\vee/F}).
\end{equation}
Recall the notation introduced in (\ref{sym_mu_notation}) for $\TSym$. For every $\alpha, \beta \in I_L^+$, the group $\Gamma$ acts on 
$\TSym^\alpha(\omega_{A/F}) \otimes \TSym^\beta(\omega_{A^\vee/F}) \otimes \omega_{A/F}^g$ via $\beta - \alpha - 1_\Phi$. Therefore, the $\Gamma$-invariants of $\TSym^a(\omega_{A/F}) \otimes \TSym^b(\omega_{A^\vee/F}) \otimes \omega_{A/F}^g$ decompose into the direct sum over all 
\begin{equation}\label{alpha_beta_component}
\TSym^\alpha(\omega_{A/F}) \otimes \TSym^\beta(\omega_{A^\vee/F}) \otimes \omega_{A/F}^g,
\end{equation}
where $\beta - \alpha - 1_\Phi$ is of Hecke character type with respect to $\Gamma$ and critical.
In particular, for critical $\beta - \alpha - 1_\Phi$, each $\TSym^\alpha(\omega_{A/F}) \otimes \TSym^\beta(\omega_{A^\vee/F}) \otimes \omega_{A/F}^g$ is a $\Gamma$-stable direct summand of $\TSym^a(\omega_{A/F}) \otimes \TSym^b(\omega_{A^\vee/F}) \otimes \omega_{A/F}^g$.

\begin{proposition}[\cite{kingssprang}, Proposition 2.27]\label{fundamental_class_prop}
For $\beta - \alpha - 1_\Phi \in \Crit_L(\Gamma)$, there is a canonical homomorphism
\begin{align*}
&H^{g-1}(\Gamma, \TSym^\alpha(\omega_{A/F}) \otimes \TSym^\beta(\omega_{A^\vee/F}) \otimes \omega_{A/F}^g) \to \\
&\TSym^{\alpha+1_\Phi}(\omega_{A/F}) \otimes \TSym^\beta(\omega_{A^\vee/F}).
\end{align*}
\end{proposition}

\begin{proof}
Let us review the construction in the case where $\Gamma$ is a free subgroup of finite index in $\cO_L^\times$. The general case is easily reduced to this case. Note that $\Gamma$ is of rank $g-1$ by Dirichlet's unit theorem.
Since the action of $\Gamma$ on $\TSym^\alpha(\omega_{A/F}) \otimes \TSym^\beta(\omega_{A^\vee/F}) \otimes \omega_{A/F}^g$ is trivial, it suffices by the universal coefficient theorem to construct a canonical homomorphism
\[
\can: H^{d-1}(\Gamma,\ZZ) \otimes_\ZZ \omega_{A/F}^g \to \TSym^{1_\Phi}(\omega_{A/F})
\]
Choose an \emph{ordering of the CM-type $\Phi$}. This induces an orientation on 
the norm $1$ hypersurface 
\[
L^1_{\RR,\Phi} \coloneqq \Bigl\{(r_\tau)_{\tau \in \Phi} \in \RR_{>0}^\Phi \mid \prod_{\tau \in \Phi} r_\tau = 1\Bigr\} \subset \RR_{>0}^\Phi
\]
on which the group $\Gamma$ acts freely via 
$
\gamma \cdot (r_\tau)_{\tau \in \Phi} = (|\tau(\gamma)|^2 r_\tau)_{\tau \in \Phi}$.
This gives rise to an isomorphism 
\begin{align}\label{topological_fund_class_via_ordering}
H^{g-1}(\Gamma,\ZZ) = H^{g-1}(\Gamma \bs L^1_{\RR,\Phi}, \ZZ) \cong \ZZ.
\end{align}
The choice of ordering on $\Phi$ also determines an isomorphism
\begin{align}\label{alternating_equals_symmetric_ordering}
\omega_{A/F}^g \cong \TSym^{1_\Phi}(\omega_{A/F})
\end{align}
whose inverse sends a basis element $\bigotimes_{\tau \in \Phi} \omega(\tau)$ for $0 \neq \omega(\tau) \in \omega_{A/F}(\tau)$ to $\bigwedge_{\tau \in \Phi} \omega(\tau)$ with respect to the chosen ordering.
The isomorphisms (\ref{topological_fund_class_via_ordering}) and (\ref{alternating_equals_symmetric_ordering}) together then give rise to an isomorphism
\[
\can: H^{g-1}(\Gamma,\ZZ) \otimes \omega_{A/F}^g \xto{\sim} \TSym^{1_\Phi}(\omega_{A/F}) 
\]
which is \emph{independent of the choice of the ordering of $\Phi$}.
\end{proof}

\begin{definition}\label{final_ek_class_def}
Let $A/F$ be an abelian variety with CM by $\cO_L$ of CM-type $\Phi$. Let $\beta - \alpha - 1_\Phi \in \Crit_L(\Gamma)$ be a critical Hecke character type and put $a = d(\alpha)$, $b = d(\beta)$. Let $\frf$ and $\frc$ be coprime integral ideals and let $D \coloneqq A[\frc]$ and $x \in U_D(F)$ a $\frf$-torsion section. Let $\Gamma \subset \cO_L^\times$ be a finite index subgroup fixing $x$. Let $\frf \in F[D]^{0,\Gamma}$. Consider the projection of $\EK_\Gamma^{b,a}(f,x)$ to (\ref{alpha_beta_component}) via the splitting of the Hodge filtration. We denote its image under the canonical homomorphism of Proposition \ref{fundamental_class_prop} by
\[
\EK_{A,\Gamma}^{\beta, \alpha}(f,x) \in \TSym^{\alpha + 1_\Phi}(\omega_{A/F}) \otimes_F \TSym^{\beta}(\omega_{A^\vee/F}).
\]
\end{definition}

In the following Proposition we analyse the behaviour of the classes $\EK_{A,\Gamma}^{\beta, \alpha}(f,x)$ under conjugation by elements in the absolute Galois group $G_\QQ$.

\begin{proposition}\label{SignLemma}
Assume the notations of Definition \ref{final_ek_class_def}. For any $\sigma \in G_\QQ$, one has
\[
\EK_{A,\Gamma}^{\beta, \alpha}(f, x)^\sigma = \varepsilon_\Phi(\sigma) \EK_{A^\sigma, \Gamma}^{\sigma\beta, \sigma\alpha}(f^\sigma, x^\sigma).
\]
\end{proposition}

\begin{proof}
Let $a = d(\alpha)$ and $b = d(\beta)$. The $(\alpha, \beta)$-component of $\EK^{b,a}_\Gamma(f,x)$ lies in 
\[
H^{g-1}(\Gamma, \Sym^{\alpha}(\omega_{A/F}) \otimes_F \Sym^\beta(\omega_{A^\vee/F}) \otimes \omega_{A/F}^g).
\] 
By Proposition \ref{prop_base_change}, we have
$\EK^{b,a}_{A,\Gamma}(f,x)^\sigma = \EK^{b,a}_{A^\sigma,\Gamma}(f^\sigma,x^\sigma)$.
Under conjugation by $\sigma$, the projection to the $(\alpha,\beta)$-component (\ref{alpha_beta_component}) is carried to the $(\sigma\alpha, \sigma\beta)$-component. It remains to understand the behaviour of the map in Proposition \ref{fundamental_class_prop} under conjugation. By construction,
we may assume that $\Gamma \subset \cO_L^\times$ is free of finite index and we have to show that the diagram
\[
\begin{tikzcd}
H^{g-1}(\Gamma, \ZZ) \otimes \omega_{A/F}^g \arrow{r}{\can_A} \ar{d}[swap]{\id \otimes \sigma} & \Sym^{1_\Phi}(\omega_{A/F}) \arrow{d}{\sigma} \\
H^{g-1}(\Gamma, \ZZ) \otimes \omega_{A^\sigma/F}^g \arrow{r}{\can_{A^\sigma}} & \Sym^{1_{\sigma\Phi}}(\omega_{A^\sigma/F})
\end{tikzcd}
\]
commutes up to the sign $\varepsilon_\Phi(\sigma)$. Choose an ordering of the set of complex places $\langle c \rangle \backslash J_L$. This induces an ordering of 
every CM-type $\Phi'$ of $L$, in particular on $\Phi$ and $\sigma \Phi$. We get induced orientations on 
$L^1_{\RR, \Phi}$ and $L^1_{\RR, \sigma\Phi}$. The canonical map
$L^1_{\RR,\Phi} \to L^1_{\RR, \sigma\Phi}$
which maps $(r_\tau)_{\tau \in \Phi}$ to the family given by
\[
\tau' \mapsto \begin{cases} r_{\tau'} &\textrm{ if } \tau' \in \tau\Phi \cap \Phi \\ r_{\ol{\tau}'} &\textrm{ if } \tau' \in \tau\Phi \cap \ol{\Phi} \end{cases} 
\]
for $\tau \in \sigma\Phi$ is a $\Gamma$-equivariant isomorphism. Using the chosen orderings on $\Phi$ and $\sigma\Phi$, this isomorphism identifies with the identity map on $(\RR_{>0}^g)^{N=1}$, and in
particular, the two induced isomorphisms
$H^{g-1}(\Gamma,\ZZ) \cong \ZZ$
coincide.
On the other hand, we observe that with respect to our orderings the diagram
\[
\begin{tikzcd}
\omega_{A/F}^g \arrow{r} \arrow{d}[swap]{\sigma} & \Sym^{1_\Phi}(\omega_{A/F}) \arrow{d}{\sigma} \\
\omega_{A^\sigma/F}^g \arrow{r} & \Sym^{1_{\sigma\Phi}}(\omega_{A^\sigma/F})
\end{tikzcd}
\]
commutes up to the sign $\varepsilon_\Phi(\sigma)$. 
This proves the claim.
\end{proof}

\subsection{Eisenstein-Kronecker series}
Let us first introduce some notations that will be used in the next sections:
\begin{notation}\label{zNotation}
For an element $z = (z(\sigma))_{\sigma \in J_L} \in \CC^{J_L}$ we write 
$\ol{z} \coloneqq (\ol{z(\ol{\sigma})})_{\sigma \in J_L}$.
For any $\alpha \in I^+_L$, we write 
$z^\alpha \coloneqq \prod_{\sigma \in J_L} z(\sigma)^{\alpha(\sigma)}$, cf.\ Notation \ref{OmegaNotation}. 
Once we have fixed a CM-type $\Phi$ of $L$ and are given an element $z = (z(\sigma))_{\sigma \in \Phi} \in \CC^{\Phi}$, we write 
$Nz \coloneqq \prod_{\sigma \in \Phi} |z(\sigma)|^2$.
\end{notation}

Let $\omega(A) \in \omega_{A/\CC}$ be a fixed $L \otimes \CC$-generator. This induces an $L \otimes \CC$-linear isomorphism $\Lie(A/\CC) \cong \CC^\Phi$, where $L$ acts diagonally via the embeddings of $\Phi$ on $\CC^\Phi$, and via the exponential sequence
\[
0 \to H_1(A(\CC),\ZZ) \to \Lie(A/\CC) \xto{\exp} A(\CC) \to 0,
\]
the choice of $\omega(A)$ gives rise to a complex uniformization
\[
\theta: \CC^\Phi / \Lambda \xto{\sim} A(\CC).
\]

\begin{definition}[\cite{kingssprang}, Definition 3.24]
Let $\Gamma \subset \cO_L^\times$ be a finite index subgroup and let $\beta - \alpha \in \Crit_L(\Gamma)$ be a Hecke character type with respect to $\Gamma$ which is critical with respect to $\Phi$. For any element $t \in \Lambda \otimes \QQ$ and $s \in \CC$ with $\Re(s) > \frac{d(\beta)-d(\alpha)}{2} + g$ define
\[
E^{\beta,\alpha}(t, s; \Lambda, \Gamma) \coloneqq \sideset{}{'}\sum_{\lambda \in \Gamma \bs (\Lambda + \Gamma t)} \frac{\ol{\lambda}^\beta}{\lambda^\alpha N(\lambda)^s}.
\]
Here the summation ranges over all non-zero $\Gamma$-cosets of $\Lambda + \Gamma t$. 
\end{definition}

For any integral ideal $\frc \subset \cO_L$, the complex uniformization induces an isomorphism
\[
A[\frc](\CC) \cong \frc^{-1} \Lambda / \Lambda.
\]
Via the Hodge decomposition $H_1(A(\CC),\CC) \cong \omega_{A^\vee/\CC} \oplus \ol{\omega_{A^\vee/\CC}}$, the period pairing $H^1_\dR(A/\CC) \times H_1(A(\CC),\CC) \to L \otimes \CC$ induces a pairing
\[
\langle \cdot, \cdot \rangle : \ol{\omega_{A/\CC}} \times \omega_{A^\vee/\CC} \to \CC^{\ol{\Phi}}.
\]
If $\varphi: A \to B$ is an isogeny, one easily shows that $\varphi^\ast$ and $\varphi_{\#}$ are adjoint to each other with respect to this pairing.

The following result, which expresses the classes defined in Definition \ref{final_ek_class_def} in analytic terms, is the key to relating these classes to special values of Hecke $L$-functions. We first need to introduce some notation: For any $\omega \in \omega_{A/F}$ and $\alpha \in I_L^+$, we write
\[
\omega^{[\alpha]} = \bigotimes_{\sigma \in J_L} \omega(\sigma)^{[\alpha(\sigma)]},
\]
if $\omega = (\omega(\sigma))_{\sigma \in J_L}$ is the decomposition according to $\omega_{A/F} = \bigoplus_{\sigma \in \Phi} \omega_{A/F}(\sigma)$.

\begin{theorem}[\cite{kingssprang}, Corollary 3.28]\label{analytic_ek_classes}
Let $\beta - \alpha \in \Crit_L(\Gamma)$. Let $\frf$ and $\frc$ be coprime integral ideals in $\cO_L$. Fix $L \otimes \CC$-generators $\omega(A)$ and $\omega(A^\vee)$ of $\omega_{A/\CC}$ and $\omega_{A^\vee/\CC}$. Let $\Lambda$ be the period lattice associated to $\omega(A)$. Let $x \in \frf^{-1} \Lambda/\Lambda$ and let $f: \frc^{-1} \Lambda/\Lambda \to \CC$ be a $\Gamma$-invariant function such that $\sum_{t \in \frc^{-1}\Lambda/\Lambda} f(t) = 0$. 
Then there is an equality
\[
\EK^{\beta, \alpha-1_\Phi}_{A,\Gamma}(f,x) =   \frac{(-1)^{\frac{g(g-1)}{2}}(\alpha-1)!}{\langle \ol{\omega(A)}, \omega(A^\vee)\rangle^\beta_A} \sum_{t \in \Gamma \bs (\frc^{-1} \Lambda / \Lambda)} f(-t) E^{\beta, \alpha}(t + x, 0; \Lambda, \Gamma) \cdot \omega(A)^{[\alpha]} \otimes \omega(A^\vee)^{[\beta]}
\]
in $\TSym^{\alpha}(\omega_{A/\CC}) \otimes \TSym^\beta(\omega_{A^\vee/\CC})$. 
\end{theorem}

Finally, let us recall the relation between Eisenstein-Kronecker series and Hecke $L$-functions. Let $L$ be a totally imaginary number field and $\Phi$ a CM-type of $L$. 
Let $\chi: L_f^\times \to T^\times$ be an algebraic Hecke character whose infinity-type $\chi_a = \beta- \alpha$ is critical with respect to $\Phi$. 
Let $\frf \neq \cO_L$ be an integral ideal such that $\chi$ can be regarded as a having conductor dividing $\frf$. In the following, we let $\cO_f^\times$ denote the finite index subgroup of $\cO_L^\times$ consisting of units which are congruent to $1$ modulo $\frf$. We will from now on always consider $\Gamma = \cO_\frf^\times$ and drop it from the notation.
Fix an embedding $\iota \in J_T$ and set $\chi_\iota \coloneqq \iota \circ \chi$. Then, for every integral ideal $\frb$ coprime to $\frf$ with ray class $[\frb] \in I^\frf/P^\frf$, there is a bijection
\[
\cO_\frf^\times \bs (1 + \frf \frb^{-1}) \xto{\sim} \{ \fra \in [\frb] \mid \fra \textrm{ integral }\}, \;\;\; y \mapsto y \frb.
\]
Using this bijection, it follows that (cf.\ \cite{kingssprang}, Lemma 4.2)
\begin{align*}
L_\frf(\chi_\iota,s,[\frb]) 
&= \chi_\iota(\frb) N\frb^{-s} E^{\iota\beta,\iota\alpha}(1,s;\frf \frb^{-1}, \cO_\frf^\times).
\end{align*}
and evaluating at $s = 0$ gives
\begin{equation}\label{ek_classes_and_l_values_equation}
L_\frf(\chi_\iota,0,[\frb]) = \chi_\iota(\frb) E^{\iota\beta,\iota\alpha}(1,0;\frf \frb^{-1}).
\end{equation}

\subsection{Eisenstein-Kronecker classes and special values of Hecke $L$-functions}\label{section_ek_classes_and_special_values}
Recall the following periods of the abelian variety $A$ and its dual $A^\vee$:

\begin{definition}\label{def_omega_periods}
Let $A/F$ be of CM-type $(L,\Phi)$. In the following, we identify $L \otimes \CC \cong \CC^{J_L}$. Let $\omega(A) \in \omega_{A/F}$ and $\omega(A^\vee) \in \omega_{A^\vee/F}$ denote fixed $L \otimes F$-generators. Moreover, let $\gamma$ be an $L$-basis of $H^1(A(\CC),\QQ)$ and let $\gamma^\vee$ denote its dual basis with respect to the pairing $H^1(A(\CC),\QQ) \otimes_L H^1(A^\vee(\CC),\QQ) \to (2 \pi i)^{-1} L$. In this situation, we denote by $\Omega \in \CC^{\Phi}$ and $\Omega^\vee \in \CC^{\ol{\Phi}}$ the unique elements such that
\[
I_{\infty,A}(\omega(A)) = \Omega \cdot \gamma \;\;\; \textrm{ and } \;\;\; I_{\infty,A^\vee}(\omega(A^\vee)) = \Omega^\vee \cdot \gamma^\vee.
\]
\end{definition}

\begin{remark}
Note that the period $\Omega$ defined above coincides with period defined in (\ref{omega_period}) for the motive $M = h^1(A)$ and by regarding $\omega(A)$ in $h^1_\dR(A)$ via the Hodge filtration $\omega_{A/F} \subset h^1_\dR(A)$. Similarly, $\Omega^\vee$ is associated to the motive $h^1(A^\vee)$.
\end{remark}

Let $E$ denote the reflex field of $(L,\Phi)$. Let $F/E$ be a finite extension such that $F$ contains the Galois closure of $L$ and let $A/F$ be an abelian variety of CM-type $(L,\Phi)$ whose CM-order is $\cO_L$ and we assume that all $[\frf]$-torsion points are defined over $F$. Moreover, we assume that $H_1(A(\CC),\ZZ)$ is a free $\cO_L$-module and we fix an $\cO_L$-linear isomorphism
\begin{equation}\label{ol_structure}
\xi : \cO_L \xrightarrow{\sim} H_1(A(\CC),\ZZ).
\end{equation}
Note that this condition can always be achieved by applying Serre's tensor construction.

Let $\gamma \in h^1_B(A)$ denote the dual of the $L$-basis $\xi(1) \in H_1(A(\CC),\QQ)$. Let $\gamma^\vee \in h^1_B(A^\vee)$ be the dual basis of $\gamma$ with respect to the canonical pairing $h^1_B(A) \otimes_L h^1_B(A^\vee) \to (2\pi i)^{-1} L$.
Let $\omega(A)$ and $\omega(A^\vee)$ denote $L \otimes F$-generators of $\omega_{A/F}$ and $\omega_{A^\vee/F}$ (cf.\ \cite{kingssprang}, Definition 1.15). Let $\Omega = \langle \omega(A),\gamma \rangle \in \CC^{\Phi}$ and $\Omega^\vee = \langle \omega(A^\vee), \gamma^\vee \rangle \in \CC^{\ol{\Phi}}$ denote the periods defined Definition \ref{def_omega_periods}.
For any integral ideal $\frb$ which is coprime to $\frf$,
the abelian variety $A(\frf \frb^{-1}) \coloneqq \frf \frb^{-1} \otimes_{\cO_L} A$ still has complex mulitplication by $\cO_L$. By \cite{kingssprang}, Proposition 4.7, the differential form 
$\frf \frb^{-1} \otimes \omega(A) \coloneqq [\frf]^\ast ([\frb]^\ast)^{-1} \omega(A)$
on $A(\frf \frb^{-1})$ generates $\omega_{A(\frf \frb^{-1})/F}$ and induces a complex uniformization of $A(\frf \frb^{-1})$ such that the $\frf$-torsion points are given by
\[
A(\frf \frb^{-1})(\CC)[\frf] = \frb^{-1} \Omega / \frf \frb^{-1} \Omega.
\] 
\begin{definition}\label{x_b_torsion_point}
Let $x_{\frb} \in A(\frf \frb^{-1})[\frf]$ denote the $\frf$-torsion point of $A(\frf \frb^{-1})$ which corresponds to the point
\[
\Omega + \frf \frb^{-1}\Omega  \in  A(\frf \frb^{-1})(\CC)[\frf].
\] 
By our assumptions on $A/F$, the torsion point $x_{\frb}$ is defined over $F$. Note that $x_\frb$ depends on the choice of the isomorphism (\ref{ol_structure}) but not on the choice of $\omega(A)$.
\end{definition}

To the above data we now can associated Eisenstein-Kronecker classes on $A(\frf \frb^{-1})$. We turn these into classes on $A$ in the following way: 

\begin{definition}\label{def_modified_ek_class}
For critical $\beta - \alpha \in I_L$, we define
\[
\tEK_{A}^{\beta, \alpha-1_\Phi}(f, x_\frb) \in \Sym^{\alpha}(\omega_{A/F}) \otimes_F \Sym^{\beta}(\omega_{A^\vee/F})
\]
to be the image of $\EK_{A(\frf \frb^{-1})}^{\beta, \alpha-1_\Phi}(f, x_\frb)$ under the isomorphism
$\TSym^{\alpha}([\frb]^\ast ([\frf]^\ast)^{-1} ) \TSym^{\beta}([\frb]_\#^{-1} [\frf]_\#)$
followed by the canonical identification
$\TSym^{\alpha}(\omega_{A/F}) \otimes_F \TSym^{\beta}(\omega_{A^\vee/F}) \cong \Sym^{\alpha}(\omega_{A/F}) \otimes_F \Sym^{\beta}(\omega_{A^\vee/F})$ via (\ref{TSymSym}).
\end{definition}

While the notation is slightly ambiguous, we hope that it will always be clear from the context on which abelian variety the element $f$ and the torsion point $x_\frb$ are considered.
We note that Proposition \ref{SignLemma} implies that $\tEK_A^{\beta,\alpha}(f,x_\frb)^\sigma = \varepsilon_\Phi(\sigma) \cdot \tEK_{A^\sigma}^{\sigma\beta,\sigma\alpha}(f^\sigma,x^\sigma_\frb)$ via the canonical isomorphism $A(\frf \frb^{-1})^\sigma \cong A^\sigma(\frf \frb^{-1})$.

In our applications, we will make use of specific choices of elements $f \in F[D]^{0,\Gamma}$ in the definition of Eisenstein-Kronecker classes: Let $\frc \subset \cO_L$ be an integral ideal such that all $[\frc]$-torsion points of $A$ are defined over $F$. Then  $F[D] \cong \Map(A[\frc](F), F)$ and $F[D]^{0,\Gamma}$ consists of $\Gamma$-invariant maps $f: A[\frc](F) \to F$ satisfying $\sum_{t \in A[\frc](F)} f(t) = 0$.

\begin{definition}\label{f_c_element}
Let $\frc \subset \cO_L$ be an integral ideal and set $D \coloneqq \ker([\frc])$. Let $\Gamma \subset \cO_L^\times$ be a finite index subgroup. Assume that all $[\frc]$-torsion points of $A$ are defined over $F$. We define the element $f_{[\frc]} \in F[D]^{0,\Gamma}$ as
\[
f_{[\frc]} \coloneqq N\frc \cdot 1_{e} - 1_{\ker{[\frc]}},
\]
where $1_X$ denotes the characteristic function of a $\Gamma$-invariant subset $X \subset \ker([\frc])$ in $F[D]^{\Gamma}$.
\end{definition}

The following theorem establishes the connection between Eisenstein-Kronecker classes and special values of Hecke $L$-functions. It is essentially a reformulation of the results in \cite{kingssprang}, 4.4 in our setup. 
Recall that we have fixed an isomorphism $\xi$ as in (\ref{ol_structure}), which determined elements $\gamma \in h^1_B(A)$ and $h^1_B(A^\vee)$. 
Now let $\sigma \in J_F$ and choose a fixed extension in $G_\QQ$, which we also denote by $\sigma$. Similarly, we fix the following data corresponding to the abelian variety $A^\sigma$: 
\begin{itemize} 
\item Let $\gamma_\sigma \in h^1_B(A^\sigma)$ be an $L$-basis with dual $\gamma_\sigma^\vee \in h^1_B(A^{\sigma,\vee})$. 
\item Let $\lambda_\sigma \in L_f^\times$ be the unique idele such that $I_f(\gamma)^\sigma = \lambda_\sigma \cdot I_f(\gamma_{\sigma})$.
\item Let $\ell_\sigma \in L^\times$ such that
$\ell_{\sigma,v} \lambda_{\sigma,v} \equiv 1 \mod \frf_v$ for every finite place $v$ of $L$. The existence of $\ell_\sigma$ is guaranteed by the primary decomposition of $L/\frf$, cf.\ the discussion after Theorem \ref{main_theorem_uniformizations}. Note that then $\ell_{\sigma,v} \lambda_{\sigma,v} \equiv 1 \mod (\frf \frb^{-1})_v$ for every integral ideal $\frb$ coprime to $\frf$.
\item Let $\omega(A^\sigma)$ and $\omega(A^{\vee,\sigma})$ be $L \otimes F$-generators of $\omega_{A^\sigma/F}$ and $\omega_{A^{\sigma,\vee}/F}$, respectively.
\item Let $\Omega_\sigma \in \CC^{\sigma\Phi}$ and $\Omega_\sigma^\vee \in \CC^{\ol{\sigma\Phi}}$ denote the periods associated to $\gamma_\sigma$, $\gamma_\sigma^\vee$ and $\omega(A^\sigma)$, $\omega(A^{\sigma,\vee})$.
\end{itemize} 

\begin{theorem}\label{LValueProposition}
Let $\frf \neq \cO_L$, $\frb$ and $\frc$ be pairwise coprime integral ideals of $L$. Let $A/F$ be as above and moreover assume that all $[\frf]$- and $[\frc]$-torsion points of $A$ are defined over $F$.
Let $\chi$ an algebraic Hecke character of conductor dividing $\frf$ whose infinity-type $\beta - \alpha$ is critical with respect to $\Phi$. Let $x_\frb \in A(\frf \frb^{-1})[\frf](F)$ be the torsion point defined in Definition \ref{x_b_torsion_point}. Let $D = A(\frf \frb^{-1})[\frc]$ and let $f_{[\frc]} \in F[D]^{0,\Gamma}$ be defined as in Definition \ref{f_c_element}.
Let $\iota \in J_{T,\sigma|_E}$ and assume that we are given the data as listed above.
Then
\[
\begin{split}
&\chi_\iota(\frb) \cdot \tEK_{A^\sigma}^{\iota \beta, \iota \alpha - 1_{\sigma\Phi}}(f_{[\frc]}, x_\frb^\sigma) = \\
(-1)^{\frac{g(g-1)}{2}} \frac{(\alpha-1)! (2\pi i)^{d(\beta)}}{\Omega_\sigma^{\iota \alpha} (\Omega_\sigma^\vee)^{\iota \beta}} &\chi_\iota(\lambda_\sigma)^{-1}  
\left( N\frc L_\frf(\chi_\iota, 0, [\ell_\sigma \lambda_\sigma \frb]) - \chi_\iota(\frc)^{-1} L_\frf(\chi_\iota, 0, [\ell_\sigma \lambda_\sigma \frb \frc]) \right) \cdot \\
&\omega(A^\sigma)^{\iota \alpha} \otimes \omega(A^{\vee,\sigma})^{\iota \beta}.
\end{split}
\]
In particular, the element on the left hand side only depends on the class $[\frb] \in I^\frf/P^\frf$.
\end{theorem}

\begin{proof}
Fix $\sigma \in G_\QQ$ and simply denote $\ell_\sigma$ and $\lambda_\sigma$ by $\ell$ and $\lambda$. By our assumptions on $A$ and choice of $\gamma$, there exists a uniformization
\[
\theta: \CC^\Phi/\cO_L \xrightarrow{\sim} A(\CC),
\]
which induces the isomorphism $L \xrightarrow{\sim} H^1(A(\CC),\QQ)$, $1 \mapsto \gamma$.
There exists a complex uniformization
\[
\theta_\sigma : \CC^{\sigma\Phi}/[\lambda]^{-1}_L \xrightarrow{\sim} A^\sigma(\CC),
\]
inducing the isomorphism $L \xrightarrow{\sim} H^1(A^\sigma(\CC),\QQ)$, $1 \mapsto \gamma_\sigma$
on rational cohomology, such that the diagram
\[
\begin{tikzcd}
L/\cO_L \arrow{r}{\theta} \arrow{d}{\lambda^{-1}} & A_\tor \arrow{d}{\sigma}  \\
L/[\lambda]^{-1}_L \arrow{r}{\theta_\sigma} & A^\sigma_\tor
\end{tikzcd}
\]
is commutative. We also obtain an induced commutative diagram
\[
\begin{tikzcd}
L/\frf \frb^{-1} \arrow{r}{\theta} \arrow{d}{\lambda^{-1}} & A(\frf \frb^{-1})_\tor \arrow{d}{\sigma}  \\
L/\frf \frb^{-1} [\lambda]^{-1}_L \arrow{r}{\theta_\sigma} & A^\sigma(\frf \frb^{-1})_\tor.
\end{tikzcd}
\]
By definition, the torsion point $x_\frb$ corresponds to the point $1 + \frf \frb^{-1} \in L/\frf \frb^{-1}$.
Thus, the point $x_\frb^\sigma$ corresponds to 
\[
\ell + \frf \frb^{-1} [\lambda]^{-1}_L \in L/\frf \frb^{-1} [\lambda]^{-1}_L.
\]
Note that, by our assumptions on $\ell$, the ideal $\ell \lambda \frb$ is an integral ideal coprime to $\frf$.
The equality (\ref{ek_classes_and_l_values_equation}) with $\ell \lambda \frb$ in place of $\frb$ gives 
\begin{align}\label{ek_series_L_value}
\chi_\iota(\ell \lambda \frb)^{-1} \cdot L_\frf(\chi_\iota, 0 , [\ell \lambda \frb]) = E^{\iota \beta, \iota \alpha}(1,0;\frf(\ell \lambda\frb)^{-1}).
\end{align}
We set
\begin{align*}
\omega(A^\sigma(\frf\frb^{-1})) \coloneqq \ell^{-1} \cdot [\frf]^\ast ([\frb]^\ast)^{-1} \omega(A^\sigma) \\
\omega(A^\sigma(\frf \frb^{-1})^\vee) \coloneqq  [\frf]^{-1}_\# [\frb]_\# \omega(A^{\sigma,\vee})
\end{align*}
and abbreviate $\fra \coloneqq \frf (\ell \lambda \frb)^{-1}$.
By the above and \cite{kingssprang}, Proposition 4.7, the period lattice of $A^\sigma(\frf \frb^{-1})$ associated with $\omega(A^\sigma(\frf \frb^{-1}))$ is then given by 
$\fra \Omega_\sigma$,
and the $\frf$-torsion point $x_\frb^\sigma$ corresponds to the point $\Omega_\sigma + \fra \Omega_\sigma$ in 
$A^\sigma(\frf\frb^{-1})(\CC)[\frf] = \frf^{-1} \fra \Omega_\sigma / \fra \Omega_\sigma$.
The assertion of the Proposition now follows, just as in the proof of \cite{kingssprang}, Theorem 4.9, from Theorem \ref{analytic_ek_classes}. For the convenience of the reader, let us recall the argument in our setting:
Applying Theorem \ref{analytic_ek_classes} to the abelian variety $A^\sigma(\frf \frb^{-1})$ and the critical infinity-type $\iota \beta - \iota \alpha$, we obtain that
\begin{align}\label{ks_corollary_in_action}
\begin{split}
\EK^{\iota \beta, \iota\alpha-1}_{A^\sigma(\frf \frb^{-1})}(f_{[\frc]}, x_\frb^\sigma) &= \frac{(-1)^{\frac{g(g-1)}{2}} (\alpha-1)!}{\langle \ol{\omega(A^\sigma(\frf \frb^{-1}))}, \omega(A^\sigma(\frf \frb^{-1})^\vee)\rangle^{\iota \beta}} \cdot \\
&\sum_{t \in \cO_\frf^\times \bs \frc^{-1} \fra \Omega_\sigma / \fra \Omega_\sigma} f_{[\frc]}(-t) E^{\iota \beta, \iota \alpha}(t+\Omega_\sigma, 0 ;\fra \Omega_\sigma) \cdot \\
&\omega(A^\sigma(\frf \frb^{-1}))^{[\iota \alpha]} \otimes \omega(A^\sigma(\frf \frb^{-1})^\vee)^{[\iota \beta]}
 \end{split}
\end{align}
in $\TSym^{\iota \alpha}(\omega_{A^\sigma(\frf \frb^{-1})/F}) \otimes \TSym^{\iota \beta}(\omega_{A^\sigma(\frf \frb^{-1})^\vee/F})$. 
By the compatibility of the pairing $\langle \cdot, \cdot \rangle_A$ with isogenies, we have
\[
\langle \ol{\omega(A^\sigma(\frf \frb^{-1}))}, \omega(A^\sigma(\frf \frb^{-1})^\vee)\rangle^{\iota \beta} = 
\ell^{-\iota \beta} \cdot \langle \ol{\omega(A^\sigma)}, \omega(A^{\sigma,\vee})\rangle^{\iota \beta}.
\]
Moreover, we have
\[
\Omega_\sigma^\vee = \frac{2\pi i \langle \ol{\omega(A^\sigma)},\omega(A^{\sigma,\vee})\rangle_A}{\ol{\Omega}_\sigma}
\]
by \cite{kingssprang}, Proposition 4.6 (note that the proof doesn't actually require the Betti element $\gamma_\sigma$ to come from an $\cO_L$-basis), so that 
\begin{align}\label{bracket_term}
\langle \ol{\omega(A^\sigma(\frf \frb^{-1}))}, \omega(A^\sigma(\frf \frb^{-1})^\vee)\rangle^{\iota \beta} = \ell^{-\iota \beta} \cdot \frac{\ol{\Omega}_\sigma^{\iota \beta} \Omega_\sigma^{\vee,\iota \beta} }{(2\pi i)^{d(\beta)}}.
\end{align}
Moreover, it follows from the definition of the Eisenstein-Kronecker series that
\begin{align}\label{ek_series_trivial_identity}
E^{\iota \beta, \iota \alpha}((t + 1) \Omega_\sigma, 0; \fra \Omega_\sigma) = \frac{\ol{\Omega}_\sigma^{\iota \beta}}{\Omega_\sigma^{\iota \alpha}} \cdot E^{\iota \beta, \iota \alpha}(t + 1,0; \fra).
\end{align}
for all $t \in \cO_\frf^\times \bs \frc^{-1} \fra / \fra$.
The distribution relation \cite{kingssprang}, Proposition 4.12 gives
\begin{align}\label{distribution_relation_in_action}
\begin{split}
\sum_{t \in \cO_\frf^\times \bs \frc^{-1} \fra/ \fra} &f_{[\frc]}(-t) E^{\iota \beta, \iota \alpha}(t+1, 0 ; \fra) = \\
&N\frc E^{\iota \beta,\iota \alpha}(1,0; \fra) - 
E^{\iota\beta,\iota\alpha}(1,0; \frc^{-1} \fra)
\end{split}
\end{align}
By putting together (\ref{ek_series_L_value}), (\ref{ks_corollary_in_action}), (\ref{bracket_term}), (\ref{ek_series_trivial_identity}) and (\ref{distribution_relation_in_action}), we get the equality
\[
\begin{split}
\EK^{\iota \beta, \iota\alpha-1}_{A^\sigma(\frf \frb^{-1})}(f_{[\frc]}, x_\frb^\sigma) &= 
(-1)^{\frac{g(g-1)}{2}}\frac{(\alpha-1)! (2\pi i)^{d(\beta)}}{\Omega_\sigma^{\iota \alpha} \Omega_\sigma^{\vee,\iota \beta}} \cdot \ell^{-\iota \beta} \chi_\iota(\ell \lambda \frb)^{-1} \cdot \\
&\left(N\frc L_\frf(\chi_\iota,0,[\ell \lambda \frb]) - \chi_\iota(\frc)^{-1} L_\frf(\chi_\iota, 0, [\ell \lambda\frb \frc])\right) \cdot \\
&\omega(A^\sigma(\frf \frb^{-1}))^{[\iota \alpha]} \otimes \omega(A^\sigma(\frf \frb^{-1})^\vee)^{[\iota \beta]}
\end{split}
\]
Recall the definition of $\omega(A^\sigma(\frf \frb^{-1}))$ and observe that 
\[
\ell^{\iota \alpha -\iota \beta} \chi_\iota(\ell \lambda \frb)^{-1} = \chi_\iota(\frb)^{-1} \chi_\iota(\lambda)^{-1}
\] 
where we freely regard $\chi_\iota$ both as a character of ideals and ideles.
Passing to $\Sym^{\iota \alpha}(\omega_{A^\sigma/F}) \otimes \Sym^{\iota \beta}(\omega_{A^{\sigma,\vee}/F})$ via the isogenies $[\frf]$ and $[\frb]$ and our identificaiton of $\TSym$ with $\Sym$ then gives the claim.
\end{proof}

\section{Deligne's conjecture for Hecke characters}\label{section_8_deligne_conj}
We use the Eisenstein-Kronecker classes from the previous section to define de Rham classes of the reflex motive $M(\Xi)$. More precisely, we attach to any critical algebraic Hecke characters some normalized class $\cEK$. We reinterpret Theorem \ref{LValueProposition} in this set-up and together with Theorem \ref{a_eta_elements_transformation} and Blasius' period relation, we deduce Deligne's conjecture for arbitrary critical algebraic Hecke characters.

Except for the statement of Theorem \ref{deligne_conjecture}, we fix the following set-up throughout this section:
Let $L$ be a totally imaginary number field and let $\Phi$ be a CM-type of $L$.
Let $\chi: L_f^\times \to T^\times$ be an algebraic Hecke character whose infinity-type $\chi_a = \beta - \alpha$ is critical with respect to $\Phi$ and assume $T \supset L^\alpha$. Fix an integral ideal $\frf \neq \cO_L$ such that $\chi$ is of conductor dividing $\frf$.
Note that $\chi \neq N^{-1}$ as $\chi$ is critical. Fix an integral ideal $\frc$ which is coprime to $\frf$ and satisfies $\chi(\frc) \neq N\frc^{-1}$, so that $N\frc - \chi(\frc)^{-1} \in T^\times$. 
Let $E$ denote the reflex field of $(L,\Phi)$. Let $F/E$ be a finite Galois extension containing the Galois closures of $L$ and $T$ and let $A/F$ be an abelian variety of CM-type $(L,\Phi)$ such that the $[\frf]$- and $[\frc]$-torsion points of $A$ are defined over $F$. Moreover, assume that the $\cO_L$-module $H_1(A(\CC),\ZZ)$ is free of rank $1$.

\subsection{The de Rham realization of $h^1(A)^\alpha \otimes h^1(A^\vee)^\beta$ and the normalized class $\cEK$}\label{de_rham_realization_abelian_variety_split_case}
Under our standing assumptions the de Rham realization of $h^1(A)_T^\alpha \otimes_{T} h^1(A^\vee)_T^\beta$ is particularly convenient as the torus $\TT_{L,F}$ is split. One has
\begin{align}\label{temporary_decomposition}
h^1(A)^\alpha_\dR = \bigoplus_{\iota \in J_{L^\alpha}} \Sym^{\iota \alpha}(h^1_\dR(A)).
\end{align}

\begin{proposition}
There is a canonical isomorphism
\[
(h^1(A)^\alpha \otimes_{L^\alpha} T)_\dR = \bigoplus_{\iota \in J_T} \Sym^{\iota \alpha}(h^1_\dR(A)).
\]
\end{proposition}

\begin{proof}
The decomposition (\ref{temporary_decomposition}) is one of $L^\alpha$-modules, where $L^\alpha$ acts on $\Sym^{\iota \alpha}(h^1_\dR(A))$ via the embedding $\iota: L^\alpha \into F$. By the assumption $T \subset F$, one has the canonical decomposition
$F \otimes_{\iota, L^\alpha} T = \prod_{\iota' \in J_{T,\iota}} F$,
and hence
\begin{align*}
\Sym^{\iota \alpha}(h^1_\dR(A)) \otimes_{\iota,L^\alpha} T &= \Sym^{\iota \alpha}(h^1_\dR(A)) 
\otimes_{F} (F \otimes_{\iota, L^\alpha} T) \\
&= \bigoplus_{\iota' \in J_{T,\iota}} \Sym^{\iota \alpha}(h^1_\dR(A)),
\end{align*}
from which the claim follows.
\end{proof}

\begin{proposition}\label{de_rham_split} 
There is a canonical isomorphism
\[
\bigl(h^1(A)^\alpha_T \otimes_T h^1(A^\vee)^\beta_T \bigr)_\dR = \bigoplus_{\iota \in J_T} \Sym^{\iota \alpha}(h^1_\dR(A)) \otimes_{F} \Sym^{\iota \beta}(h^1_\dR(A^\vee)).
\]
Moreover, under this identification, one has
\[
F^{d(\alpha)+d(\beta)} \bigl(h^1(A)^\alpha_T \otimes_T h^1(A^\vee)^\beta_T \bigr)_\dR = \bigoplus_{\iota \in J_{T/E}} \Sym^{\iota \alpha}(\omega_{A/F}) \otimes_{F} \Sym^{\iota \beta}(\omega_{A^\vee/F}).
\]
\end{proposition}

\begin{proof}
(i) For each $\iota \in J_T$, let us denote $V_\iota = \Sym^{\iota \alpha}(h^1_\dR(A))$ and $W_\iota = \Sym^{\iota\beta}(h^1_\dR(A^\vee))$ for brevity. We regard $V_\iota$ and $W_\iota$ as $T \otimes F$-modules via the homomorphism $\iota: T \otimes F \to F$, $t \otimes x \mapsto \iota(t)x$. Then, for $\iota, \iota' \in J_T$, there is a canonical isomorphism
\begin{align*}
V_\iota \otimes_{T \otimes F} W_{\iota'} &= V_\iota \otimes_{F} (F \otimes_{\iota, T \otimes F, \iota'} F) \otimes_{F} W_{\iota'},
\end{align*}
and with the observation that
\[
F \otimes_{\iota, T \otimes F, \iota'} F = \begin{cases} F &\textrm{ if } \iota = \iota' \\ 0 &\textrm{ if } \iota \neq \iota' \end{cases}
\]
the claim follows. Assertion (ii) follows from the constructions of $h^1(A)^\alpha$ and $h^1(A^\vee)^\beta$ and the fact that for the abelian variety $A/F$, the Hodge filtration is given by $h^1_\dR(A) \supset \omega_{A/F} \supset 0$ (and analogously for $A^\vee$). Note that 
$\Sym^{\iota \alpha}(\omega_{A/F}) \neq 0$ if and only if $\iota|_E = 1_E$ since $\omega_{A/F}(\sigma) \neq 0$ if and only if $\sigma \in \Phi$, and $\iota \alpha$ has support $\Phi$ if and only if $\iota|_E = 1_E$. 
\end{proof}

Via Proposition \ref{de_rham_split} and the Eisenstein-Kronecker classes, we can now associate to $\chi$ a class $\cEK$ in the de Rham realization of $h^1(A)^\alpha_T \otimes h^1(A^\vee)^\beta_T$:

\begin{definition}\label{normalized_ek_class}
Via Proposition \ref{de_rham_split}, we define the \emph{normalized Eisenstein-Kronecker class associated to $\chi$} by
\[
\cEK \coloneqq \frac{(-1)^{\frac{g(g-1)}{2}}}{(\alpha-1)!(N\frc - \chi(\frc)^{-1})}   \sum_{[\frb] \in I^\frf/P^\frf} \chi(\frb) \cdot 
(\wtilde{\EK}_A^{\iota \beta, \iota \alpha - 1}(f_{[\frc]},x_\frb) )_{\iota \in J_{T/E}}
\]
in $F^{d(\alpha)+d(\beta)} (h^1(A)^\alpha_T \otimes_T h^1(A^\vee)^\beta_T)_\dR$, where we let the sum range over a set of integral representatives of $I^\frf/P^\frf$ which are coprime to $\frc$.
\end{definition}

\subsection{The class $\cEK$ and $L$-values}
We now reformulate Proposition \ref{LValueProposition} in terms of the motive $h^1(A)_T^\alpha \otimes h^1(A^\vee)^\beta_T$. This reinterprets the $L$-value $L(\chi,0)$ as a period. 
Fix an $\cO_L$-linear isomorphism $\xi: \cO_L \cong H_1(A(\CC),\ZZ)$ and let $\gamma \in h^1_B(A)$ be the corresponding element. 
Let $\sigma \in J_F$. Let $\gamma_\sigma \in h^1_B(A^\sigma)$ be an $L$-basis and let $\gamma_\sigma^\vee \in h^1_B(A^{\sigma,\vee})$ be the dual basis with respect to the canonical pairing $h^1_B(A^\sigma) \otimes_L h^1_B(A^{\sigma,\vee}) \to (2\pi i)^{-1} L$. Recall the definition of the elements from Proposition \ref{t_elements_independence}: For every $\sigma \in J_F$, choose an extension to $G_\QQ$, which we denote by the same letter. Let $\lambda_\sigma \in L_f^\times$ be the unique element such that $I_f(\gamma)^\sigma = \lambda_\sigma I_f(\gamma_\sigma)$. Then
$t_\sigma = \chi(\lambda_\sigma)^{-1} \varepsilon_\Phi(\sigma) \in T^\times$
only depends on $\sigma \in J_F$. 

\begin{theorem}\label{the_main_theorem}
Let $\sigma \in J_F$ and $\gamma_\sigma$ and $\gamma_\sigma^\vee$ as above. Let $\eta = \sigma|_E \in J_E$. One has
\[
I_\infty(\cEK^\sigma) = (2\pi i)^{d(\beta)}  \bigl(L_\frf(\chi_\iota,0)\bigr)_{\iota \in J_{T,\eta}} \cdot t_\sigma  (\gamma_\sigma^{\balpha} \otimes (\gamma_\sigma^\vee)^{\bbeta} )
\]
in $(h^1(A)_T^\alpha \otimes h^1(A^\vee)_T^\beta)^\sigma_B \otimes \CC$.
\end{theorem}

\begin{proof}
Choose $L \otimes F$-generators $\omega(A^\sigma)$ and $\omega(A^{\sigma,\vee})$ of $\omega_{A^\sigma/F}$ and $\omega_{A^{\sigma,\vee}/F}$ and let $\Omega_\sigma$ and $\Omega_\sigma^\vee$ denote the periods as in Definition \ref{def_omega_periods}.
Proposition \ref{MAlphaPeriod} implies
\[
I_\infty( \omega(A^\sigma)^{\balpha} \otimes \omega(A^{\sigma,\vee})^{\bbeta}) = \Omega_\sigma^{\balpha} \Omega_\sigma^{\vee,\bbeta} \cdot (\gamma_\sigma^{\balpha} \otimes \gamma_\sigma^{\bbeta}).
\]
Now, for each $\sigma \in J_F$ choose an element $\ell_\sigma \in L^\times$ such that $\ell_{\sigma,v} \lambda_{\sigma,v} \equiv 1 \mod \frf_v$ for all finite places $v$ of $L$. 
Together with Proposition \ref{SignLemma} and Theorem \ref{LValueProposition}, we obtain
\begin{align*}
I_\infty(\cEK^\sigma) &= \frac{(-1)^{\frac{g(g-1)}{2}}}{(\alpha-1)!(N\frc - \chi(\frc)^{-1})} \sum_{[\frb]} \chi(\frb) \varepsilon_\Phi(\sigma) \cdot I_\infty\Bigl( \wtilde{\EK}^{\iota\beta,\iota\alpha-1}_{A^\sigma}(f_{[\frc]}, x_\frb^\sigma) \Bigr)_{\iota \in J_{T,\eta}} \\
&= (2 \pi i)^{d(\beta)} (N\frc- \chi(\frc)^{-1})^{-1} \varepsilon_\Phi(\sigma) \chi(\lambda_\sigma)^{-1} \cdot \\
&\sum_{[\frb]} \Bigl((N\frc L_\frf(\chi_\iota,0, [\ell_\sigma \lambda_\sigma \frb]) - \chi_\iota(\frc)^{-1} L_\frf(\chi_\iota,0, [\ell_\sigma \lambda_\sigma \frb \frc]))\Bigr)_{\iota \in J_{T,\eta}}  (\gamma_\sigma^{\balpha} \otimes (\gamma_\sigma^\vee)^{\bbeta}) \\
&= (2\pi i)^{d(\beta)} \bigl(L_\frf(\chi_\iota,0)\bigr)_{\iota \in J_{T,\eta}} \cdot t_\sigma (\gamma_\sigma^{\balpha} \otimes (\gamma_\sigma^\vee)^{\bbeta}),
\end{align*}
where the first and last equality follow by reindexing.
\end{proof} 

We now come to the main result of this paper. For its statement, we momentarily drop the standing assumptions of this section.

\begin{theorem}\label{deligne_conjecture}
Let $L$ and $T$ be number fields and $\chi: \AA_L^\times \to T^\times$ a critical algebraic Hecke character of conductor dividing $\frf$. Then
$L_\frf(\chi,0)$ and $c^+(\chi)$ coincide up to a factor in $T$.
\end{theorem}

\begin{proof}
Since $\chi$ is critical, $L$ must be a totally real field or a totally complex field containing a CM-field.
In the case where $L$ is totally real, the Theorem follows by classical results due to Euler, Siegel and Klingen. In the totally complex case there exists a CM-type $\Phi$ of $L$ such that $\chi_a = \beta - \alpha$ as in Section \ref{setup_hecke_character}. We may assume that $L^\alpha \subset T$ since Deligne's conjecture is invariant under extension of coefficients, cf.\ \cite{deligne}, Remarque 2.10. Choose a finite Galois extension $F/E$ such that $F$ contains the Galois closure of $L$ and $T$ and an abelian variety $A/F$ of CM-type $(L,\Phi)$  such that the $\frf$- and $\frc$-torsion points of $A$ are defined over $F$ and such that $H_1(A(\CC),\ZZ)$ is free of rank $1$ as an $\cO_L$-module. We may enlarge $\frf$ such that $\frf \neq \cO_L$ and choose an integral ideal $\frc$ coprime to $\frf$ such that $\chi(\frc) \neq N\frc^{-1}$. We can associate to $\chi$ the normalized Eisenstein-Kronecker class $\cEK$ as in Definition \ref{normalized_ek_class} and regard it as an element in $F^{d(\alpha)+d(\beta)} R_{F/\QQ}(h^1(A)^\alpha_T \otimes h^1(A^\vee)^\beta_T)_\dR$. Via (\ref{alt_constr_m_xi}) we obtain an element
\[
e_\Xi \cdot \cEK \in F^{d(\alpha)+d(\beta)} RM(\Xi)_\dR.
\]
Note that $I_\infty(\cEK) = \sum_{\sigma \in J_F} I_\infty(\cEK^\sigma)$, where on the left, we regard $\cEK$ in $R_{F/\QQ}(h^1(A)^\alpha_T \otimes h^1(A^\vee)^\beta_T)_\dR$ and on the right, every $\cEK^\sigma$ is regarded as an element in $(h^1(A^\sigma)^\alpha_T \otimes h^1(A^{\sigma,\vee})^\beta_T)_\dR$. Theorem \ref{the_main_theorem} therefore implies
\begin{align*}
I_\infty(\cEK) = (2\pi i)^{d(\beta)} \cdot L(\chi,0) \cdot \sum_{\eta \in J_E} e(\eta) \sum_{\sigma \in J_{F,\eta}} t_\sigma (\gamma_\sigma^{\balpha} \otimes (\gamma_\sigma^\vee)^{\bbeta})
\end{align*}
and applying the idempotent $e_\Xi$ gives
\[
I_\infty(e_\Xi \cdot \cEK) = (2\pi i)^{d(\beta)} \cdot L(\chi,0) \cdot \sum_{\eta \in J_E} e(\eta) a_\eta,
\]
with the elements $a_\eta$ from Definition \ref{def_a_eta_elements}. Now, Deligne's conjecture follows from Theorem \ref{a_eta_elements_transformation} and Blasius' Theorem \ref{blasius_period_relation}. Note that we do not have to check that $e_{\Xi} \cdot \cEK \neq 0$. Indeed, if $L(\chi,0) = 0$, there is nothing to show and if $L(\chi,0) \neq 0$, then $e_\Xi \cdot \cEK$ also does not vanish since no non-zero element in $T \otimes \CC$ annihilates $\sum_{\eta \in J_E} e(\eta)a_\eta$ by Theorem \ref{a_eta_elements_transformation}.
\end{proof}

\bibliographystyle{amsalpha}
\bibliography{deligne}

\end{document}